\documentclass[11pt, reqno]{amsart}
\usepackage{amsmath,amssymb,latexsym,esint,cite,mathrsfs,mathtools}
\usepackage{verbatim,wasysym}
\usepackage[left=2.6cm,right=2.6cm,top=2.6cm,bottom=2.5cm]{geometry}

\usepackage[english]{babel} 

\usepackage{setspace} 

\usepackage{microtype}
\usepackage{color,enumitem,graphicx}

\usepackage[colorlinks=true,urlcolor=blue, citecolor=red,linkcolor=blue,
urlcolor=green!45!black, 
linktocpage,pdfpagelabels, bookmarksnumbered,bookmarksopen]
{hyperref}

\usepackage[hyperpageref]{backref}

\usepackage[hang]{footmisc}
\usepackage{xpatch}
\usepackage{footnotebackref} 

\makeatletter
\xpretocmd{\@adminfootnotes}{\let\@makefntext\BHFN@OldMakefntext}{}{}
\renewcommand\@makefntext[1]{%
 \@ifundefined{@makefnmark}
 {}
 {%
 \renewcommand\@makefnmark{%
 \mbox{%
 \textsuperscript{%
 \normalfont
 \hyperref[\BackrefFootnoteTag]{\@thefnmark}%
 }%
 }\,%
 }%
 \BHFN@OldMakefntext{#1}%
 }%
}
\makeatother

\usepackage{etoolbox} 

\patchcmd{\section}{\normalfont}{\normalfont \Large \bfseries}{}{}
 
\patchcmd{\subsection}{\normalfont}{\normalfont \large}{}{}
\patchcmd{\subsection}{-.5em}{.5\linespacing}{}{}

\patchcmd{\subsubsection}{-.5em}{.5\linespacing}{}{}

\DeclareRobustCommand{\SkipTocEntry}[5]{}

\let\oldtocsection=\tocsection
\let\oldtocsubsection=\tocsubsection
\let\oldtocsubsubsection=\tocsubsubsection

\renewcommand{\tocsection}[2]{\hspace{0em}\oldtocsection{#1}{#2}}
\renewcommand{\tocsubsection}[2]{\hspace{1em}\oldtocsubsection{#1}{#2}}
\renewcommand{\tocsubsubsection}[2]{\hspace{2em}\oldtocsubsubsection{#1}{#2}}

\newtheorem{lemma}{Lemma}[section]
\newtheorem{corollary}[lemma]{Corollary}
\newtheorem{proposition}[lemma]{Proposition}
\newtheorem{theorem}[lemma]{Theorem}

\theoremstyle{definition}

\newtheorem{remark}[lemma]{Remark}
\newtheorem{example}[lemma]{Example}
\newtheorem{definition}[lemma]{Definition}

\numberwithin{equation}{section}

\usepackage{orcidlink} 
\usepackage[square, comma, numbers, sort&compress]{natbib} 

\newcommand{\R}{\mathbb{R}}
\newcommand{\N}{\mathbb{N}}
\newcommand{\mc}[1]{\mathcal{#1}}
\newcommand{\dive}{{\rm div}}

\newcommand{\diam}{{\rm diam}}

\newcommand{\osc}{{\rm osc}}

\newcommand{\nequiv}{\not \equiv}
\def\eps{\varepsilon}

\def\pabs#1{\left|{#1}\right|}
\def\norm#1{\|#1\|}

\newcommand{\wto}{\rightharpoonup}

\renewenvironment{proof}[1][\proofname]{\medskip \noindent {\bfseries #1. }}{\qed \bigskip}

\title[Concavity for $p$-Laplace equations]
{Concavity and perturbed concavity \\ for $p$-Laplace equations}

\author[M.\ Gallo]{Marco Gallo \orcidlink{0000-0002-3141-9598}}
\author[M.\ Squassina]{Marco Squassina \orcidlink{0000-0003-0858-4648}}

\address[M.\ Gallo, M.\ Squassina]{\newline\indent Dipartimento di Matematica e Fisica
	\newline\indent
	Università Cattolica del Sacro Cuore
	\newline\indent
	Italy, Brescia, BS, Via della Garzetta 48, 25133}
\email{\href{mailto:marco.gallo1@unicatt.it}{marco.gallo1@unicatt.it}}
\email{\href{mailto:marco.squassina@unicatt.it}{marco.squassina@unicatt.it}}

\thanks{\emph{Fundings.} The first author is supported by INdAM-GNAMPA Projects “Mancanza di regolarità e spazi non lisci: studio di autofunzioni e autovalori”, codice CUP \#E53C23001670001\#, and “Metodi variazionali per problemi dipendenti da operatori frazionari isotropi e anisotropi", codice CUP
\#E5324001950001\#. 
The second author is member of INdAM-GNAMPA%
\bigskip
}


\begin{document}

\maketitle

\begin{abstract}
\vspace{-1.5em}
In this paper we study convexity properties for quasilinear Lane-Emden-Fowler equations of the type
$$\begin{cases}
-\Delta_p u = a(x) u^q & \quad \hbox{ in $\Omega$},\\ 
u >0 & \quad \hbox{ in $\Omega$}, \\
u =0 & \quad \hbox{ on $\partial \Omega$},
\end{cases}$$
when $\Omega \subset \mathbb{R}^N$ is a convex domain. 
In particular, in the subhomogeneous case $q \in [0,p-1]$, the solution $u$ inherits concavity properties from $a$ whenever assumed, while it is proved to be concave up to an error if $a$ is near to a constant. 
More general problems are also taken into account, including a wider class of nonlinearities. 
These results generalize some contained in \cite{Ken85} and \cite{Sak87}. 

Additionally, some results for the singular case $q \in [-1,0)$ and the superhomogeneous case $q>p-1$, $q \approx p-1$ are obtained. 
Some properties for the $p$-fractional Laplacian $(-\Delta)^s_p$, $s\in (0,1)$, $s \approx 1$, are shown as well. 

We highlight that some results are new even in the semilinear framework $p=2$; in some of these cases, we deduce also uniqueness (and nondegeneracy) of the critical point of $u$. 
\end{abstract}

\bigskip

\smallskip

\begin{center}
\begin{minipage}{10cm}
		\small
		\tableofcontents
\end{minipage}
\end{center}

\newpage

\section{Introduction}

\subsection{Some background}
\label{sec_background}

Qualitative properties of solutions of PDEs are a classical topic, and often the features of the domain and of the nonlinearity are inherited by the solutions. 
Consider the equation
\begin{equation}\label{eq_general_intro}
 \begin{cases}
-\Delta u = a(x) g(u) & \quad \hbox{ in $\Omega$},\\ 
u >0 & \quad \hbox{ in $\Omega$}, \\
u =0 & \quad \hbox{ on $\partial \Omega$},
\end{cases}
\end{equation}
with $\Omega \subset \R^N$: when $\Omega$ is the ball and $a(x)$ is radially symmetric and decreasing, a seminal result by \cite[Theorem 1']{GNN79} (see also \cite{DaSc04, MMPS14} for the $p$-Laplacian) states that the solution $u$ itself is radially symmetric.
Several generalizations of \cite{GNN79} have been taken into account considering, for example, sets which are symmetric only in one direction.
In this paper, we are interested in the case of convex domains $\Omega$, with no a priori assumption of symmetry. We will mainly focus on the power case $g(u)=u^q$ -- namely, Lane-Emden-Fowler equations -- but more general cases will be taken into account.

Generally, for a Dirichlet problem set in a convex domain one may expect that the solutions of \eqref{eq_general_intro} are concave:
when $a(x)$ is constant and $\Omega$ is the ball, this is the case for the torsion problem $-\Delta u=1$ (i.e. $q=0$) where the solution is explicit; the result anyway holds also in more general domains, like deformations of ellipses \cite{HNST18, Ste22} (see also \cite{ChWe23} for other nontrivial examples)%
\footnote{If $-\Delta u =1$ in $\Omega$, and $u$ is shown to be concave in $\Omega_k:=\{u \geq k\}$, then $v_k:= u-k$ solves $-\Delta v_k =1$ in $\Omega_k$ with $v_k=0$ on $\partial \Omega_k$, and it is concave. 
Actually, a \emph{propagation from the boundary} argument \cite{KeMn93, Ste22, ChWe23}, typical of equations with $g(0)>0$, ensures that concavity on $\partial \Omega_k$ is enough to achieve concavity in the whole $\Omega_k$.
}.
 Actually, the result keeps holding also for singular equations $q\leq -1$ (in arbitrary domains, see Theorem \ref{thm_intr_recap} below).
On the other hand, when $q>0$, the solution is seen to be never concave, no matter what the domain is \cite[Remark 4.1]{AAGS23}; this is essentially due to the fact that $g(0)=0$.

For general convex sets the situation is worse even for $q=0$: if $\Omega$ has some flat part, e.g. in triangles, $u$ is never concave \cite[Theorem 18 in Section 7]{Ken84} (see also \cite[Section 11]{KeMn93}, \cite{Ste22}). 
 What one can obtain is that some \emph{transformation} of $u$ is concave: 
in the case of the Laplacian, $\Omega$ a general convex domain and $a(x)$ constant, the story so far can be summarized in Theorem \ref{thm_intr_recap} below. 
To state it, we recall that a function $v$ is concave if $\lambda v(x) + (1-\lambda) v(y) \leq v\big( \lambda x + (1-\lambda) y\big)$
for $x,y \in \Omega$, $\lambda \in [0,1]$, while it is \emph{strictly} concave if the inequality is strict for $\lambda \in (0,1)$ and $x \neq y$; moreover, $v$ is \emph{strongly} concave if $x \mapsto v(x)- \frac{m}{2}|x|_2^2$ is concave for some $m>0$.
In addition, we recall that a function $u>0$ is \emph{$\alpha$-concave}, $\alpha \in \R \cup \{\pm \infty\}$, if:
\begin{itemize}
\item $u$ is constant, if $\alpha=+\infty$;
\item $\frac{1}{\alpha} u^{\alpha}$ is concave, if $\alpha\in (-\infty, 0) \cup (0,+\infty) $;
\item $\log(u)$ is concave, if $\alpha=0$;
\item $u$ is \emph{quasiconcave} (i.e. $u^{-1}([k,+\infty))$ are convex for any $k \in \R$, see also \eqref{eq_def_quasiconc}), if $\alpha=-\infty$.
\end{itemize}
In particular, if $\alpha=-1$, $u$ is said \emph{harmonic concave}. 
When $\alpha \in \R$ we define similarly strict and strong $\alpha$-concavity. 
Moreover, we recall that $u$ $\alpha$-concave implies $u$ $\beta$-concave for every $\beta \leq \alpha$. Notice that the definition for $\alpha=0$ is coherent with $\frac{1}{\alpha} (u^\alpha - 1) \to \log(u)$ as $\alpha \to 0$.

Here and after, by writing $\partial \Omega \in C^{k,\alpha}$ for some $k \in \N$, we will implicitly assume $\alpha \in (0,1]$; moreover, $2^*:=\frac{2N}{N-2}$ for $N\geq 3$ and $2^*:=+\infty$ for $N=2$, stands for the Sobolev critical exponent.
See Section \ref{sec_superhomog} for the definition of ground state.

\begin{theorem}
\label{thm_intr_recap}
Let $\Omega \subset \R^N$, $N \geq 2$, be open, bounded and convex. Let $u$ be a positive solution of
 $$
 \begin{cases}
-\Delta u = \lambda u^q & \quad \hbox{ in $\Omega$},\\ 
u =0 & \quad \hbox{ on $\partial \Omega$},
\end{cases}
$$
for some $q \in \R$ and $\lambda>0$. 
We have the following assertions.
\begin{itemize}
\item If $q \in (-\infty, 0)$ and $\partial \Omega \in C^{2,\alpha}$, then $u$ (unique) is locally strongly $\frac{1-q}{2}$-concave. 
\item If $q \in [0, 1]$, then $u$ (unique) is locally strongly $\frac{1-q}{2}$-concave. 
If in addition $\Omega$ is strongly convex with $\partial \Omega \in C^{2,\alpha}$, then $u$ is strongly $\frac{1-q}{2}$-concave.
\item If $q \in (1, 2^*-1) $, then 
\begin{itemize}
\item if $N=2$, then the (unique) ground state solution is locally strongly $\frac{1-q}{2}$-concave; 
\item if $\Omega$ is strongly convex with $\partial \Omega \in C^{2,\alpha}$, then there exists a solution which is strongly $\frac{1-q}{2}$-concave. 
\end{itemize}
\end{itemize}
\end{theorem}
The nonstrict concavity part of Theorem \ref{thm_intr_recap} is a summary of several results: $q=0$ \cite{MaLi71}, $q =1$ \cite{BrLi76}, $q \in (0,1)$ \cite{Ken85}, $q>1$ \cite{Lin94}, and $q<0$ \cite{BGP99}.
The results are optimal in the following sense:
for $q \in [-1,0)$ \cite[pages 329-330]{BGP99} (see also Remark \ref{rem_alpha_mag_1} below), if $\Omega$ is the ball, then $u$ is not $\alpha$ concave for any $\alpha>\frac{1-q}{2}$. 
For $q = 0$ \cite[Remark 4.2.3 and Theorem 6.2]{Ken85}, \cite[page 328]{BGP99} there exists $\Omega$, subset of a cone, such that $u$ is not $\alpha$-concave for any $\alpha>\frac{1}{2}$.
For $q=1$ \cite[Theorem 15 in Section 7]{Ken84}, \cite[pages 328-329]{BGP99}, for each $\alpha>0$ there exists a suitably narrow $\Omega_{\alpha}$ such that $u$ is not $\alpha$-concave: 
in particular it suggests that, for every $\Omega$ sufficiently good (e.g. regular and with some control on the curvature) $u$ could be $\alpha$-concave for some $\alpha=\alpha(\Omega)>0$ sufficiently small -- e.g., in the square, $\alpha=\frac{1}{2}$ -- and this seems an open question; see also \cite{IST20} where it is conjectured a stronger $2$-$\log$-concavity for the eigenfunction. 
 Let us moreover recall that, even if $g \geq 1$ and it is smooth -- namely, a local perturbation of the torsion problem $g=1$ -- and $\Omega$ is smooth and symmetric, the solutions might not even be quasiconcave \cite{HNS16, AAGS23}; additionally, even if $\Omega$ is starshaped and close to a convex set, then the level sets of the solution of the torsion problem need not to be even connected \cite{GlGr22}. Finally, we also mention that assuming $\Omega$ convex in a single direction is not sufficient to get (quasi)concavity in that direction \cite{Wet11}.

Anyway, if $\Omega$ is chosen good enough, one can recover better concavity properties: when $\Omega$ is a ball and $q=0$, the torsional function $u$ is such that $\sqrt{\norm{u}_{\infty}-u}$ is concave (actually, this property -- stronger than concavity -- characterizes ellipsoids \cite{HNST18}), while if $q=1$ the eigenfunction is $\alpha$-concave for some $\alpha \in (\frac{1}{N},1)$ (for instance, $\alpha >\frac{\sqrt{3}+2}{4}\approx 0.93$ when $N=2$, see \cite{Lind94} for some explicit estimate of $\alpha$). 
\begin{remark}
\label{rem_alpha_mag_1}
It is easy to see that, whenever $a\equiv 1$, $g$ is not too singular and a Hopf boundary lemma holds, solutions $u\in C^1(\overline{\Omega}) \cap C^2(\Omega)$ of \eqref{eq_general_intro} are never $\alpha$-concave with $\alpha>1$: indeed, a straightforward computation shows that
$$-\Delta u^{\alpha} = - \alpha(\alpha-1) u^{\alpha-2} |\nabla u|^2 + \alpha u^{\alpha-1} g(u);$$
by staying close to $\partial \Omega$ we have $|\nabla u| \geq C >0$, while we may assume $u(x)=\eps$ with $\eps$ small. 
Thus
$$-\Delta u^{\alpha} \leq - \alpha(\alpha-1) \eps^{\alpha-2} C^2 + \alpha \eps^{\alpha-1} g(\eps);$$
If $\eps g(\eps) = o(1)$ as $\eps \to 0$, we have thus $-\Delta u^{\alpha}<0$ near the boundary; that is, $u^{\alpha}$ is not concave.
This remark in particular applies to $g(t)=t^q$ with $q \geq -1$, coherent with the above statements.
\end{remark}

Local strong concavity -- which implies strict concavity -- of Theorem \ref{thm_intr_recap} has been investigated by several authors \cite{APP81, SWYY85}, \cite[Theorem 4.4 and Corollary 4.6]{CaFr85}, \cite[Corollary 3.6]{BGP99}: the proofs are mainly based on continuation arguments and the famous constant rank theorem \cite[Theorem 1.1]{CaFr85}, which has been subsequently generalized by \cite{KoLe87} (see also \cite{BiGu09}). 
The argument runs as follows: considered for any $q\neq 1$ (resp. $q=1$) $w := -\textup{sign}(1-q) u^{\frac{1-q}{2}}$ (resp. $w:=-\log(u)$), we have that $w$ solves
\begin{equation}
\label{eq_intro_transformed}
\Delta w =-\frac{1-q}{w} \left(\frac{1+q}{(1-q)^2}|\nabla w|^2 + \frac{1}{2}\right) \qquad \hbox{(resp. $\Delta w = |\nabla w|^2 + \lambda $).}
\end{equation}
In each case we have that the right hand side is positive (on the image of $w$, see \cite[Lemma 3.1]{BGP99} for $q<-1$) with inverse which is convex in $w$, thus by \cite{KoLe87} the Hessian of $w$ has constant rank: since there exist points with full rank -- near the extremal point \cite[Lemma at page 207]{Bas76}, or near the boundary (see Proposition \ref{prop_info_bound_c2}) whenever this is strongly convex -- we obtain the global strict concavity (actually, the Hessian matrix has full rank everywhere in $\Omega$); see also \cite[page 80]{BMS23} for an alternative argument based on developable graphs. 
In \cite[Lemmas 2.5, 3.6 and 4.8]{LeVa08} evolutive arguments are employed in presence of smooth strictly convex domains, and a strong concavity is obtained for $q\geq 0$; we mention also \cite{APP81, MSY12} where estimates on the curvature of the level sets are given. 
We highlight that a remarkable consequence of the strict convexity is the uniqueness and nondegeneracy of the critical point -- namely, a maximum -- of the solution. 

\smallskip

In the nonautonomous case, \cite[Theorem 4.1]{Ken85} provided the following result: if $u$ is a positive solution of the Dirichlet problem
$$ \begin{cases}
-\Delta u = a(x) & \quad \hbox{ in $\Omega$},\\ 
u =0 & \quad \hbox{ on $\partial \Omega$},
\end{cases}$$
and $a(x)$ is $\theta\geq 1$ concave, then $u$ is $\frac{\theta}{1+2\theta}$-concave; we see that for $\theta \to +\infty$ we recover the torsion problem.
We recall also the results by \cite[Theorem 6.1]{BrLi76} and \cite[Appendix]{SWYY85} (see also \cite{KST24}), where eigenfunctions of $-\Delta u = (\lambda-V(x))u$ are shown to be $\log$-concave if the potential $V$ is convex and nonnegative -- actually the concavity is strong if $D^2 V>0$; see also \cite{Gom07} for some discussions in the superlinear case.

The theorem in \cite{Ken85} is sharp in the following sense \cite[Theorem 6.2]{Ken85}: there exists $\Omega$, subset of an open cone, such that for any $\theta \in [1,+\infty]$ there exists a $\theta$-concave function $a_{\theta}$ and a corresponding solution $u$ which is not $\alpha$-concave for any $\alpha>\frac{\theta}{1+2\theta}$.
Moreover, in \cite[Theorem 16 in Section 7]{Ken84}, it is shown that the condition $\theta \geq 1$ cannot be relaxed to $\theta>0$; it remains open anyway to show if the threshold $\theta=1$ is sharp (see also Corollary \ref{cor_strict_conc_const}). 

\smallskip

When $a(x)$ has no convexity property, we cannot expect concavity for $u$. 
A \emph{quantitative} version of the convexity principle has been developed in \cite{BuSq20} (see also \cite[Proposition 3.2]{ABCS23}): in a particular case it states that, if $u$ is a positive solution of
$$
 \begin{cases}
-\Delta u = a(x) u^q & \quad \hbox{ in $\Omega$},\\ 
u =0 & \quad \hbox{ on $\partial \Omega$},
\end{cases}
$$
$q \in [0,1)$, then there exists a convex function $\bar{v}$ such that
$$\norm{u^{\frac{1-q}{2}} - \bar{v}}_{\infty} \leq C_q \norm{\nabla a}_{\infty},$$
thus $u^{\frac{1-q}{2}}$ is \emph{close} to a convex function if $a$ is close to a constant, with a control on the error. Here $C_q \to +\infty$ as $q \to 1$: this is essentially related to the fact that the transformed equation of the eigenfunction problem \eqref{eq_intro_transformed} has a nonlinearity which is not strictly monotonic in $t$ (see Section \ref{sec_eigen_phi_unbound} for more comments); a different, parabolic, approach to deal with the eigenfunction problem can be found in \cite{GMrS25}.
The results in \cite{BuSq20, ABCS23} apply also to more general operators where classical regularity of solutions holds.

Let us mention that a similar result has been achieved in regards of radial symmetry: indeed in \cite{CCPP24} the authors show (roughly speaking) that the solutions of \eqref{eq_general_intro} in a ball satisfy
$$|u(x)-u(y)| \leq C \textup{def}(a)^{\alpha} \quad \hbox{for each $x,y \in B_1$, $|x|=|y|$},$$
where $\alpha \in (0,1]$ and $\textup{def}(a)$ is a quantity which measures how far $a$ is from being radially symmetric and decreasing.
We believe that this kind of quantitative results may have also some interest from an engineering point of view.

\medskip

The techniques on which the previous results are based mainly involve regularity of solutions and maximum principles on the \emph{concavity function} related to a function $v: \Omega \to \R$ (namely, a transformation of the solution):
$$\mc{C}_{v}(x,y,\lambda):= \lambda v(x) + (1-\lambda) v(y) - v\big( \lambda x + (1-\lambda) y\big)$$
for $x,y \in \Omega$, $\lambda \in [0,1]$. 
It is clear that $v$ is concave if and only if $\mc{C}_v \leq 0$. 
When the solutions are sufficiently regular, the abovementioned results have been generalized also to fully nonlinear frameworks \cite{BiSa13}.

Anyway, due to the regularity restrictions, these techniques cannot be directly applied to $p$-Laplace equations: a classical idea, thus, is to regularize the operator, apply the result and pass to the limit. 
This procedure requires at least two delicate steps: the first is the uniqueness of the solution, which is needed to discuss concavity properties of a \emph{fixed} solution. 
The second ingredient is the form of the regularization: as a matter of fact, we need a regularization process which preserves the concavity structure of the original equation. 
This is what has been done by \cite{Sak87} for $p \in (1,+\infty)$ and $q=0,1$, then generalized to $q \in [0,1]$ by \cite{BMS22} (and more general cases, see \eqref{eq_intr_general_p_g} below). 
Namely, they obtain that solutions of $-\Delta_p u = u^q$ are $\frac{p-1-q}{p}$-concave; this power turns to be relevant also for regularity information in singular equations, see Theorem \ref{thm_exist_singular}.
See also \cite[Section 3]{Lind94} and \cite[Section 3]{Kaw90} for explicit computations in the ball. 

Regarding strict concavity, very little is known in the case of the $p$-Laplacian: indeed, it seems that a direct application of the constant rank theorem is not generally the case when $p\neq 2$; nevertheless, in \cite{BMS23} the authors show that the concavity of the solutions is strict when $\Omega \subset \R^2$ (actually strong far from the boundary and from the critical point). 
The proof is delicate: to reach the goal, the authors show first that the solution has a single critical point, and exploit this information to apply a constant rank theorem out of the critical point.
Summing up, what is known \cite{Sak87, BMS22, BMS23} in the power case is the following result.
\begin{theorem}
\label{thm_intr_recap_plaplac}
Let $\Omega \subset\R^N$ be open, bounded and convex. Let $u$ be a weak positive solution of
 $$ \begin{cases}
-\Delta_p u = \lambda u^q & \quad \hbox{ in $\Omega$},\\ 
u =0 & \quad \hbox{ on $\partial \Omega$},
\end{cases}$$
for some $q \in [0,p-1]$ and $\lambda>0$. 
Then $u$ is $\frac{p-1-q}{p}$-concave. 
If $N=2$ and $\partial \Omega \in C^2$, then $u$ is strictly $\frac{p-1-q}{p}$-concave in $\Omega$, and locally strongly $\frac{p-1-q}{p}$-concave in $\Omega\setminus\{\bar{x}\}$, where $\bar{x}$ is the unique critical point of the solution.
\end{theorem} 
A different approach, based on concave envelopes of viscosity solutions, can be found in \cite{CrFr20} in the case of the eigenfunction (see also \cite{ALL97}).

Before presenting our results, let us recall briefly what happens when $g(u)$ is assumed general in \eqref{eq_general_intro} and $a(x)$ constant. 
A first result was given by \cite[Theorem 3.3]{Ken85}, who showed $\alpha$-concavity of solutions under some assumptions on $g(t)$ (see \eqref{eq_cond_extra_kenn}). 
A more natural transformation has been studied by \cite{BMS22} (see also \cite{CaFr85, Kaw85W, Ken88}): namely, under some suitable assumptions on $g$ (see Remark \ref{rem_BMS_assumptions}), the authors showed that, given a positive solution $u$ of
\begin{equation}\label{eq_intr_general_p_g}
 \begin{cases}
-\Delta_p u = g(u) & \quad \hbox{ in $\Omega$},\\ 
u =0 & \quad \hbox{ on $\partial \Omega$},
\end{cases}
\end{equation}
and set
$$\varphi(t):= \int_1^t (G(\tau))^{-\frac{1}{p}} d \tau,$$
$G(t) :=\int_0^t g(\tau) d\tau$, it results that $\varphi(u)$ is concave -- actually strictly concave when $N=2$, see \cite{BMS23}. 
This transformation seems to be relevant also in other frameworks, for instance the boundary behaviour of solutions of singular equations \cite[Section 2.6]{Gla02}.
Moreover, this approach is quite effective in the study of concavity of quasilinear equations of the type $-\dive(\alpha(u) \nabla u) + \frac{\alpha(u)}{2} |\nabla u|^2 = g(u)$: indeed, when $\alpha \nequiv const$, here power concavity seems to be not the right choice (even if $g$ is a power), while a transformation of the type $\varphi(t):= \int_1^t \big(\frac{\alpha(\tau)}{G(\tau)}\big)^{1/p} d \tau$, shaped on $\alpha$ (and $g$) turns to be successful. 
Finally, we highlight that we will use this abstract $\varphi$ to obtain a result on power transformations for the power singular equation (see Section \ref{sec_singular}).

We refer to \cite{AAGS23,GMrS25} for more references on the topic of concavity.

\subsection{Main results}

Aim of the paper is to generalize some of the previous results to the case $p \in (1,+\infty)$, $q \in [0,p-1]$ and $a$ (possibly) nonconstant, with or without concavity assumptions on it. 
We will further propose a general scheme which can be applied to more general nonlinearities $g(u)$ (and even $f(x,u)$), which extends also the semilinear setting proposed in \cite{BMS22}. 
Additional results will be considered as well (including the cases $q \in [-1, 0)$ and $q>p-1$ near $p-1$, as well as fractional equations), briefly commented below but fully presented in Section \ref{sec_further_results}. 
For the sake of clarity we focus in the introduction only to some of the results of the paper.
 
Consider the \emph{Lane-Emden-Fowler equation}
\begin{equation}\label{eq_general_problem*}
\begin{cases}
-\Delta_p u = a(x) u^q & \quad \hbox{ in $\Omega$},\\ 
u >0 & \quad \hbox{ in $\Omega$}, \\
u =0 & \quad \hbox{ on $\partial \Omega$};
\end{cases}
\end{equation}
we highlight that, by a suitable change of variable, this equation can be related also to the \emph{porous media equation} or the \emph{fast diffusion equation} (see e.g. \cite{LeVa08}).
In the whole paper by ``solution'' we will mean \emph{weak solution}. 
We start by a result on the exact concavity of (a power of) $u$, when $a$ is assumed concave as well.

\begin{theorem}[Exact concavity]
\label{thm_main_conc_exact}
Let $\Omega \subset \R^N$, $N\geq 2$, be open, bounded, convex, with $\partial \Omega \in C^{1,\alpha}$, and let $p\in (1,+\infty)$.
Let $q \in [0,p-1]$, and assume $a \in C^{1,\alpha}_{loc}(\Omega)$, $\alpha \in (0,1)$, and $a > 0$ on $\Omega$, $a$ $\theta$-concave with $\theta \geq 1$. 
Then the solution of \eqref{eq_general_problem*} is $\frac{\theta(p-1-q)}{1+ \theta p}$-concave. 
If $a$ is constant (i.e. $\infty$-concave), then the solution is $\frac{p-1-q}{p}$-concave. 
\end{theorem}

Notice that, $\frac{\theta(p-1-q)}{1+ \theta p} \to \frac{p-1-q}{p}$ as $\theta \to +\infty$; moreover for $p=2$ the result is coherent with \cite[Theorem 4.1]{Ken85}, while for $p\neq 2$ and $a \equiv const$ it is coherent with \cite{Sak87, BMS22}. 

\begin{remark}
We highlight that the $C^{1,\alpha}$ regularity of the boundary is exploited in Theorem \ref{thm_main_conc_exact} (and its corollaries) only to gain the uniqueness of the solution. According to \cite{BeKa02, KaLi06}, if one focuses only to ground state solutions, or if $q=p-1$, this assumption can be relaxed.
\end{remark}

\begin{remark}
We observe that, due to the fact that we discuss the qualitative behaviour of a fixed solution $u$, the result in Theorem \ref{thm_main_conc_exact} can be adapted also to other equations, e.g. Kirchhoff equations
$$-\left( \rho+ \int_{\Omega} |\nabla u|^p\right) \Delta_p u = a(x) u^q \quad \hbox{in $\Omega$},$$
$\rho \geq 0$, where the weight of the operator is actually a fixed positive number.
\end{remark}

An application of the previous result is given by the (possibly nonradial) Hardy-Hénon type equations; see also \cite[Remark 2.3]{LeVa08} for further examples. 
Notice that the classical $2$-norm $|x|$ is not even quasiconcave (in any subset of $\R^N$): see \cite{AmGl14} for symmetry breaking results in the ball, while \cite{AvBr20} for symmetry results for singular powers of Hardy-Leray type. 

\begin{example}[Hardy-Hénon type equation]
\label{exam_hardy_henon}
Let $\Omega \subset \R^N$, $N\geq 2$, be open, bounded, convex, with $\partial \Omega \in C^{1,\alpha}$, and let $p\in (1,+\infty)$ and $q \in [0,p-1]$. 
Assume that $u$ is a solution of one of the following problems, with the conditions
$$u>0 \quad \hbox{in $\Omega$}, \qquad u=0 \quad \hbox{on $\partial \Omega$.}$$
\begin{itemize}
\item Assume $\omega \in [0,1]$. Consider $u$ solution of
$$-\Delta_p u = d(x, \partial \Omega)^{\omega} u^q \quad \hbox{ in $\Omega$}.$$
\item Assume $\Omega \subset B_R(0)$ for some $R>0$, and let $\sigma \geq 1$, $\omega \in [0,1]$. Consider $u$ solution of
$$-\Delta_p u = (R^{\sigma}-|x|^{\sigma})^{\omega} u^q \quad \hbox{ in $\Omega$}.$$
\item Assume $\Omega \subset B_R(0)$ for some $R>0$, and let $k>0$, $\sigma>1$, $\omega \in (0,1]$. Assume $R < \left(\frac{k(\sigma-1)\omega}{2 \sigma +\omega}\right)^{\frac{2}{s}}$. 
Consider $u$ solution of
$$-\Delta_p u = \frac{1}{(k + |x|^{\sigma})^2} u^q \quad \hbox{ in $\Omega$}.$$
Notice that for $N=p=\sigma=2$, $k=1$ and $q=0$ the equation is related, via stereographic projection, to the torsion problem in small subdomains of $\mathbb{S}^2$ (see e.g. \cite[Theorem 1.5 and its proof]{KST24}).
\item
Assume $\Omega$ is a subset of the half-plane $\{ x\in \R^N \mid \sum_{i=1}^k x_i > 0\}$, $k \in \{1,\dots,N\}$, and let $\omega \in [0,1]$.
Consider $u$ solution of 
$$-\Delta_p u = \left(\sum_{i=1}^k x_i\right)^{\omega} u^q \quad \hbox{ in $\Omega$}.$$
In particular, if $\Omega$ is a subset of the hyperoctant $\{ x\in \R^N \mid x_i > 0 \; \hbox{for each $i$}\}$, and $u$ is a solution of
$$-\Delta_p u = |x|_1^{\omega} u^q \quad \hbox{ in $\Omega$}$$
where $|x|_1$ is the $1$-norm in $\R^N$.
\item Assume $\Omega \subset \{(x,y) \in \R^2 \mid x,y > 0\}$ and $\omega_1, \omega_2 \geq 0$ with $\omega:= \omega_1+\omega_2 \in [0,1]$. 
Consider $u$ solution of
$$-\Delta_p u = x^{\omega_1} y^{\omega_2} u^q \quad \hbox{ in $\Omega$}.$$
\end{itemize}
Then $u$ is $\frac{p-1-q}{\omega+ p}$-concave. 
\end{example}

We move now to perturbed concavity: by assuming $a$ close to a constant function, then a power of $u$ is close to a concave function with a comparable error.
We present two possible approaches and results: see Remark \ref{rem_ulam_thm} for a comparison. 
We highlight that the estimates on which the following result is based have an interest on their own (see Theorem \ref{thm_compar_1}).

\begin{theorem}[Perturbed concavity I]
\label{thm_approx_concav_compar}
Let $\Omega \subset \R^N$, $N\geq 2$, be open, bounded, convex, with $\partial \Omega \in C^{1,\alpha}$, and let $p\in (1,+\infty)$.
Assume $a \in L^{\infty}(\Omega)$, $\alpha \in (0,1)$. 
Let $a_{\infty} \in (0,+\infty)$ be a constant and $u_{\infty}$ be a positive solution of
$$\begin{cases}
-\Delta_p u_{\infty} = a_{\infty} & \quad \hbox{ in $\Omega$},\\ 
u_{\infty} =0 & \quad \hbox{ on $\partial \Omega$}.
\end{cases}$$
Then the solution $u \in C^{0,\beta}(\overline{\Omega})$, $\beta \in (0,1]$, of \eqref{eq_general_problem*} with $q=0$ satisfies
\begin{equation}\label{eq_confr_explic}
\norm{u^{\frac{p-1}{p}}-u_{\infty}^{\frac{p-1}{p}}}_{\infty} \leq
C \norm{a -a_{\infty}}_{\infty}^{\kappa }
\end{equation}
 for some $C=C(p, \Omega, a_{\infty}, \alpha, \norm{a}_{\infty})>0$ and some $\kappa =\kappa(p, N, \beta) \in (0,1)$. 
\end{theorem}

We highlight that estimate \eqref{eq_confr_explic} can be expressed also in terms of the concavity function of $u^{\frac{p-1}{p}}$, see Remark \ref{rem_ulam_thm}.
When $u_{\infty}$ is merely concave, \eqref{eq_confr_explic} cannot give precise information on the exact concavity of $u$. If $u_{\infty}$ is assumed \emph{strictly} concave, then some information can be deduced on $\eps$-uniform concavity (see Corollary \ref{cor_uniform_eps_conc_const}). 
Nevertheless, when $u_{\infty}$ is \emph{strongly} concave, and a family of solutions is assumed to converge in $C^2$ to $u_{\infty}$, the concavity of $u_{\infty}$ is inherited definitely by such a family, for $a(x)$ sufficiently close to $a_{\infty}$. 
This result highlights that concavity of solutions may occur also in presence of nonconcave sources, and this partially answers to a question raised in \cite[Theorem 16 in Section 7]{Ken84}.
Namely, in the semilinear case we combine the above result with the information on the limiting problem given by Theorem \ref{thm_intr_recap}.

\begin{corollary}
\label{cor_strict_conc_const}
Let $\Omega \subset \R^N$, $N\geq 2$, be open, bounded, strongly convex, with $\partial \Omega \in C^{2,\alpha}$, and let $p=2$. 
Consider $(a_n)_n : \Omega \to \R$, and assume that, for some $a_{\infty}>0$ constant
$$a_n\to a_{\infty} \quad \hbox{in $C^{0,\alpha}(\Omega)$ as $n \to +\infty$}.$$
Then the positive solution $u_n$ of 
\begin{equation}\label{eq_intr_a_n_conv}
\begin{cases}
-\Delta u_n = a_n(x) & \quad \hbox{ in $\Omega$},\\ 
u_n =0 & \quad \hbox{ on $\partial \Omega$},
\end{cases}
\end{equation}
is such that $u_n$ is strongly $\frac{1}{2}$-concave for $n \geq n_0 \gg 0$. In particular, for these values of $n$, the level sets of $u_n$ are strictly convex and $u_n$ has a single (and nondegenerate) critical point in $\Omega$.
\end{corollary}

We comment now a different concavity perturbation result.
We recall, for some $a \in L^{\infty}(\Omega)$, the \emph{oscillation} related to $a$
$$\textup{osc}(a):=\sup_{\Omega} a - \inf_{\Omega} a.$$
Recall moreover the inner parallel set to $\Omega$, for any $\delta>0$,
$$\Omega_{\delta}:= \left\{ x \in \Omega \mid d(x,\partial \Omega) >\delta\right\}.$$

\begin{theorem}[Perturbed concavity II]
\label{thm_approx_concav_harmonic}
Let $\Omega \subset \R^N$, $N\geq 2$, be bounded, strongly convex, with $\partial \Omega \in C^{2,\alpha}$, and let $p\in (1,+\infty)$.
Let $q \in [0,p-1)$, and assume $a \in C^{1,\alpha}_{loc}(\Omega) \cap C^{0,\alpha}(\Omega)$, $\alpha \in (0,1)$, and $a>0$ on $\Omega$. 
Then the solution of \eqref{eq_general_problem*} satisfies 
$$\mc{C}_{u^{ \frac{p-1 - q}{p}}} \leq C \textup{osc}(a) \quad \hbox{ on $ \overline{\Omega} \times \overline{\Omega} \times [0,1]$}$$
where $C=C(u,\delta,a,p,q)>0$ is given by 
\begin{equation}\label{eq_constant_C}
C: = C_{p,q} \Big( \tfrac{ \norm{u}_{\infty}}{\min_{\overline{\Omega_{\delta}}} u}\Big)^{\frac{p-1-q}{p}} \Big( 2 + \tfrac{\osc(a)}{\min_{\overline{\Omega_{\delta}}} a}\Big) \tfrac{1}{ \min_{\overline{\Omega_{\delta}}} a} 
\end{equation}
and $C_{p,q}:=\frac{(q+1)^{\frac{1}{p}}p^{1-\frac{1}{p}} }{p-1-q}$; here $\delta>0$ small is suitably chosen.
\end{theorem}

We highlight that, since on the left hand side of \eqref{eq_constant_C} we consider the concavity of $u^{\frac{p-1-q}{p}}$, it is natural for the oscillation $\osc(a)$ to appear on the right hand side; in \cite{GMrS25} more general results -- involving the concavity function of $a$ -- are taken into account.

We notice that the constant appearing in \eqref{eq_constant_C} is bounded with respect to oscillation of $a(x)$, in the sense that it could be implicitly expressed only in terms of maximum and minimum of $a(x)$. 
As a consequence we have the following result.

\begin{corollary}
\label{cor_conv_oscul}
Let $\Omega \subset \R^N$, $N\geq 2$, be bounded, strongly convex, with $\partial \Omega \in C^{2,\alpha}$, and let $p\in (1,+\infty)$.
Let $q \in [0,p-1)$, and assume $a_{\infty}>0$ constant and $(a_n)_n \subset C^{1,\alpha}_{loc}(\Omega)\cap C^{0,\alpha}(\Omega)$, $\alpha \in (0,1)$, with $\norm{a_n - a_{\infty}}_{\infty}\to 0$ as $n \to +\infty$. 
Then solutions $(u_n)_n$ of \eqref{eq_general_problem*} with $a=a_n$ verify
 $$ \mc{C}_{u_n^{ \frac{p-1 - q}{p}}} \leq O(\norm{a_n - a_{\infty}}_{\infty}) \to 0 \quad \hbox{as $n\to +\infty$}.$$
\end{corollary}

\begin{remark}
\label{rem_ulam_thm}
As a further consequence of Theorem \ref{thm_approx_concav_harmonic}, by Hyers-Ulam Theorem \cite{HyUl52} we know that there exists a concave function $\bar{v}$ such that 
\begin{equation}\label{eq_confr_exists_conca} 
\norm{u^{ \frac{p-1 - q}{p}} - \bar{v}}_{\infty} \lesssim \textup{osc}(a); 
\end{equation}
notice that this is actually an equivalence, since the existence of such a function easily implies an estimate as the one in Theorem \ref{thm_approx_concav_harmonic}: indeed 
$$\mc{C}_{u^{ \frac{p-1 - q}{p}}} =\mc{C}_{u^{ \frac{p-1 - q}{p}} - \bar{v}} + \mc{C}_{\bar{v}} \lesssim \norm{u^{ \frac{p-1 - q}{p}} - \bar{v}}_{\infty}.$$
When $q=0$, we can compare this statement with the one in Theorem \ref{thm_approx_concav_compar}: both the results tell us that $u^{\frac{p-1}{p}}$ is near to a concave function $\bar{v}$. Theorem \ref{thm_approx_concav_compar} gives an exact information on who $\bar{v}$ is, and requires only $a \in C^{0,\alpha}(\Omega)$ (and no restriction on the sign) since no approximation argument (see Section \ref{sec_approx_argu}) is required to obtain the result; moreover, it can be generalized also to $a(x)$ close to a concave function (rather than a constant one), see Corollary \ref{corol_confron_3}. 
Theorem \ref{thm_approx_concav_harmonic}, instead, obtains $\bar{v}$ via an abstract result, but on the other hand gives a more accurate rate of convergence for the error, that is $\sim \textup{osc}(a)$ (instead of $\sim \textup{osc}(a)^{\theta \frac{\min\{1,p-1\}}{p}}$); moreover, Theorem \ref{thm_approx_concav_harmonic} is valid also for $q \neq 0$. The two results seem thus complementary. 
It would be interesting anyway to investigate some properties of the function $\bar{v}$ (e.g., if it satisfies some differential equation).
\end{remark}

\begin{remark}
We believe that the condition $a \in C^{1,\alpha}_{loc}(\R^N)$ in Theorem \ref{thm_approx_concav_harmonic} is merely technical (while $a \in L^{\infty}(\R^N)$ is crucial), due to the regularity issues of the approximation process.
Being $a \in C^1(\Omega)$ we can rephrase the conclusion of Theorem \ref{thm_approx_concav_harmonic} as
$$\mc{C}_{u^{ \frac{p-1 - q}{p}}} \leq C \norm{\nabla a}_{\infty} \quad \hbox{ on $ \overline{\Omega} \times \overline{\Omega} \times [0,1]$}$$
for some constant $C= C_{p,q}\diam(\Omega) \Big( \frac{ \norm{u}_{\infty}}{\min_{\overline{\Omega_{\delta}}} u}\Big)^{\frac{p-1-q}{p}} \Big( 2 + \frac{\diam(\Omega) \norm{\nabla a}_{\infty} }{\min_{\overline{\Omega_{\delta}}} a}\Big) \frac{1}{ \min_{\overline{\Omega_{\delta}}} a} >0$,
similarly to \cite{BuSq20}; notice anyway that the estimate involving $\osc(a)$ is much sharper, for example by considering fast oscillating weights $a(x)=\sin(k|x|)+2$ with $k\gg 0$.
 
The weakening of the regularity anyway is not only of technical interest: as a matter of fact, when studying bang-bang problems in dynamics of population \cite{FeVe22}, a discontinuous weight $a(x)$ appears naturally; the study of the shape of the solutions is thus strictly related to the shape of the portion where the resources are concentrated, and generally determining even the connectedness of this shape is an hard problem. See \cite{GMrS25} for some results in the parabolic counterpart of this problem.
 \end{remark}
 
In the paper we show also several results which are mainly based on perturbation techniques: in particular, we treat the singular case $q \in [-1,0)$ by an approximation process.
Moreover, by exploiting some uniform convergence to the eigenfunction problem, we show that if $q>p-1$ is sufficiently close to $p-1$, or $s\in (0,1)$ is sufficiently close to $1$, then the solutions of $-\Delta_p u_q = u_q^q$ and 
$$(-\Delta)^s_p u_s =\lambda_s u_s^{p-1}$$
actually enjoy some weaker form of log-concavity far from the boundary. 
Some comments on the literature of these topics and precise statements are given in the corresponding Sections \ref{sec_singular}--\ref{sec_fractional}. 
We highlight that some of the convergences provided have an interest on their own, and could be exploited for several other applications.

\medskip

We comment now the main difficulties of the paper.
The proofs of the main Theorems \ref{thm_main_conc_exact} and \ref{thm_approx_concav_harmonic} rely on a suitable approximation process; among other technical difficulties, differently from \cite{Ken85, Sak87, BMS22, BuSq20}, the co-presence of the approximating components (the ``$\eps>0$ parts'', see Section \ref{sec_approx_argu}) and the spatial-depending components (essentially $a(x)$) requires some fine management of the nonlinearities in play, in order to achieve the desired concavity and perturbed concavity results (see Section \ref{sec_gener_result}). 
Moreover, achieving the perturbed result by means of the approximation process will require also a suitable analysis of the boundary (see the proof of Theorem \ref{thm_gen_perturb_concav}).
In addition, a singular multiplicative decomposition (see \eqref{eq_singular_decompos}) will be employed to deal with the case $a(x) u^q$, $q \neq 0$. 
Finally, the particular choice of the additive decomposition of the nonlinearity (see \eqref{eq_particular_f_choice}) will also allow to unify some results contained in \cite{Ken85} and \cite{BMS22} (treating, for example, sum of powers $u^q+u^r$, see Corollary \ref{corol_sum_powers}). 
Regarding the singular case, the idea is to approximate the nonlinearity with $t \mapsto (t+\frac{1}{n})^q$ and pass to the limit the abstract result which holds for general $g$; this approach is indeed different and easier compared with the one proposed in the semilinear case by \cite{BGP99}; see Remark \ref{rem_comp_singular} for some comparison.

\medskip

\textbf{Outline of the paper:} 
in Section \ref{sec_prelimin} we recall some properties on approximation of domains, on $\alpha$-concave functions and concavity functions, together with a weaker notion of $\eps$-uniform concavity. 
In Section \ref{sec_diff_solutions} we provide some tools on the difference of two solutions in quasilinear frameworks: this topic is of independent interest, but it has some consequences related to concavity of functions. 
Then in Section \ref{sec_diff_bound} we present some general results regarding the maxima of concavity functions on the boundary or in the interior: some of the results are known, others are refinement of known results.
In Section \ref{sec_approx_argu} we develop an approximation argument for nonregular equations, which is suitable for achieving concavity properties: the discussion is set in a general framework, and some abstract consequences are presented in Section \ref{sec_gener_result}; then we provide the proofs of the main theorems in Section \ref{sec_applic}, together with some comments on perturbed concavity for eigenfunctions and the case of singular equations, where a second approximation argument is set in motion.
We collect then in Section \ref{sec_further_results} several additional results in other frameworks (superhomogeneous, $p$ large, fractional), which would be interesting to develop further in the future. 
We conclude the paper with a collection in Appendix \ref{sec_append_p_lapl} of some (partially known) tools for the $p$-Laplacian.

\medskip

\textbf{Notations:} 
we define $|x|^2:=|x|_2^2:= \sum_i x_i^2$, $|x|_1:=\sum_i |x_i|$, $\norm{u}_p^p:=\int_{\Omega} |u|^p $ for $p \in (1,+\infty)$, and $\norm{u}_{\infty}:= \textup{supess}_{x\in \Omega} |u(x)|$; $p^*:= \frac{Np}{N-p}$ for $p<N$ and $p^*:=\infty$ for $p\geq N$ will denote the Sobolev critical exponent. Moreover $[u]_{C^{0,\beta}}:= \sup_{x,y \in \Omega, \, x \neq y} \frac{|u(x)-u(y)|}{|x-y|^{\beta}}$ and $\norm{u}_{C^{0,\beta}}:= \norm{u}_{\infty} + [u]_{C^{0,\beta}}$ for $\beta \in (0,1]$, while $D^2 u$ denotes the Hessian matrix.
We define also $\textup{osc}(a):=\sup_{\Omega} a - \inf_{\Omega} a$, $d(x,\partial \Omega):= \inf_{y \in \partial \Omega} |x-y|$, $\textup{int}(\Omega)$ the interior of $\Omega$ and $\Omega_{\delta}:= \left\{ x \in \Omega \mid d(x,\partial \Omega) >\delta\right\}$; we say that $\Omega_n \to \Omega$ in Hausdorff distance if $\max\{\sup_{x \in \Omega} d(x,\Omega_n), \sup_{x \in \Omega_n} d(x,\Omega)\} \to 0$ as $n \to +\infty$. The normal vector to the boundary $\nu$ will be always assumed to pointing inward. 
Definitions of strong convex domains, (harmonic, joint) concavity functions, and of $\eps$-uniform concavity are given in Section \ref{sec_prelimin} (see also Section \ref{sec_concav_interior}); definition of ground state is given in Section \ref{sec_superhomog} while uniform ellipticity of operators is defined in \eqref{eq_strict_elliptic}.
Definitions of solution and of $\alpha$-concavity are given above (see also Section \ref{sec_prelimin}).

\section{Preliminary properties and definitions of concavity}
\label{sec_prelimin}

We start by recalling some known properties on convex sets (which are always Lipschitz \cite[Corollary 1.2.2.3]{Gri11}). 
We recall that a set is \emph{strictly convex} if $[x,y]\setminus \{x,y\} \subset \Omega$ for each $x,y \in \overline{\Omega}$, while it is \emph{strongly convex} if it is connected and the principal curvatures are well defined and strictly positive (see also \cite[Definition 3.1.2]{Ant24} for discussions on weak curvatures); we highlight that strong convexity implies strict convexity \cite[Proposition 6.1.6 and page 208]{KrPa99}. 
Moreover, for a general $C^2$ convex set, the curvatures are always nonnegative \cite[Corollary 2.1.28]{Hor94} (see also \cite[Proposition 6.18]{KrPa99}); the inequality cannot generally be improved even for strictly convex sets (consider for example $\{y>x^4\}$). 
We recall that any Lipschitz domain satisfies the uniform (interior and exterior) cone condition \cite[Proposition 2.4.4 and Theorem 2.4.7]{HePi18} (see also \cite[Theorem 1.2.2]{Gri11}), every domain $\Omega$ with $\partial \Omega \in C^{1,1}$ satisfies the uniform (interior and exterior) sphere condition \cite{LePe20}. 
For any $\delta>0$, we set
$$\Omega_{\delta}:= \left\{ x \in \Omega \mid d(x,\partial \Omega) >\delta\right\}$$
often called \emph{inner parallel bodies} or (referring to $\partial \Omega_{\delta}$) \emph{surfaces parallel to the boundary}. 
We recall that every convex domain satisfies the unique nearest point property \cite[Theorem 2.1.30]{Hor94}, while every domain $\Omega$ with $\partial \Omega \in C^2$ has a neighborhood $\Omega\setminus \Omega_{\delta}$ where this property holds \cite{Foo84}. 

We state some properties on $\Omega_{\delta}$ and some approximations of convex sets; see also \cite[Lemma 2.3.2]{Hor94}, \cite[Theorem 5.1]{BaZa17}, \cite[Theorem 5.1]{Dok76} and \cite[Theorem 3.2.1]{Ant24} for other relevant approximations. 
Recall that quasiconcavity of $v$ is equivalent to require
\begin{equation}\label{eq_def_quasiconc}
v\big(\lambda x + (1-\lambda) y\big) \geq \min \big\{v(x),v(y)\}
\end{equation}
for each $x,y \in \Omega$, $\lambda \in [0,1]$.

\begin{proposition}[Domain approximation and distance function]
\label{prop_boundar_approximat}
The following properties hold.
\begin{itemize}
\item[(i)]
Let $\Omega \subset \R^N$ be open, bounded and convex. Then there exists $(\Omega^k)_{k\in \N}$, $\Omega^k \subseteq \Omega$, strongly convex, with $\partial \Omega^k \in C^{\infty}$, and such that $\Omega^{k} \subset \Omega^{k+1}$, $\Omega = \bigcup_k \Omega^k$ and $\Omega^k \to \Omega$ in Hausdorff distance as $k\to +\infty$.
\item[(ii)] Let $\Omega \subset \R^N$ be open, bounded, and convex [resp. strictly convex]. 
Then the distance function $d(\cdot, \partial\Omega)$ is Lipschitz and concave [resp. also strictly quasiconcave].%
\footnote{We observe that the distance function is not strictly concave (even if $\Omega$ is strictly convex): for instance consider $x_1$, $x_2$ sufficiently close to $\partial \Omega$, and such that both the points lie on the same perpendicular line to the boundary (namely, if $d(x_i, \partial \Omega)=|x_i- x_i^*|$ then $x_1^*=x_2^*=:x^*$); then clearly $\bar{x} = \frac{x_1+x_2}{2}$ is such that $d(\bar{x}, \partial \Omega) = |\bar{x}-x^*| = \frac{d(x_1, \partial \Omega) + d(x_2,\partial \Omega)}{2}$.
} 
As a consequence, for any $\delta>0$, $\Omega_{\delta}$ is convex [resp. strictly convex].

\item[(iii)] Let $\Omega \subset \R^N$ be open and bounded, and let $k \in \N \cup \{\infty\}$. 
If $\partial \Omega \in C^k$, $k \geq 2$ then $d(\cdot, \partial \Omega) \in C^k(\Omega \setminus \Omega_{\delta_0})$ for some $\delta_0$;%
\footnote{We notice that a viceversa does not hold: a set satisfying $d(\cdot, \partial \Omega) \in C^k(\Omega \setminus \Omega_{\delta_0})$ is also known as \emph{proximally $C^k$}.} 
this means that, for $\delta>0$ sufficiently small, $\partial \Omega_{\delta} \in C^k$ (and $\Omega_{\delta}$ can be seen as a smooth manifold).
The same holds for $k=1$ by assuming in addition that $\partial \Omega$ satisfies the unique nearest point property (e.g., $\Omega$ is convex) and for a.e. $\delta>0$ small.%
\footnote{The size $\delta$ depends on the radius of the uniform interior sphere, or the \emph{positive reach} of $M$, that is the size of the neighborhood where the unique nearest point property holds. The set of singular points of $d(\cdot, \partial\Omega)$ is also called the \emph{ridge} of $\Omega$.}
Similar statements hold also for $C^{k,\alpha}$.
\end{itemize}
\end{proposition}

\begin{proof}
Being the properties well known, we give just a sketch of the proof. 

Claim (i) is given in \cite[Corollary 6.3.10]{KrPa99} by considering $\Omega^k$ as the preimages of a suitable smooth, strongly concave exhaustion function (see also \cite[Theorem 2.3]{Blo97}, and \cite[Theorem 1.4]{Lie85} for regularized distance function arguments).
We obtain thus that $\Omega^k$ are increasing and cover $\Omega$, and basic properties imply the Hausdorff convergence (see e.g. \cite[Section 2.2.3.2]{HePi18}).

Claim (ii) follows by \cite[Theorem 2.1.24]{Hor94}, \cite[Theorem 5.4]{DeZo94}. 
We give some details on the strict quasiconcavity: consider $x_1, x_2 \in \overline{\Omega}$, $x_1 \neq x_2$, and set $d_i:=d(x_i, \partial \Omega)\geq 0$, $\bar{x}:= \frac{x_1+x_2}{2} \in \Omega$, and $\bar{d}:=\frac{d_1+d_2}{2}\geq 0$. Fix now a direction $\xi \in S^{N-1}$, and set $y_i^{\xi}:= x_i + d_i \xi \in \overline{\Omega}$, $\bar{y}^{\xi}:=\frac{y_1+y_2}{2} =\bar{x} + \bar{d}\xi$. 
By convexity, we have $\bar{y}^{\xi} \in \overline{\Omega}$, thus varying $\xi$ we obtain $B(\bar{x}, \bar{d}) \subset \overline{\Omega}$, which implies $d(\bar{x}, \partial \Omega) \geq \bar{d}$; this gives mid-concavity (and thus concavity). We now distinguish:
\begin{itemize}
\item $d_1 \neq d_2$: in this case $d(\bar{x}, \partial \Omega) \geq \bar{d} > \min\{d_1,d_2\}$.
\item $d_1 = d_2$: in this case $y_1^{\xi} \neq y_2^{\xi}$ for each $\xi$ (otherwise $x_1 - x_2 = (d_1 - d_2) \xi = 0$, impossible), thus by strict convexity $\bar{y}^{\xi} \in \Omega$, that is $B(\bar{x}, \bar{d}) \subset \Omega$,
and hence $d(\bar{x}, \partial \Omega) > \bar{d} = \min\{d_1,d_2\}$.
\end{itemize}
In both cases we have 
$$d\left(\tfrac{x_1+x_2}{2}, \partial \Omega\right) > \min\left\{d(x_1,\partial \Omega), d(x_2, \partial \Omega)\right\}$$
that is, strict mid-quasiconcavity, which implies strict quasiconcavity, and in particular strict convexity of the level sets \cite[Corollary 3.36]{ADSZ10}.

For (iii), when $k \geq 2$, we rely on \cite[Lemma 14.16]{GiTr01}, while for $k=1$ we rely on \cite{Foo84}. 
See also \cite{KrPa81} for $C^{k,\alpha}$ conclusions ($k\geq 2$) and \cite[Theorem 5.7]{DeZo94} for $C^{1,1}_{loc}$ conclusions.

By the end of the proof of \cite[Lemma 14.16]{GiTr01} we also see that the distance is \emph{proper}, that is for small $\delta <\delta_0 $ we have that $d(x,\partial \Omega) < \delta$ implies $\nabla d(\cdot, \partial \Omega) \neq 0$; 
when $k=1$ we rely instead on Sard's lemma.
By exploiting the implicit function theorem (as in the proof of \cite[Theorem 5.1]{BaZa17}, see end of page 13 therein), we achieve that, for such $\delta$, $\partial \Omega_{\delta} \in C^k$; 
see also the preimage theorem \cite[Corollary in Section 1.2.2, page 37]{Maz11} in combination with \cite[Theorem 1.2.1.5]{Gri11} or \cite[Section 1.2]{KrPa99}. 
For Lipschitz boundaries see \cite[Theorem 4.1]{Dok76}. 
\end{proof}

We recall some miscellanea results by \cite[Properties 7 and 8, Lemma A.2]{Ken85} and \cite[Lemma 3.2]{Ken88} (see also \cite[Theorems 1 and 3, Corollary 2]{Lin81}). 
We say that a function $h: \Lambda\subset \R^N \to \R$ (not necessarily positive) is \emph{harmonic concave} if, for each $z_1,z_2 \in \Lambda$, $\lambda \in [0,1]$, 
$$ h\big(\lambda z_1 + (1-\lambda)z_2 \big) \geq \begin{cases}
\frac{h(z_1) h(z_2)}{\lambda h(z_2) + (1-\lambda) h(z_1)} &\quad \hbox{if $\lambda h(z_2) + (1-\lambda) h(z_1)>0$}, \\
 0 &\quad \hbox{if $h(z_1)=h(z_2)=0$}.
\end{cases}$$
If $h$ is positive, this definition coincides with $(-1)$-concavity, see \cite{Ken85}.

\begin{proposition}[Power concavity properties]
\label{prop_propriet_concav_varie}
We have the following properties.
\begin{itemize}
\item Let $h>0$ be $\alpha$-concave, and $k>0$ be $\beta$-concave, with $\alpha, \beta \in [0,+\infty]$. Then $hk$ is $\gamma$-concave, where $\frac{1}{\gamma}=\frac{1}{\alpha}+\frac{1}{\beta}$.
\item Let $\Omega \subset \R^N$ be convex and $I\subset \R$ be an interval, and let $h=h(x,t)$ be positive and such that $(x,t) \in \Omega \times I \mapsto t^2 h(x,t)$ is (jointly) concave. Then $h$ is (jointly) harmonic concave.
 \item If $h,k$ are $\alpha$-concave functions for $\alpha \geq 1$, then $h+k$ is $\alpha$-concave. 
 If $h$ is harmonic concave, then $h-k$ is harmonic concave for every constant $k\geq 0$.
 \item Let $h=h(x)$ be $\alpha$-concave, then $(t,x) \mapsto h(x)$ is $\alpha$ (jointly) concave; similarly if $h=h(t)$.
 \end{itemize}
\end{proposition}

\smallskip

We consider now the \emph{concavity function}, for a $v:\Omega \to \R$, as 
$$\mc{C}_{v}(x,y,\lambda):= \lambda v(x) + (1-\lambda) v(y) - v\big( \lambda x + (1-\lambda) y\big),$$
the \emph{joint concavity function}, for a $h=h(x,t): \Omega \times \R \to \R$, as
$$\mc{JC}_{h}\big((x,t),(y,s),\lambda\big):= \lambda h(x,t) + (1-\lambda) h(y,s) - h\big(\lambda x + (1-\lambda)y, \lambda t + (1-\lambda)s \big),$$
and, when $h>0$, the \emph{(jointly) harmonic concavity function} as
$$\mc{HC}_{h}\big((x,t),(y,s),\lambda\big):=\frac{h(x,t) h(y,s)}{\lambda h(y,s) + (1-\lambda) h(x,t)} - h\big(\lambda x + (1-\lambda)y, \lambda t + (1-\lambda)s \big).$$
Notice that \cite{BuSq20} 
$$\mc{HC} \leq \mc{JC}.$$
We recall some relations on $\mc{HC}$ which allows to deal with basic operations \cite[Lemma A.1]{BuSq20}.

\begin{proposition}[\cite{BuSq20}]
\label{prop_diff_harm_concav}
Let $f,g$ such that $f,g>0$. Then, for any $x,y \in \Omega$, $t, s \in \R$, $\lambda \in [0,1]$,
$$\mc{HC}_{f+g}\big((x,t),(y,s),\lambda\big) \geq \mc{HC}_f\big((x,t),(y,s),\lambda\big) + \mc{HC}_g\big((x,t),(y,s),\lambda\big)$$
and, if moreover $f-g>0$,
$$\mc{HC}_{f-g}\big((x,t),(y,s),\lambda\big) \leq \mc{HC}_f\big((x,t),(y,s),\lambda\big) - \mc{HC}_g\big((x,t),(y,s),\lambda\big).$$
\end{proposition}

\smallskip

We finally recall some definitions of generalized concavity and some related properties (see also Proposition \ref{prop_large_p} for quasiconcavity). 
\begin{definition}[$\eps$-concavity]
\label{def_unif_strong_concav}
Let $\eps>0$. We say that $v: \Omega \to \R$ is \emph{$\eps$-uniformly concave} with continuous modulus $\rho: (0,+\infty) \to (0,+\infty)$ if 
$$v\left(\frac{x+y}{2}\right)\geq \frac{v(x)+v(y)}{2} + \rho(|x-y|) \qquad \hbox{for each 
$x,y \in \Omega$, $|x-y|\geq \eps$},$$
namely $\mc{C}_v(x,y,\frac{1}{2}) \leq - \rho(|x-y|)$ for $|x-y|\geq \eps$. We say that $v$ is \emph{$\eps$-strongly concave} with parameters $m>0$ if $\rho(t) = \frac{1}{8} m t^2$.%
\footnote{The ``$\lambda$-version'' actually reads as: $v(\lambda x + (1-\lambda) y) \geq \lambda v(x) + (1-\lambda) v(y) + \frac{1}{2} \lambda(1-\lambda) m |x-y|_2^2$.
}
If $\rho \equiv const$ we simply say that $v$ is $\eps$-uniformly concave (notice that the value of $\rho$ depends a priori on $\eps$). Finally, we say that $v$ is \emph{uniformly concave} if it is $\eps$-uniformly concave for each $\eps>0$.
\end{definition}
We highlight that an $\eps$-uniform concave function is \emph{not} necessarily concave \cite[Example 2.5]{GrRa22}; anyway this class of functions enjoy several properties, see again \cite{GrRa22}. 
Clearly uniform concavity implies concavity (by assuming $v$ continuous).
Similarly, strong concavity implies both $\eps$-strong concavity and strict concavity.

We are interested in properties inherited by a sequence of functions from their limit: we observe that, if $v_n \to v$ in $C^k(\Omega)$ and $v$ is strictly convex, then it is not ensured that $v_n$ is strictly convex for $n$ large. Indeed, consider 
$v_n(x):=\begin{cases} x^{2k} & \hbox{ for $x \notin [-\frac{1}{n}, \frac{1}{n}]$}, \\ -x^{2k} + \frac{2}{n^{2k}} & \hbox{ for $x \in [-\frac{1}{n},\frac{1}{n}]$}, \end{cases}$ and $v(x):=x^{2k};$
clearly (by substituting $v_n$ with a smooth mollification) $v_n \to v$ in $C^{2k-1}([-1,1])$ and $v$ is strictly convex, but $v_n$ are not even convex. 
With the same example, we see that for $k=1$ $v_n \to v$ in $C^1([-1,1])$ (but not in $C^2([-1,1])$) and $v(x)=x^2$ is strongly concave; thus $C^1$ convergence is not sufficient to inherit convexity from strong convexity.

On the other hand, by looking at the Hessian matrix, it is clear the following result; we recall that, when $v \in C^2(\Omega)$, strong concavity is equivalent to say that $D^2 v \succcurlyeq m I$, that is, $D^2 v - m I$ is positive semi-definite (i.e., the eigenvalues of $D^2v$ lie in $[m,+\infty)$).

\begin{proposition}
\label{prop_strong_conc}
Assume $v_n : \Omega \to \R$ converge in $C^2(\Omega)$ to $v$, which is assumed to be strongly concave on $\Omega$. 
Then $v_n$ is strongly concave for $n$ sufficiently large. 
\end{proposition}

A weaker property, anyway, is inherited by the sequence also in case of $L^{\infty}$-convergence, when $v$ is assumed strictly concave.

\begin{proposition}
\label{prop_eps_strong_conc}
We have the following properties.
\begin{itemize}
\item Assume $\Omega$ bounded and $v: \overline{\Omega} \to \R$ strictly concave, then $v$ is uniformly concave. 
\item Assume $v_n : \Omega \to \R$ converge in $L^{\infty}(\Omega)$ to $v$, which is assumed uniformly [resp. strongly] concave in $\overline{\Omega}$. Then, for each $\eps>0$, there exists $n_{\eps}\gg 0$ such that, for each $n\geq n_{\eps}$, $v_n$ is $\eps$-uniformly [resp. $\eps$-strongly] concave.
\end{itemize}
\end{proposition}

\begin{proof}
Set $\Lambda_{\eps}:= \left\{(x,y) \in \overline{\Omega}\times\overline{\Omega} \mid |x-y| \geq \eps \right\}$, the first claim comes straightforwardly by considering $\rho_{\eps}:=- \sup_{\Lambda_{\eps}} \mc{C}_v(x,y,\frac{1}{2})>0$.
Consider the second point and fix $\eps>0$. Then for $(x,y) \in \Lambda_{\eps}$ we have $\mc{C}_v(x,y,\frac{1}{2}) \leq -\rho_{\eps}(|x-y|)$.
Since $\mc{C}_{v_n}\left(x,y,\tfrac{1}{2}\right) \to \mc{C}_v\left(x,y,\tfrac{1}{2}\right) $ uniformly on $\Lambda_{\eps}$, we have, for a fixed $\delta>0$ and $n$ sufficiently large, $\mc{C}_{v_n}\left(x,y,\tfrac{1}{2}\right) \leq \mc{C}_v\left(x,y,\tfrac{1}{2}\right) + \delta \leq -\rho_{\eps}(|x-y|) + \delta$ on $\Lambda_{\eps}$. 
If $\rho_{\eps}$ is constant, it is sufficient to fix $\delta < \rho_{\eps}$. 
Otherwise, if $\rho_{\eps}(t)=\frac{1}{8}m_{\eps} t^2$, then $\mc{C}_{v_n}\left(x,y,\tfrac{1}{2}\right) \leq -\frac{1}{8}m_{\eps} |x-y|^2+\frac{\delta}{\eps^2} |x-y|^2$ and we choose $\delta < \frac{1}{8} \eps^2 m_{\eps}.$
\end{proof}

\section{Estimates on the difference of two solutions}
\label{sec_diff_solutions}

In this Section we want to compare the solutions of two problems, where the space component may act differently; this comparison will in particular lead to the proof of Theorem \ref{thm_approx_concav_compar} and its generalizations. 
Since the problem is quasilinear, we cannot directly work on a problem solved by the difference of the two solutions.
We start by some considerations on general functions. 
We recall the following result by \cite[Theorems 1.1 and 1.2]{Le007}.
\begin{lemma}[\cite{Le007}]
\label{lem_estim_mixed}
Let $\Omega \subset \R^N$, $N \geq 1$, be open and with the uniform interior cone condition. Consider a function $u \in C^{0,\beta}(\overline{\Omega}) \cap L^q(\Omega)$, $\beta \in (0,1]$ and $q \in [1,+\infty)$. 
Then 
\begin{equation}\label{eq_diff_soluz_origin}
\norm{u}_{\infty} \leq C \norm{u}_{C^{0,\beta}}^{1-\kappa_q} \norm{u}_q^{\kappa_q}
\end{equation}
where $$\kappa_q:=\kappa_{N,\beta,q}:=\frac{\beta}{\beta + \frac{N}{q}} \in (0,1)$$
and $C=C_{N,\beta,q}>0$ is given as follows:
\begin{itemize}
\item if $\Omega = \R^N$, then $C_{N,\beta,q}=\max\{\frac{N}{N+\beta}, \omega_N^{-\frac{1}{q}}\}>0$ (here $\omega_N$ is the volume of the unit ball);
\item if $\Omega \neq \R^N$, let $r_0$ be the radius of the (uniform) interior cone; then $C_{N,\beta,q} = \max\big\{\frac{N}{N+\beta}, \omega_K^{-\frac{1}{q}}, r_0^{\beta}\big\}$ (here $\omega_K$ is the volume of the cone scaled to radius 1).
\end{itemize}
As a consequence of Lemma \ref{lem_estim_mixed}, if $u \in C^{k,\beta}(\overline{\Omega}) \cap W^{k,q}(\Omega)$, $k \in \N$, $\beta \in (0,1]$ and $q \in [1,+\infty)$, then
$$\norm{u}_{C^{k}} \leq C \norm{u}_{C^{k,\beta}}^{1-\kappa_q} \norm{u}_{W^{k,q}}^{\kappa_q}.$$
\end{lemma}

As a corollary of Lemma \ref{lem_estim_mixed} we gain a preliminary estimate on the difference of two general functions. 
We notice that, if $\beta \in (0,1]$, then $\kappa_{p^*} = \tfrac{p\beta }{p\beta + N-p} \in (0,1)$ (for any $p \in (1,N)$) and $\kappa_{2^*} = \tfrac{2\beta }{2\beta + N-2} \in (0,1)$ (for any $N \geq 3$).
\begin{corollary}
Let $\Omega \subset \R^N$, $N \geq 1$ be open, with the uniform interior cone condition, and let $p \in (1,+\infty)$. 
Let $u,v \in C^{0,\beta}(\overline{\Omega}) \cap W^{1,p}(\Omega)$, $\beta \in (0,1]$, then 
\begin{equation}\label{eq_stima_unif_diff_p}
\norm{u-v}_{\infty}\leq C \left(\norm{u}_{C^{0,\beta}}+\norm{v}_{C^{0,\beta}}\right)^{1-\kappa_{q}} \norm{\nabla u-\nabla v}_{p}^{\kappa_{q}}
\end{equation}
with $q=p^*$ and $C=C(N,\beta,p, \Omega)>0$ if $p<N$, while $q<\infty$ and $C=C(N,\beta, p, \Omega, q)>0$ if $p\geq N$.%
\footnote{Noticed that $\kappa_{N,\beta,q} \to 1$ and $C_{N,\beta,q} \to \max\big\{\frac{N}{N+\beta}, r_0^{\beta}\big\} >0$ as $q \to +\infty$, we see that the dependence of the constant on $q$ is essentially given by the Sobolev embedding constant of $W^{1,p}(\Omega) \hookrightarrow L^q(\Omega)$.}

Similarly, if $u,v \in C^{0,\beta}(\overline{\Omega})$ with $u-v \in D^{1,2}(\Omega)$, then 
\begin{equation}\label{eq_stima_unif_diff_2}
\norm{u-v}_{\infty}\leq C \left(\norm{u}_{C^{0,\beta}}+\norm{v}_{C^{0,\beta}}\right)^{1-\kappa_{q}} \norm{\nabla(u-v)}_{2}^{\kappa_{q}}
\end{equation}
with $q=2^*$ and $C=C(N,\beta, \Omega)>0$ if $N\geq 3$, while $q <\infty$ and $C=C(N,\beta, \Omega, q)>0$ if $N\leq 2$.
\end{corollary}

We want to prove Theorem \ref{thm_approx_concav_compar}. We discuss here a more general case.

\begin{theorem}
\label{thm_compar_1}
Let $\Omega \subset \R^N$, $N \geq 1$, be open, with the uniform interior cone condition, and let $p \in (1,+\infty)$. 
Consider $a_1, a_2 \in L^{m}(\Omega)$, $m \in (1,+\infty]$, and the problems
\begin{equation}\label{eq_confron_a1_a2}
\begin{cases}
-\Delta_p u_1 = a_1(x) & \quad \hbox{ in $\Omega$},\\ 
u_1 =0 & \quad \hbox{ on $\partial \Omega$},
\end{cases}
\qquad
\begin{cases}
-\Delta_p u_2 = a_2(x) & \quad \hbox{ in $\Omega$},\\ 
u_2 =0 & \quad \hbox{ on $\partial \Omega$},
\end{cases}
\end{equation}
with positive solutions $u_1, u_2 \in C^{0,\beta}(\overline{\Omega})$ for some $\beta \in (0,1]$. 
\begin{itemize}
\item If $p \geq 2$ assume $m \geq \frac{p}{p-2}$. 
Then 
$$\norm{u_1-u_2}_{\infty} \leq C \norm{a_1 -a_2}_{m}^{\frac{\kappa_{q}}{p-1}}$$
where $C=C(p, m, \Omega, \norm{u_1}_{C^{0,\beta}(\Omega)}, \norm{u_2}_{C^{0,\beta}(\Omega)}, q)>0$, and $q=p^*$ if $2 \leq p<N$, while $q<\infty$ if $p\geq \max\{2,N\}$.
\item If $p \leq 2$ assume $m \geq 2$ and $u_1, u_2 \in W^{1,\infty}(\Omega)$. 
Then
$$\norm{u_1 - u_2}_{\infty} \leq C \norm{a_1 -a_2}_{m}^{\kappa_{q}} $$
where $C=C(p, m, \Omega, \norm{u_1}_{C^{0,\beta}(\Omega)}, \norm{u_2}_{C^{0,\beta}(\Omega)}, \norm{\nabla u_1}_{\infty}, \norm{\nabla u_2}_{\infty}, q)>0$ and $q=2^*$ if $N\geq 3$, while $q <\infty$ if $N\leq 2$.
\end{itemize}
\end{theorem}

\begin{proof}
By definition of weak solution we have, for any $\varphi \in W^{1,p}_0(\Omega)$,
$$\int_{\Omega} \left( |\nabla u_1|^{p-2} \nabla u_1 - |\nabla u_2|^{p-2} \nabla u_2\right) \cdot \nabla \varphi = \int_{\Omega} (a_1(x) -a_2(x)) \varphi;$$
we choose $\varphi = u_1-u_2 \in W^{1,p}_0(\Omega)$, so that
$$\int_{\Omega} \left( |\nabla u_1|^{p-2} \nabla u_1 - |\nabla u_2|^{p-2} \nabla u_2\right) \cdot (\nabla u_1- \nabla u_2) 
\leq 
 \norm{a_1 -a_2}_{m} \norm{u_1-u_2}_{r} $$
for $\frac{1}{m}+\frac{1}{r}=1$, $r \leq \max\{p,2\}$. 
Let us now distinguish two cases.

\smallskip

If $p \geq 2$, 
by \cite[Section 12(I)]{Lin17}, Hölder and Poincaré inequalities we have
$$ 2^{2-p} \norm{\nabla u_1 - \nabla u_2}_p^p \leq \norm{a_1 -a_2}_{m} \norm{u_1-u_2}_{r} \leq |\Omega|^{\frac{p-r}{p r}} C_{\Omega,p} \norm{a_1 -a_2}_{m} \norm{\nabla u_1- \nabla u_2}_p ;$$
here $C_{\Omega,p}$ is the best Poincaré constant on $W^{1,p}_0(\Omega)$.
Thus
$$ \norm{\nabla u_1 - \nabla u_2}_p \leq \left( 2^{2-p} |\Omega|^{\frac{p-r}{pr}} C_{\Omega,p}\right)^{\frac{1}{p-1}} \norm{a_1 -a_2}_{m}^{\frac{1}{p-1}}.$$
Set
$C_1:=\norm{u_1}_{C^{0,\beta}(\Omega)}+\norm{u_2}_{C^{0,\beta}(\Omega)}$ by \eqref{eq_stima_unif_diff_p} we obtain
$$\norm{u_1-u_2}_{\infty} \lesssim C_1^{1-\kappa_{q}} \norm{a_1 -a_2}_{m}^{\frac{\kappa_{q}}{p-1}}$$
with $q=p^*$ if $2 \leq p<N$, while $q<\infty$ if $p\geq \max\{2,N\}$.

\medskip

If $1 < p \leq 2$ then, by \cite[Section 12(VII)]{Lin17} we have
$$ (p-1)\int_{\Omega} \frac{|\nabla u_1 - \nabla u_2|^2}{(1+|\nabla u_1|^2 + |\nabla u_2|^2)^{\frac{2-p}{2}}} \leq \norm{a_1 -a_2}_{m} \norm{u_1-u_2}_{r} \leq |\Omega|^{\frac{2-r}{2 r}} C_{\Omega,2} \norm{a_1 -a_2}_{m} \norm{\nabla u_1- \nabla u_2}_2.$$
Set
$C_2:=1+\norm{\nabla u_1}_{\infty}+ \norm{\nabla u_2}_{\infty}$ 
we gain
$$C_2^{-\frac{2-p}{2}}(p-1) \norm{\nabla u_1 - \nabla u_2}_2^2 \leq |\Omega|^{\frac{2-r}{2 r}} C_{\Omega,2} \norm{a_1 -a_2}_{m} \norm{\nabla u_1- \nabla u_2}_2$$
thus
$$\norm{\nabla u_1 - \nabla u_2}_2 \leq \left( \tfrac{1}{p-1}C_2^{\frac{2-p}{2}} |\Omega|^{\frac{2-r}{2 r}} C_{\Omega,2}\right) \norm{a_1 -a_2}_{m} .$$
Therefore, exploiting \eqref{eq_stima_unif_diff_2},
$$\norm{u_1 - u_2}_{\infty} \lesssim C_1^{1-\kappa_{q}} C_2^{\frac{2-p}{2}} \norm{a_1 -a_2}_{m}^{\kappa_{q}},$$
with $q=2^*$ if $N\geq 3$, while $q <\infty$ if $N\leq 2$.
\end{proof}

\begin{corollary}
\label{corol_confron_2}
Let $\Omega \subset \R^N$, $N \geq 2$, be open, with the uniform interior cone condition, and let $p \in (1,+\infty)$. 
Let $a_1, a_2 \in L^{\infty}(\Omega)$ and $u_1, u_2$ be positive solutions of \eqref{eq_confron_a1_a2}. 
Then 
\begin{itemize}
\item Let $p \geq 2$, and assume $\Omega$ satisfies the uniform (interior and exterior) cone condition. Then 
$$\norm{u_1-u_2}_{\infty} \leq C \norm{a_1 -a_2}_{\infty}^{\frac{\kappa_{q}}{p-1}}$$
where $C=C(p, \Omega, \norm{a_1}_{\infty}, \norm{a_2}_{\infty}, q)>0$, and $q=p^*$ if $2 \leq p<N$, while $q<\infty$ if $p\geq N \geq 2$. 
\item Let $p \leq 2$ and assume $\partial \Omega \in C^{1,\alpha}$. 
Then
$$\norm{u_1 - u_2}_{\infty} \leq C \norm{a_1 -a_2}_{\infty}^{\kappa_{q}} $$
where $C=C(p, \Omega, \alpha, \norm{a_1}_{\infty}, \norm{a_2}_{\infty}, q)>0$, and $q=2^*$ if $N\geq 3$, while $q <\infty$ if $N= 2$. 
\end{itemize}
\end{corollary}

\begin{proof}
By \cite[Corollary 4.2]{Tru67} (see also \cite[Theorem 2.3]{BPS22}) we have the $C^{0,\beta}(\overline{\Omega})$ regularity, while by \cite[Theorem 1]{Lie88} (see also \cite[Corollary 1.1]{GuVe89} and \cite[Section 4.4]{LaUr66}) we have the $C^{1,\beta}(\overline{\Omega})$ (and thus $W^{1,\infty})$ regularity. 
Hence $u$ satisfies the assumptions of Theorem \ref{thm_compar_1} and we achieve the claim, noticing that the constants $C$ therein depend in a positive power way from the norms of $u_1,u_2$, and these last two, thanks to the abovementioned regularity results (exploiting that the right hand side does not depend on $u$), can be estimated from above by the norms of $a_1,a_2$.
\end{proof}

We deal now with perturbed concavity properties that can be deduced by this perturbation result; we recall by Remark \ref{rem_ulam_thm} that such result can be express in terms of the concavity function.

\begin{corollary}
\label{corol_confron_3}
Let $\Omega \subset \R^N$, $N \geq 2$, be open, bounded, convex, and let $p \in (1,+\infty)$. 
Let $a_1, a_2 \in L^{\infty}(\Omega)$ and $u_1, u_2$ as in Corollary \ref{corol_confron_2}. 
Assume $a_2 >0$ to be $\theta$-concave, $\theta \geq 1$. 
Then $u_2$ is $\frac{\theta(p-1)}{1+\theta p}$-concave and the following facts hold. 
\begin{itemize}
\item Let $p \geq 2$. 
 Then 
 $$\norm{u_1^{\frac{\theta(p-1)}{1+\theta p}}-u_2^{\frac{\theta(p-1)}{1+\theta p}}}_{\infty} \leq C\norm{a_1 -a_2}_{\infty}^{\frac{\kappa_{q}}{p-1}\frac{\theta(p-1)}{1+\theta p}}$$
where $C=C(p, \Omega, \norm{a_1}_{\infty}, \norm{a_2}_{\infty}, q)>0$, and $q=p^*$ if $2 \leq p<N$, while $q<\infty$ if $p\geq N \geq 2$. 
\item Let $p \leq 2$ and assume $\partial \Omega \in C^{1,\alpha}$. 
Then
$$\norm{u_1^{\frac{\theta(p-1)}{1+\theta p}}-u_2^{\frac{\theta(p-1)}{1+\theta p}}}_{\infty} \leq C\norm{a_1 -a_2}_{\infty}^{\kappa_{q}\frac{\theta(p-1)}{1+\theta p}}$$
where $C=C(p, \Omega, \alpha, \norm{a_1}_{\infty}, \norm{a_2}_{\infty}, q)>0$, and $q=2^*$ if $N\geq 3$, while $q <\infty$ if $N= 2$. 
\end{itemize}
The same results apply for $\theta=\infty$ by substituting $\frac{\theta(p-1)}{1+\theta p}$ with $\frac{p-1}{p}$.
\end{corollary}

\begin{proof}
The claim follows from Corollary \ref{corol_confron_2}, Theorem \ref{thm_main_conc_exact} and the fact that, for $\gamma \in (0,1)$, 
\begin{equation*}
\norm{u_1^{\gamma}-u_2^{\gamma}}_{\infty} \leq \norm{u_1-u_2}_{\infty}^{\gamma}.
\vspace{-2em}
\tag*{\qed}
\end{equation*}
\renewcommand{\qedsymbol}{}
\end{proof}

\subsubsection{Applications}

\begin{proof}[Proof of Theorem \ref{thm_approx_concav_compar}]
We see that the claim follows by Corollary \ref{corol_confron_3}, by considering $a_2$ constant.
\end{proof}

\begin{proof}[Proof of Corollary \ref{cor_strict_conc_const}]
Since the problem is linear, we have
$$-\Delta(u_n - u) = a_n(x) - a_{\infty} \quad \hbox{in $\Omega$}$$
with $u_n-u=0$ on $\partial \Omega$. It is well known that \cite[Theorem 2.30]{FrRo22}
$$\norm{u_n-u}_{C^{2,\alpha}} \leq C \norm{a_n - a}_{C^{0,\alpha}};$$
thus, in particular, $u_n \to u$ in $C^2(\Omega)$. Hence by Theorem \ref{thm_intr_recap}, Corollary \ref{corol_confron_3} and Proposition \ref{prop_strong_conc}, we have the claim.
\end{proof}

Arguing as in Corollary \ref{cor_strict_conc_const}, a weaker version can be stated also in general convex domains.

\begin{corollary}
Let $\Omega \subset \R^N$, $N\geq 2$, be open, bounded, convex, and let $p=2$. 
Consider $(a_n)_n : \Omega \to \R$, and assume that, for some $a_{\infty}>0$ constant
$$a_n\to a_{\infty} \quad \hbox{in $C^{0,\alpha}(\Omega)$ as $n \to +\infty$}.$$
Then, for each $\delta>0$, the positive solution $u_n$ of \eqref{eq_intr_a_n_conv} is such that $u_n$ is strongly $\frac{1}{2}$-concave in $ \Omega_{\delta}$, for $n \geq n_0(\delta) \gg 0$. 
In particular, for these values of $n$, the level sets of $u_n$ are strictly convex in $\Omega_{\delta}$ and $u_n$ has a single (and nondegenerate) critical point in $\Omega_{\delta}$. 
\end{corollary}

A weaker version of Corollary \ref{cor_strict_conc_const} can be stated also in the quasilinear setting. 
\begin{corollary}
\label{cor_uniform_eps_conc_const}
Let $\Omega \subset \R^2$, be open, bounded, strictly convex, with $\partial \Omega \in C^{1,\alpha}$,
and let $p\in (1,+\infty)$.
Consider $(a_n)_n : \Omega \to \R$, equibounded in $C^{0,\alpha}(\Omega)$, $\alpha \in (0,1)$, 
and assume that, for some $a_{\infty}>0$ constant
$$a_n\to a_{\infty} \quad \hbox{in $L^{\infty}(\Omega)$ as $n \to +\infty$}.$$
Then, for any $\eps>0$, the positive solution $u_n$ of 
$$\begin{cases}
-\Delta_p u_n = a_n(x) & \quad \hbox{ in $\Omega$},\\ 
u_n =0 & \quad \hbox{ on $\partial \Omega$},
\end{cases}$$
is such that $u_n$ is $\eps$-uniformly $\frac{1}{2}$-concave for $n \geq n_0(\eps) \gg 0$.
\end{corollary}

\begin{proof}
By Theorem \ref{thm_intr_recap_plaplac} we have that $v_{\infty}=\sqrt{u_{\infty}}$ is strictly concave in $\Omega$. 
By Hopf boundary lemma \cite[Theorem 5]{Vaz84}, and Corollary \ref{corol_concav_bounda} we observe that, being $\Omega$ strictly convex, $\mc{C}_{v_{\infty}}(x,y,\frac{1}{2})$ cannot be zero if $|x-y|\geq \eps$. 
Thus $v_{\infty}$ is strictly concave in $\overline{\Omega}$, and the claim follows by Corollary \ref{corol_confron_3} and Proposition \ref{prop_eps_strong_conc}.
\end{proof}

We refer also to Corollary \ref{corol_hard_hen_perturb} for another application of Corollary \ref{corol_confron_3}, where $a_2$ is nonconstant, based on Theorem \ref{thm_main_conc_exact} and Example \ref{exam_hardy_henon}. 

\begin{remark}
\label{rem_diff_log}
We see that Corollary \ref{corol_confron_2} can be applied to deduce some information on $\log$-concavity. 
Indeed, as in the proof of Lemma \ref{lem_uniquen}, we first have $\frac{u_1}{u_2}, \frac{u_2}{u_1} \in L^{\infty}(\Omega)$.
Thus, by the mean value theorem, there exists $\lambda^* \in (0,1)$ such that
\begin{align*}
\pabs{\log(u_1)-\log(u_2)} &= \tfrac{1}{2}\pabs{\log\big(\tfrac{u_1}{u_2}\big) - \log\big(\tfrac{u_2}{u_1}\big)} 
 = \frac{1}{2(\lambda^*\frac{u_1}{u_2} + (1-\lambda^*) \frac{u_2}{u_1})} \pabs{\tfrac{u_1}{u_2} - \tfrac{u_2}{u_1}} \\
 & \leq C \pabs{\tfrac{u_1}{u_2} + \tfrac{u_2}{u_1}} |u_1-u_2|
 \leq C' |u_1-u_2|,
\end{align*}
with $C' = C'(\norm{\frac{u_1}{u_2}}_{\infty}, \norm{\frac{u_2}{u_1}}_{\infty})$. 
Hence by Corollary \ref{corol_confron_2}, we have
$$\norm{\log(u_1)-\log(u_2)}_{\infty} \leq C \norm{a_1 -a_2}_{\infty}^{\kappa}$$
for some $\kappa \in (0,1)$, and the constant $C$ depends also on $\norm{\frac{u_1}{u_2}}_{\infty}$ and $\norm{\frac{u_2}{u_1}}_{\infty}$.

We notice that, clearly, $\log(u_1), \log(u_2) \notin L^{\infty}(\Omega)$, but the difference does so. On the other hand we have, for each $\delta>0$, $\log(u_1), \log(u_2) \in L^{\infty}(\Omega_{\delta})$, thus, if $a_n \to a_0$ uniformly, then
$$\log(u_n) \to \log(u_0) \quad \hbox{ in $L^{\infty}(\Omega_{\delta})$};$$
if for example one is able to deduce that $\log(u_0)$ is strictly or strongly concave in $\Omega_{\delta}$, then some concavity information on $\mc{C}_{\log(u_n)}$ for $n$ large can be deduced as well, see Propositions \ref{prop_strong_conc} and \ref{prop_eps_strong_conc}. 
\end{remark}

\section{General lemmas about concavity} 
\label{sec_diff_bound}

To discuss the concavity of a function $v$, we move the attention to the negativity of its concavity function $\mc{C}_v$ and, in particular, to its supremum. The discussion will be splitted according to the location of the maximum, that is, on (or near) the ``boundary'' or the ``interior'' of the domain.
We highlight that the function $v$ will take the place of the transformation of a solution $u$, see Corollary \ref{corol_concav_bounda} below.

\subsection{Concavity on the boundary}

In this Section we discuss the possibility of the concavity function $\mc{C}_v$ to have a maximum on the boundary of its domain. 
Here we do not use the equation, but only the information on the boundary; see anyway \cite[Lemma 3.3]{BGP99} where also the equation is exploited (in a singular framework). 

\smallskip

We start by a result which deals \emph{precisely} with the boundary of (a general) $\Omega$, but gives no good quantitative information to be inherited by functions approximating $v$ (compare it with Proposition \ref{prop_bound_restricted_info} below and subsequent comments). 

\begin{lemma}
\label{lemma_concav_bordo}
Let $\Omega \subset\R^N$, $N\geq 1$, be open, bounded, convex. Let $v:\overline{\Omega} \to \R$ and 
 $(\bar{x},\bar{y}, \bar{\lambda}) \in \overline{\Omega} \times \partial \Omega \in [0,1]$.
Then the following facts hold. 
\begin{itemize}
\item[a)] if $[\bar{x},\bar{y}] \subset\partial \Omega$, assume 
\begin{equation}\label{eq_cond_bordo_extra}
v=const \; \hbox{ on $[\bar{x},\bar{y}]$}; 
\end{equation}
then $\mc{C}_v(\bar{x},\bar{y}, \bar{\lambda})=0$;
\item[b)] if $[\bar{x},\bar{y}] \not \subset \partial \Omega$ and $\bar{\lambda} \in (0,1)$, assume
\begin{equation}\label{eq_cond_per_max_cont}
\mc{C}_v(\bar{x},\bar{y},\lambda) <0 \quad \hbox{for $\lambda \approx 0^+$};
\end{equation}
then $(\bar{x},\bar{y}, \bar{\lambda})$ is not a global maximum of $\mc{C}_v$$_{|\overline{\Omega} \times \overline{\Omega} \times [0,1]}$. 
\end{itemize}
As a consequence, if conditions in a)-b) hold, then $(\bar{x},\bar{y}, \bar{\lambda})$ is not a positive global maximum of $\mc{C}_v$. 
As a further consequence, if $\Omega$ is strictly convex, $\bar{\lambda}\in (0,1)$, $\bar{x}\neq \bar{y}$ and condition in b) hold, then $(\bar{x},\bar{y}, \bar{\lambda})$ is not a global maximum of $\mc{C}_v$.
\end{lemma}

\begin{proof}
We can assume $\bar{\lambda}\in (0,1)$, otherwise $\mc{C}_v(\bar{x},\bar{y}, \bar{\lambda})=0$. Set $\bar{z}:= \bar{\lambda} \bar{x} + (1-\bar{\lambda})\bar{y}$. 
If $[\bar{x},\bar{y}]\subset \partial \Omega$, then $\bar{z} \in \partial \Omega$ and hence by \eqref{eq_cond_bordo_extra} we have $v(\bar{x})=v(\bar{y})=v(\bar{z})$, which implies $\mc{C}_v(\bar{x},\bar{y}, \bar{\lambda}) =0$. 
Thus we focus on $[\bar{x},\bar{y}] \not \subset \partial \Omega$.

Roughly speaking, by the assumptions $v$ is concave along the segment $[\bar{x},\bar{y}]$, near $\bar{y}$; by moving $\bar{y}$ a little closer to $\bar{x}$, but keeping their convex combination the same, we expect that this procedure increases the value of $\mc{C}_v$.
Let us see this in details. 
By \eqref{eq_cond_per_max_cont}, for some small $\lambda^* \in (0, \bar{\lambda})$, ($\lambda^* \approx 0$) we have
$$\mc{C}_v(\bar{x},\bar{y},\lambda^*)< 0.$$
Set $\bar{y}^*:=\lambda^* \bar{x} + (1-\lambda^*) \bar{y}$, ($\bar{y}^* \approx \bar{y}$) it means
\begin{equation}\label{eq_dim_z*}
v(\bar{y}^*) > \lambda^*v(\bar{x}) + (1-\lambda^*) v(\bar{y}).
\end{equation}
Choose $\mu \in (0,1)$ in such a way 
$\mu \bar{x} + (1-\mu) \bar{y}^* = \bar{\lambda} \bar{x} + (1-\bar{\lambda})\bar{y}$, 
that is
$$\mu:= \frac{\bar{\lambda}-\lambda^*}{1-\lambda^*} \in (0,1);$$
notice that $\bar{y}^*$, rather than $\bar{y}$, is closer to $\bar{x}$, while $\mu$ maintains the middle combination constant.
Moreover \eqref{eq_dim_z*} clearly implies
$$\mu v(\bar{x}) + (1-\mu) v(\bar{y}^*) > \bar{\lambda} v(\bar{x}) + (1-\bar{\lambda}) v(\bar{y})$$
from which $\mc{C}_v(\bar{x}, \bar{y}^*, \mu) > \mc{C}_v(\bar{x},\bar{y},\bar{\lambda})$, that is the claim. 
We conclude by observing that case a) cannot occur on a strictly convex domain, unless $\bar{x}=\bar{y}$.
\end{proof}

\begin{remark}
\label{rem_direction_derivat}
In order to apply Lemma \ref{lemma_concav_bordo}, we notice that \eqref{eq_cond_per_max_cont} is equivalent to check
$$ \frac{v\big(\bar{y} + \lambda(\bar{x}-\bar{y})\big)-v(\bar{y})}{\lambda} > v(\bar{x})-v(\bar{y}) \quad \hbox{for $\lambda \approx 0^+$}.$$
Thus it is sufficient to verify
\begin{equation}\label{eq_rem_implic_Cv}
\limsup_{\lambda \to 0^+} \frac{v\big(\bar{y} + \lambda(\bar{x}-\bar{y})\big)-v(\bar{y})}{\lambda} > v(\bar{x})- v(\bar{y}) 
\end{equation}
or equivalently, if $v \in C^1(\overline{\Omega})$ 
$$\partial_{\bar{x}-\bar{y}} v(\bar{y}) \equiv \nabla v(\bar{y}) \cdot (\bar{x}-\bar{y}) > v(\bar{x})-v(\bar{y})$$
or similarly
$$v(\bar{x}) < \nabla v(\bar{y}) \cdot (\bar{x}-\bar{y}) + v(\bar{y}).$$
Notice that, saying
$$v(x) < \nabla v(\bar{y}) \cdot (x-\bar{y}) + v(\bar{y}) \quad \hbox{for $x \in \overline{\Omega}\setminus\{\bar{y}\}$}$$
means that $v$ lies strictly beneath the plane tangent to $v$ in the boundary point $\bar{y} \in \partial \Omega$ (see also \eqref{eq_cond_bordo_piani}).
\end{remark}

\begin{remark}
\label{rem_strong_assump_nderiv}
By assuming $v\in C^1(\Omega)$ and the stronger assumption -- which clearly implies \eqref{eq_rem_implic_Cv} -- 
\begin{equation}\label{eq_rem_strong_cond_Cv}
\limsup_{\lambda \to 0^+} \frac{v\big(\bar{y} + \lambda(\bar{x}-\bar{y})\big)-v(\bar{y})}{\lambda} = + \infty 
\end{equation}
the proof of Lemma \ref{lemma_concav_bordo} can be simplified: indeed, in the case $[\bar{x},\bar{y}]\not \subset \partial \Omega$, being for $\eps$ small $\bar{y} + \eps(\bar{y}-\bar{x}) \in \Omega$ and exploiting $\mc{C}_v(\bar{x},\bar{y},\bar{\lambda}) \geq \mc{C}_v\big(\bar{x},\bar{y} + \eps(\bar{y}-\bar{x}), \bar{\lambda}\big)$, we get 
$$(1-\bar{\lambda}) \frac{v(\bar{y})-v\big(\bar{y}+\eps(\bar{y}-\bar{x})\big)}{\eps} \geq \frac{v\big(\bar{\lambda}\bar{x} + (1-\bar{\lambda})\bar{y}\big) - v\big(\bar{\lambda} \bar{x} + (1-\bar{\lambda})\bar{y} + \eps (1-\bar{\lambda})(\bar{y}-\bar{x})\big)}{\eps}.$$
Thus, sending $\eps \to 0^+$, we obtain $- \infty \geq (1-\bar{\lambda})\nabla v(\bar{\lambda}\bar{x} + (1-\bar{\lambda})\bar{y}) \cdot (\bar{y}-\bar{x});$
being $\bar{\lambda}\bar{x} + (1-\bar{\lambda})\bar{y} \in \Omega$, we have a contradiction.
\end{remark}

We can finally deal with transformations of a function $u$; notice that we are not requiring $\varphi$ to be monotone nor concave.

\begin{corollary}
\label{corol_concav_bounda}
Let $\Omega \subset\R^N$, $N\geq 1$, be open, bounded and convex. 
Let $k \in \R$ and $u$ be a function with 
$$u=k \; \hbox{ and } \; \partial_{\nu} u>0 \quad \hbox{ on $\partial \Omega$};$$
here $\nu$ is the interior normal vector. 
Let $\varphi: \R \to \R$ be such that $\varphi \in C^1(\R \setminus \{k\})$ and
$$\limsup_{t \to k} \varphi'(t)= + \infty.$$ 
Set $v:= \varphi(u)$. 
If $(\bar{x},\bar{y}, \bar{\lambda})$ is a global maximum of $\mc{C}_v$$_{|\overline{\Omega} \times \overline{\Omega} \times [0,1]}$, then $\mc{C}_v(\bar{x},\bar{y}, \bar{\lambda})=0$ or $(\bar{x},\bar{y}) \in \textup{int}\left(\overline{\Omega} \times \overline{\Omega}\right)$.
In particular, if $\Omega$ is strictly convex, then $\mc{C}_v(\bar{x},\bar{y}, \bar{\lambda})=0$ may happen only if $\bar{x}=\bar{y}$ or $\bar{\lambda} \in \{0,1\}$.
\end{corollary}

\begin{proof}
Assume by contradiction $\mc{C}_v(\bar{x},\bar{y}, \bar{\lambda})>0$ and $(\bar{x},\bar{y}) \notin \textup{int}\left(\overline{\Omega} \times \overline{\Omega}\right)$; with no loss of generality, we assume
$\bar{y} \in \partial \Omega$, $\bar{x}\in \overline{\Omega}$; we want to ensure condition \eqref{eq_rem_implic_Cv} (actually, \eqref{eq_rem_strong_cond_Cv}) in order to apply Lemma \ref{lemma_concav_bordo}. 
Fixed such points, by $\partial_{\nu(\bar{y})} u(\bar{y})>0$, we have that also $\partial_{\bar{x}-\bar{y}} u(\bar{y}) > 0$ on $\partial \Omega$ (here we use that $[\bar{x},\bar{y}] \not \subset \partial \Omega$, thus $\bar{x}-\bar{y} \not \perp \nu$ is pointing inward). 
Thus
$$\ell_{\bar{x},\bar{y}}:= \partial_{x-\bar{y}} u(\bar{y}) = \lim_{\lambda \to 0^+} \frac{u(\bar{y}+\lambda (\bar{x}-\bar{y})) - u(\bar{y})}{\lambda} >0;$$
in particular, $u(\bar{y}+\lambda (\bar{x}-\bar{y})) >u(\bar{y})$ for $\lambda$ small.
As a consequence, by the mean value theorem, for each $\lambda$ small there exists $t(\lambda) \in \big(u(\bar{y}+\lambda(\bar{x}-\bar{y})), u(\bar{y})\big)$ (and thus $t(\lambda) \to u(\bar{y})=k$ as $\lambda \to 0$) such that
\begin{align*}
\MoveEqLeft \limsup_{\lambda \to 0^+} \frac{v\big(\bar{y} + \lambda(\bar{x}-\bar{y})\big)-v(\bar{y})}{\lambda} = \limsup_{\lambda \to 0^+} \frac{\varphi\big(u\big(\bar{y} + \lambda(\bar{x}-\bar{y})\big)\big)-\varphi(u(\bar{y}))}{\lambda} \\
& = \limsup_{\lambda \to 0^+} \bigg(\varphi'(t(\lambda)) \frac{u\big(\bar{y} + \lambda(\bar{x}-\bar{y})\big)-u(\bar{y})}{\lambda}\bigg) = + \infty >v(\bar{x})-v(\bar{y}).
\tag*{\qed}
\end{align*}
\renewcommand{\qedsymbol}{}
\vspace{-1em}
\end{proof}

We move now to a result which gives better information on the concavity function in a tubular neighborhood of the boundary, whenever this is assumed strongly convex; namely, $v=\varphi(u)$ is shown to be strictly convex near the boundary. 
We refer to \cite[Propositions 2.2, 2.3 and 2.4]{MRS24} (see also \cite[Lemmas 2.1 and 2.4]{Kor83Co}, \cite[Lemma 4.3]{CaFr85}, \cite[Proposition 3.2]{Sak87}).

\begin{proposition}[\cite{MRS24}]
\label{prop_bound_restricted_info}
Let $\Lambda \subset \R^N$, $N\geq 1$, be open, bounded, strictly convex, with $\partial \Lambda \in C^{2,\alpha}$. 
Let moreover $v \in C^1(\overline{\Lambda})$ be such that
\begin{equation}\label{eq_cond_bordo_piani}
v(x) > \nabla v (y) \cdot (x-y) + v(y) 
\end{equation}
for any $y \in \partial \Lambda$ and $x \in \overline{\Lambda}$, $x \neq y$. Then all the global positive maxima of $\mc{C}_v$$_{|\overline{\Lambda} \times \overline{\Lambda} \times [0,1]}$ lie in $\textup{int}\left(\overline{\Lambda} \times \overline{\Lambda}\right)\times [0,1]$.
\end{proposition}

\begin{proposition}[\cite{MRS24}]
\label{prop_info_bound_c2}
Let $\Omega \subset \R^N$, $N\geq 1$, be open, bounded, strongly convex, with $\partial \Omega \in C^{2,\alpha}$. 
Let moreover $u \in C^1(\overline{\Omega}) \cap C^2(\overline{\Omega} \setminus \Omega_{\eta})$ for some $\eta>0$, such that
$$u>0 \; \hbox{ in $\Omega$}, \quad u=0 \; \hbox{ on $\partial \Omega$}, \quad \partial_{\nu}u>0 \; \hbox{ on $\partial \Omega$}.$$
Let $\varphi \in C^2((0,+\infty), \R)$ be such that
$$\varphi''<0<\varphi' \; \hbox{ near $0$}, \quad \lim_{t \to 0^+} \varphi'(t)=+\infty, \quad \lim_{t \to 0^+} \frac{\varphi(t)}{\varphi'(t)} =\lim_{t\to 0^+} \frac{\varphi'(t)}{\varphi'(t)} =0.$$
Set $v=\varphi(u)$. Then there exists $\delta \in (0,\eta)$ such that
\begin{equation}\label{eq_hess_def_neg}
D^2 v(x) \quad \hbox{is definite negative for any $x\in\Omega \setminus \Omega_{\delta}$}
\end{equation}
and \eqref{eq_cond_bordo_piani} holds for any $y \in \Omega \setminus \Omega_{\delta}$ and $x \in \Omega$, $x \neq y$.
\end{proposition}

In the previous result, $D^2 v$ denotes the Hessian matrix of $v$. 
The strategy of application of the previous results is the following: while $\Omega$ is the domain of reference, $\Lambda$ will be a smaller domain, where the inequality \eqref{eq_cond_bordo_piani} holds actually up to the boundary $\partial \Lambda \subset \Omega \setminus \Omega_{\delta}$.
Moreover, this information is inherited by $C^2$ approximating sequences, as shown by the following result.

\begin{proposition}[\cite{MRS24}]
\label{prop_concav_bound_perturb}
Let $\Lambda \subset \R^N$, $N\geq 1$, be convex, bounded and satisfying the interior sphere condition.
Let moreover $v_{\eps}$ be such that, as $\eps \to 0$, $v_{\eps} \to v$ in $C^1(\overline{\Lambda}) \cap C^2(\overline{\Lambda} \setminus \Lambda_{\eta})$ for some $\eta>0$, where $v$ satisfies \eqref{eq_cond_bordo_piani} and \eqref{eq_hess_def_neg}. 
Then, for $\eps$ sufficiently small, 
all the global positive maxima of $\mc{C}_{v_{\eps}}$$_{|\overline{\Lambda} \times \overline{\Lambda} \times [0,1]}$ lie in $\textup{int}\left(\overline{\Lambda} \times \overline{\Lambda}\right)\times [0,1]$.
\end{proposition}

\begin{remark}[Propagation from the boundary]
We notice that an information of the type \eqref{eq_hess_def_neg}, for some class of problems (requiring, in particular, $p=2$, $f(x,t)\equiv g(t)$ and $g(0)>0$, such as the semilinear torsion problem) automatically implies the concavity of the solution on the whole domain \cite{KeMn93, Ste22, ChWe23}.
\end{remark}

 \subsection{Concavity and perturbed concavity in the interior}
\label{sec_concav_interior}

We deal now with the information on the concavity function in the interior of $\Omega$; here the equation solved by $u$ plays its role.
We start by recalling some results on exact concavity, see \cite[Theorem 3.1]{Ken85} and \cite[Proposition 2.1]{MRS24}.

\begin{theorem}
\label{thm_concav_inter_class}
Let $\Omega \subset \R^N$, $N\geq 2$, be open, bounded, 
convex, and let $v \in C^2(\Omega) \cap C(\overline{\Omega})$ be a solution of
$$-\sum_{i,j} a_{ij}(\nabla v) \partial_{ij} v = b(x,v,\nabla v) \quad \hbox{in $\Omega$}.$$
Assume $a_{ij}$ is uniformly elliptic, that is 
\begin{equation}\label{eq_strict_elliptic}
 \sum_{i,j} a_{ij}(\xi) \eta_i \eta_j \geq C |\eta|^2 \quad \hbox{for each $\xi \in (\nabla v)(\Omega)$ and $\eta \in \R^N$}
 \end{equation}
and 
$$t \in v(\Omega) \mapsto b(x,t, \xi) \; \hbox{nonincreasing $\quad$ for $\xi \in (\nabla v)(\Omega)$},$$
$$(x,t) \in \Omega \times v(\Omega) \mapsto b(x, t, \xi) \; \hbox{jointly harmonic concave $\quad$ for $\xi \in (\nabla v)(\Omega)$}.$$
Assume moreover one of the following:
\begin{itemize}
\item $t \in v(\Omega) \mapsto b(x, t, \xi)$ strictly decreasing for $\xi \in (\nabla v)(\Omega)$,
\item $(x,t) \in \Omega \times v(\Omega) \mapsto b(x, t, \xi) $ strictly jointly harmonic concave for $\xi \in (\nabla v)(\Omega)$,
\item $a_{ij}$ and $b$ smooth enough, namely $a_{ij} \in C^1(\R^N)$, $\nabla_x b, \nabla_{\xi} b \in L^{\infty}_{loc}(\Omega\times \R\times\R^N)$.
\end{itemize}
Let $(\bar{x},\bar{y},\bar{\lambda}) \in \Omega \times \Omega \times [0,1]$ be a maximum of $\mc{C}_v$$_{|\overline{\Omega} \times \overline{\Omega} \times [0,1]}$. Then 
$$\nabla v(\bar{x}) = \nabla v(\bar{y}) = \nabla v \big(\bar{\lambda}\bar{x} + (1-\bar{\lambda})\bar{y}\big)$$
and 
$\mc{C}_v(\bar{x},\bar{y},\bar{\lambda}) \leq 0$. 
\end{theorem}

\begin{proof}
The same result for $b$ strictly decreasing is given in \cite[Theorem 3.1]{Ken85} (see also \cite[Theorem 3.13]{Kaw85R}), from whose proof it is clear that $b$ strictly harmonic concave works as well. 
To deal with the nonstrict case, the problem under $a_{ij} \in C^{1,\alpha}_{loc}(\R^N)$, $b(\cdot, t, \cdot) \in C^{1,\alpha}_{loc}(\Omega \times \R^N)$ for some $\alpha \in (0,1)$ has been considered in \cite[Lemma 1.5 and Theorem 1.3]{Kor83Co} and \cite[Lemma 3.2]{CaSp82}, where the proof is based on a perturbation argument (see also \cite[proof of Proposition 2.8]{BuSq20}). 
A different and more direct proof, which allows to assume less regularity, has been given in \cite[Proposition 2.1]{MRS24} (see also \cite[Theorem 2.1]{GrPo93}). 
\end{proof}

We recall now some theorems about perturbed concavity. 
To state the results we will use the following notation: for $u: \Omega \to \R$ and $b: \Omega \times \R \to \R$ we write
\begin{align*}
\mc{JC}_{b(\cdot, u(\cdot))}(x,y,\lambda) := & \, \mc{JC}_{b}\big((x,u(x)), (y,u(y)), \lambda\big) \\
=& \, \lambda b(x,u(x)) + (1-\lambda) b(y,u(y)) - b\big(\lambda x + (1-\lambda)y, \lambda u(x) + (1-\lambda)u(y) \big),
\end{align*}
\begin{align*}
\mc{HC}_{b(\cdot, u(\cdot))}(x,y,\lambda) := & \, \mc{HC}_{b}\big((x,u(x)), (y,u(y)), \lambda\big) \\
=& \, \frac{b(x,u(x)) b(y,u(y))}{\lambda b(y,u(y)) + (1-\lambda) b(x,u(x))} - b(\lambda x + (1-\lambda)y, \lambda u(x) + (1-\lambda)u(y) \big).
\end{align*}

\smallskip

The following theorem is given in \cite[Lemmas 2.3 and 2.9]{BuSq20}; see also \cite[Theorems 2.2 and 2.3]{ABCS23}.%
\footnote{We highlight that in \cite{BuSq20, ABCS23} a different notation for $\mc{JC}$ and $\mc{HC}$ has been used: the concavity function has opposite sign (and also a reflection in $\lambda$ in \cite{ABCS23}), while the equation has the form $a_{ij}(\nabla v) = b(x,v,\nabla v)$. To obtain what we state now, it is sufficient to consider $v \leadsto -v$, $a_{ij}(\xi) \leadsto a_{ij}(-\xi)$, $b(x,t,\xi) \leadsto b(x,-t,-\xi)$; in particular, the monotonicity in $t$ of the source changes.}
See Section \ref{sec_eigen_phi_unbound} for some comments on the case $\mu=0$.

\begin{theorem}[\cite{BuSq20}]
\label{thm_approx_concav_bucsquas}
Let $\Omega \subset \R^N$ be open, convex, and let $v \in C^2(\Omega) \cap C(\overline{\Omega})$ be a solution of 
$$-\sum_{i,j} a_{ij}(\nabla v) \partial_{ij} v = b(x,v,\nabla v)$$
with $a_{ij}$ symmetric and uniformly elliptic. 
Assume that $\mc{C}_v$ assume a positive interior maximum $(\bar{x},\bar{y},\bar{\lambda}) \in \Omega \times \Omega \times [0,1]$, in particular, set $\bar{z}:=\bar{\lambda} \bar{x} + (1-\bar{\lambda})\bar{y}$, we have $v(\bar{z}) < \bar{\lambda} v(\bar{x}) + (1-\bar{\lambda}) v(\bar{y}).$ 
Then 
$$\nabla v(\bar{x})=\nabla v(\bar{y}) =\nabla v(\bar{z})=:\bar{\xi}.$$
Assume moreover that $b$ strictly decreases over the interested segment, that is
$$\partial_t b(\bar{z}, t, \bar{\xi}) \leq -\mu <0 \quad \hbox{for $t \in [v(\bar{z}), \bar{\lambda} v(\bar{x}) + (1-\bar{\lambda})v(\bar{y})]$}.$$
Then
$$\mc{C}_v(\bar{x},\bar{y},\bar{\lambda}) \leq \frac{1}{\mu} \mc{HC}_{b(\cdot, v(\cdot), \bar{\xi})}(\bar{x},\bar{y},\bar{\lambda}) 
 \leq \frac{1}{\mu} \mc{JC}_{b(\cdot, v(\cdot), \bar{\xi})} (\bar{x},\bar{y},\bar{\lambda}).$$
\end{theorem}

\begin{remark}[Comparison with mid-concavity]
\label{rem_mid_concav}
It is known that if $v$ is continuous on an open set, then
$$\mc{C}_v(x,y,\lambda) \leq 0 \; \hbox{ for each $\lambda \in [0,1]$} \iff \mc{C}_v(x,y,\tfrac{1}{2})\leq 0;$$
that is, $v$ is concave if and only if it is \emph{mid-concave}
$$\mc{C}^m_v(x,y):= \frac{v(x)+v(y)}{2} - v\left(\frac{x+y}{2}\right) \leq 0.$$
It is thus possible to develop the above theory with $\lambda$ fixed to $\frac{1}{2}$, as done e.g. in \cite{Kaw85R, GrPo93, Juu10}, but we use here the full $\mc{C}_v$, since the statements on the perturbed concavity are stronger.
We anyway highlight some differences now. 

The arguments of Remark \ref{rem_strong_assump_nderiv} still apply, thus Corollary \ref{corol_concav_bounda} holds true. 
In \cite[Theorem 2.1]{GrPo93} Theorem \ref{thm_concav_inter_class} is shown for $\mc{C}^m$, but the argument can be extended to $\mc{C}$ \cite[proof of Proposition 2.1]{MRS24}. 
See also \cite[Theorems 3.4 and 3.5]{Juu10} for a result on viscosity solutions.

Moreover, in \cite[Lemma 3.2]{GrPo93} (see also \cite[Lemma 3.12]{Kaw85R}), when $v$ negatively explodes on the boundary (that is the case of $v=\log(u)$), the authors provide a tool which ensures that $\mc{C}^m_v$ cannot get positive while approaching the boundary, in the sense: for any $x_n, y_n \in \Omega$
\begin{equation}\label{eq_distante_bordo}
d\big( (x_n, y_n), \partial(\Omega \times \Omega)\big) \to 0 \implies \limsup_{n \to +\infty} \mc{C}^m_v (x_n, y_n) \leq 0;
\end{equation}
here they use that $\Omega$ is strictly convex and $\lambda = \frac{1}{2}$ in a crucial way; see anyway \cite[Lemma 1.30]{Mor24} for a generalization.

This information is not necessary when working with exact concavity, since the argument is set on a fixed $\Omega_{\delta}$, where $v=\varphi(u)$ is bounded. 
On the other hand, this a priori information on the whole $\Omega$ is indeed used for perturbed concavity: thus, to get a perturbed concavity result on $\log(u)$ we need to work with $\mc{C}^m$; see Theorem \ref{thm_weigh_eigen_sigma}. 
By \cite[Corollary 1]{NgNi93}, from a bound on $\mc{C}^m_{\log(u)}$ we can obtain a bound also on $\mc{C}_{\log(u)}$ up to doubling the error; this still allows to apply Hyers-Ulam Theorem \cite{HyUl52}. See Remark \ref{rem_varphi_unbounded} for details.
\end{remark}

\section{The approximation argument}
\label{sec_approx_argu}

We set up now our approximation argument, inspired by \cite{Sak87, BMS22, MRS24}. 
The aim of this Section (Steps \hyperlink{Step1}{\rm{1}}--\hyperlink{Step8}{\rm{8}}) is to build a regularized problem to whose solutions the concavity results in Section \ref{sec_diff_bound} apply (up to a proper choice of the nonlinearity and of the transformation), and whose solutions convergence to the one of the original problem. 
In this process, we need to deal also with the dependence on $x$.
Since the approximation problem can be adapted to a great variety of nonlinearities (which go beyond the examples we treat in this paper), we keep the argument for a general $f=f(x,t)$ (set $F(x,t):= \int_0^t f(x,\tau) d \tau$)
\begin{equation}\label{eq_general_fxu}
\begin{cases}
-\Delta_p u = f(x,u) & \quad \hbox{ in $\Omega$},\\ 
u >0 & \quad \hbox{ in $\Omega$}, \\
u =0 & \quad \hbox{ on $\partial \Omega$},
\end{cases}
\end{equation}
 with some additional assumptions listed below; see anyway Sections \ref{sec_gener_result}, \ref{sec_pertub_concav} and \ref{sec_disc_assumpt} for the assumptions that will be considered in our main results Corollaries \ref{corol_gen_exact_concav}-\ref{corol_gen_perturb_concav}. 
 Here we assume:
\begin{itemize}
\item[(f1)]\hypertarget{(f1)} sub $p$-growth: $|f(x,t)| \leq C(t^{p-1}+1)$;
\item[(f2)]\hypertarget{(f2)} Hopf boundary lemma holds: for the sake of simplicity we will assume $f(x,t)>0$ for $t>0$;%
\footnote{More general cases could be treated, for example, $f(x,t) \geq - \zeta(t)$ with $\zeta(0)=0$, $\zeta$ continuous, nondecreasing and one of the following holds: $\zeta(t)=0$ for some $t \in [0, t_0)$ or $\int_0^1 \frac{1}{(\zeta(t)t)^{1/p}}=+\infty$; in particular, one can assume $f(x,t) \geq - C t^{p-1}$.}
\item[(f3)]\hypertarget{(f3)} regularity of solutions, ensured by: $f \in C^{1,\alpha}_{loc}(\Omega \times (0,+\infty))$, and $f(\cdot, t) \in C^{0,\alpha}(\overline{\Omega})$ for $t>0$. 
\end{itemize}
We consider also an invertible transformation $\varphi = \varphi(t)\in C^{2,\alpha}_{loc}((0,+\infty))$, such that
\begin{equation}\label{eq_def_phi_trans}
\tag{$\varphi_1$}
\lim_{t \to 0^+} \varphi'(t)=+\infty, \quad \varphi'' < 0 < \varphi' \, \hbox{ near $0$}, \quad \lim_{t\to 0^+} \frac{\varphi(t)}{\varphi'(t)} = \lim_{t \to 0^+} \frac{\varphi'(t)}{\varphi''(t)} =0, 
\end{equation}
and the following technical assumptions hold
\begin{equation}\label{eq_phi_trans_techn}
\tag{$\varphi_2$}
(\varphi')^{-2} \in C^{1,\alpha}_{loc}((0,+\infty)), \quad \big|\varphi''(t) (\varphi'(t))^{-{p-1}}\big|\leq C\left(1+|t|^{p-1}\right) \; \hbox{ for $t>0$.}
\end{equation}
Notice that $\psi:=\varphi^{-1}: \R \to (0,+\infty)$, but we are not requiring $\varphi$ to be positive, nor well defined in $0$.
In this Section we also require: 
\begin{itemize}
\item $\Omega$ strongly convex and with $\partial \Omega \in C^{2,\alpha}$,
\end{itemize}
and
\begin{itemize}
\item[(fs)]\hypertarget{(fs)}
coercivity and uniqueness, ensured by: $|f(x,t)| \leq C(1+t^{q-1})$ for some $q \in (0,p-1)$, and $t \mapsto \frac{f(x, t)}{t^{p-1}}$ strictly decreasing. 
\end{itemize}
We refer to Remark \ref{rem_weighted_eigenfunctions} for the eigenfunction case $q=p-1$, while to Section \ref{sec_gener_result} for the case of general convex $\Omega$.

\smallskip
The approximation argument is developed in several steps. 
\medskip

\textbf{Step 1.} \hypertarget{Step1}
First we observe that, under our assumptions -- in particular \hyperlink{(fs)}{\rm{(fs)}} -- actually a positive solution of the problem exists by a global minimization process: see e.g. \cite[Theorem 2.1, Section 5.2]{LaUr66}. 
If $p>N$ then clearly $u\in C^{0,1-\frac{N}{p}}(\Omega) \subset L^{\infty}(\Omega)$. Assume thus $p\leq N$ (see also Remark \ref{rem_weighted_eigenfunctions}). 
Set 
$$G(x,t,\xi):=\frac{1}{p}|\xi|^p - F(x,t),$$
$u$ is a minimizer of $v \mapsto \int_{\Omega} G(x,v,\nabla v)$. 
Thanks to \cite[Theorem 3.2, Section 5.3]{LaUr66}, we have that the global minimizer $u$ is in $L^{\infty}(\Omega)$.

\medskip

\textbf{Step 2.} \hypertarget{Step2}
More generally, by \cite[Theorem 2.1]{BPS22} (see also \cite[Propositions 1.2 and 1.3]{GuVe89}), all the solutions belong to $L^{\infty}(\Omega)$, and thus, being $\partial \Omega \in C^{1,\alpha}$, thanks to \cite[Theorem 1]{Lie88} (see also \cite[Corollary 1.1]{GuVe89}), all the solutions are in $C^{1,\beta}(\overline{\Omega})$ for some $\beta = \beta(p, N) \in (0,1)$; thus uniqueness holds by Lemma \ref{lem_uniquen}.
In particular, all the solutions are minima.

\medskip

\textbf{Step 3.} \hypertarget{Step3}
We already observed that $u\in C^{1,\beta}(\overline{\Omega})$. More precisely, we have
$$\norm{u}_{C^{1,\beta}(\overline{\Omega})} \leq C\big(p, N, \norm{u}_{\infty}, \norm{f(\cdot, u)}_{\infty}, \Omega\big).$$
By Hopf Lemma \cite[Theorem 5]{Vaz84} (see also \cite[Lemma A.3]{Sak87}, \cite[Proposition 2.2]{GuVe89}) we have 
$$ \partial_{\nu} u >0 \quad \hbox{on $\partial \Omega$};$$
thus there exists $\eta>0$ such that
$$\inf_{\Omega_{\eta}} u > 0, \quad \inf_{\Omega \setminus \Omega_{\eta}} |\nabla u|>0. $$

We observe that, being $\Omega$ convex, we have $u\in W^{2,2}_{loc}(\Omega \setminus \Omega_{\eta} )$ by \cite{Kad64} (see also \cite[Section 4.3, Theorem 5.2]{LaUr66}). 
Thus, being $\partial \Omega \in C^{2,\alpha}$ and $f(\cdot, t) \in C^{0,\alpha}(\overline{\Omega})$ we obtain thanks to \cite[Section 4.6, Theorem 6.3]{LaUr66} 
$$u \in C^2(\overline{\Omega} \setminus \Omega_{\eta}).$$
We consider moreover $\delta \in (0,\eta)$ sufficiently small to be fixed (see Step \hyperlink{Step7}{\rm{7}}) such that
$$\Omega_{\eta} \subset \Omega_{\delta}\subset \Omega;$$
with smooth boundaries. Similarly, for $k=1...5$, we may assume $\Omega_{\delta/k}$ convex and smooth (see Proposition \ref{prop_boundar_approximat}); notice that we actually can substitute these sets with nicer suitable approximations, if needed.

\medskip

\textbf{Step 4.} \hypertarget{Step4}
Consider the functional
$$J(v):=\frac{1}{p} \int_{\Omega} |\nabla v|^p - \int_{\Omega} F(x,v) $$
to which $u$ is a critical point. We consider a regularization $I_{\eps} \approx J$ defined as follows: let us introduce a function $K=K(t)$ -- to be fixed, see Step \hyperlink{Step8}{\rm{8}} -- such that 
\begin{itemize}
\item $K\geq 0$, and $K>0$ in $(0,+\infty)$,
\item $K \in C^{1}((0,+\infty))$, and $K^{\frac{2}{p}} \in C^{1,\alpha}_{loc}((0,+\infty))$, 
\item $|K'(t)| \lesssim 1+t^{p-1}$ for $t > 0$.
\end{itemize}
Then define, for each $\eps>0$
$$I_{\eps}(v):= \frac{1}{p} \int_{\Omega} \left( \eps K(v)^{\frac{2}{p}} + |\nabla v|^2\right)^{\frac{p}{2}} - \int_{\Omega} F(x,v);$$
clearly $I_0 \equiv J$. 
Similarly to $J$, for each $\eps>0$ we gain the existence of $u_{\eps}\in W^{1,p}(\Omega)$, global minimizer of $I_{\eps}$.

\smallskip
\textbf{Step 4.1.} \hypertarget{Step4.1}
By exploiting the equicoercivity and the fact that $u_{\eps}$ are minima, we obtain for $\eps \in (0,1)$
$$C \norm{u_{\eps}}_{W^{1,p}_0(\Omega)}^p \leq I_{\eps}(u_{\eps}) \leq I_{\eps}(0) \leq I_1(0)$$
thus equibounded, which means that $u_{\eps} \wto \bar{u}$ in $W^{1,p}_0(\Omega)$. 
As a consequence, exploiting that $J$ is lower semicontinuous, $J \leq I_{\eps}$, $u_{\eps}$ are minima, $I_{\eps} \to J$ (by dominated convergence theorem) and $u$ is a minimum, we obtain
$$J(\bar{u}) \leq \liminf_{\eps \to 0} J(u_{\eps}) \leq \liminf_{\eps \to 0} I_{\eps}(u_{\eps}) \leq \liminf_{\eps \to 0} I_{\eps}(u) = J(u) \leq J(\bar{u}).$$
Being $J(u)=J(\bar{u})$ and the minimum unique, we have $\bar{u}=u$ and thus $u_{\eps} \wto u$ in $W^{1,p}_0(\Omega)$.
Moreover we have, by the above computation, $J(u_{\eps}) \to J(u)$ which, together with $\int_{\Omega} F(x,u_{\eps}) \to \int_{\Omega} F(x,u)$ (given by the $p$-subhomogeneous growth), imply $\norm{\nabla u_{\eps}}_p \to \norm{\nabla u}_p$. 
Thus we have 
$$u_{\eps} \to u \quad \hbox{in $W^{1,p}_0(\Omega)$.}$$

\smallskip

If $p>N$ we have $W^{1,p}(\Omega) \hookrightarrow C^{0,1-\frac{N}{p}}(\Omega)$, thus $\norm{u_{\eps}}_{C^{0,1-\frac{N}{p}}(\Omega)}$ is equibounded. 
In the following two steps, hence, we can restrict to $p\leq N$.

\smallskip
\textbf{Step 4.2.} \hypertarget{Step4.2}
Set 
$$G_{\eps}(x,t,\xi):=\frac{1}{p}\left( \eps K(t)^{\frac{2}{p}} + |\xi|^2\right)^{\frac{p}{2}} - F(x,t),$$
$u_{\eps}$ is a minimizer of $v \mapsto \int_{\Omega} G_{\eps}(x,v,\nabla v)$.
Due to the assumptions on $K$, 
by \cite[Theorem 3.2, Section 5.3]{LaUr66} we have that
the minimizers $u_{\eps}$ are in $L^{\infty}(\Omega)$ for $\eps \in (0,1)$. More precisely
$$\norm{u_{\eps}}_{\infty}\leq C\big(\norm{u_{\eps}}_{q}, p, \textup{meas}(\Omega)\big),$$
for $q=p^*$ (if $p<N$) or some $q< \infty$ (if $p=N$) and thus, by Step \hyperlink{Step4.1}{\rm{4.1}}, they are equibounded.
As a consequence 
$$ F(x,u_{\eps}) \leq C_1, \quad K(u_{\eps}) \leq C_2 \quad \hbox{ in $\Omega$}$$
for suitable $C_1, C_2>0$.

\smallskip

\textbf{Step 4.3.} \hypertarget{Step4.3}
By Step \hyperlink{Step4.2}{\rm{4.2}} we have
$$ \frac{1}{p} |\xi|^p - C \leq G_{\eps}(x,t,\xi) \leq C |\xi|^p + C$$ 
for $|t|\leq \sup_{\eps} \norm{u_{\eps}}_{\infty}$, 
which by \cite[Theorem 3.1]{GiGi82} directly implies that the minimizers $u_{\eps}$ are Hölder continuous. 
To obtain a more explicit equibound, we see $u_{\eps}$ as solutions of the equation 
$$-\dive_{\xi}\big((\nabla_{\xi} G_{\eps})(x,u_{\eps}, \nabla u_{\eps})\big) + (\partial_t G_{\eps})(x,u_{\eps}, \nabla u_{\eps})=0 \quad \hbox{in $\Omega$}$$
that is
 $$-\dive\big(A^{\eps}(u_{\eps},\nabla u_{\eps})\big) = f_{\eps}(x,u_{\eps},\nabla u_{\eps}) \quad \hbox{ in $\Omega$}$$
where
$$A^{\eps}(t,\xi):=(\eps K(t)^{\frac{2}{p}} + |\xi|^2)^{\frac{p-2}{2}} \xi,$$
$$f_{\eps}(x,t,\xi):=f(x,t) - \frac{\eps}{p} (\eps + |\xi|^2 K(t)^{-\frac{2}{p}})^{\frac{p-2}{2}} K'(t).$$
We can thus apply \cite[Theorem 4.1, Section 5.4]{LaUr66} (see also \cite[Corollary 4.2]{Tru67}, \cite[Theorem 2.3]{BPS22})
and obtain the existence of a $\beta_0=\beta_0(p, K, \sup_{\eps}\norm{u_{\eps}}_{\infty}) \in (0,1)$ and 
 $C=C(p, K, \sup_{\eps}\norm{u_{\eps}}_{\infty}, \delta)$, such that
$$\norm{u_{\eps}}_{C^{0,\beta_0}(\overline{\Omega_{\delta/5}})}\leq C; $$
we can assume $\beta_0 < \beta$.%
\footnote{Notice that actually $ \norm{u_{\eps}}_{C^{0,\beta_0}(\overline{\Omega})}\leq C'=C'(p, K, \sup_{\eps}\norm{u_{\eps}}_{\infty}, \Omega)$ depending on the geometry of $\Omega$.}

\smallskip

\textbf{Step 4.4.} \hypertarget{Step4.4}
By exploiting the convergence in Step \hyperlink{Step4.1}{\rm{4.1}} and the uniform estimate in Step \hyperlink{Step4.3}{\rm{4.3}} (or at the end of Step \hyperlink{Step4.1}{\rm{4.1}}), we obtain
$$u_{\eps} \to u \quad \hbox{in $C^{0,\beta}(\overline{\Omega_{\delta/5}})$}.$$
By this convergence and Step \hyperlink{Step3}{\rm{3}}, there exists $C=C(\delta)$ such that
$$\tfrac{1}{C} \leq u_{\eps} \leq C \quad \hbox{ in $\Omega_{\delta/4}$}$$
for $\eps$ small. As a consequence 
$$\tfrac{1}{C_1} \leq F(x,u_{\eps}) \leq C_1, \quad \tfrac{1}{C_2} \leq K(u_{\eps}) \leq C_2 \quad \hbox{ in $\Omega_{\delta/4}$}$$
for suitable $C_1, C_2>0$.

\medskip

\textbf{Step 5.} \hypertarget{Step5}
Set $a^{\eps}_{ij}(t,\xi):= \partial_{\xi_j} A^{\eps}_i (t,\xi) $, and assumed $ \partial_t K(u_{\eps}) \leq C_2$, by Step \hyperlink{Step4.3}{\rm{4.3}} and Step \hyperlink{Step4.4}{\rm{4.4}}
we can apply \cite{Dib83, Tol84} (see also \cite{Lie88}, \cite[Section 4.6]{LaUr66}) 
to get the existence of $\beta_2 \in (0, 1)$ such that $u_{\eps} \in C^{1,\beta_2}(\Omega)$ and
$$\norm{u_{\eps}}_{C^{1,\beta_2}(\overline{\Omega_{\delta/3}})} \leq C;$$
here $\beta_2 = \beta_2(\norm{u}_{\infty}, \norm{f(\cdot, u)}_{\infty}, C_2, p)$ and $C =C(\delta, \norm{u}_{\infty}, \norm{f(\cdot, u)}_{\infty}, C_2, p)$.
We can assume $\beta_2 \leq \beta$ and $\beta_2 < \alpha$. 
By Ascoli-Arzelà theorem,
$$u_{\eps} \to u \quad \hbox{in $C^{1,\beta_2}(\overline{\Omega_{\delta/3}})$}.$$
By this convergence and Step \hyperlink{Step3}{\rm{3}}, 
$$\inf_{\Omega \setminus \Omega_{\eta}} |\nabla u_{\eps}|>0, \qquad \sup_{\overline{\Omega_{\delta/3}}} |\nabla u_{\eps}| \leq C'$$
for $\eps>0$ small. 

\medskip

\textbf{Step 6.} \hypertarget{Step6}
By \cite[Section 4.6, Theorem 6.4]{LaUr66} we conclude that there exists $\beta_3 \in (0,1)$, $\beta_3=\beta_3(C_1, C_2, \alpha, \delta)$, such that
$$u_{\eps} \in C^{2,\beta_3}(\Omega_{\delta/3});$$
we can assume $\beta_3 \leq \beta_2$. 
Moreover 
$$\norm{u_{\eps}}_{C^{2,\beta_3}(\overline{\Omega_{\delta/2}}\setminus \Omega_{\delta})} \leq C$$
for some $C= C(C_1, C_2, \alpha, \delta)$. 
From which, by Ascoli-Arzelà theorem, we obtain
$$u_{\eps} \to u \quad \hbox{in $C^1(\overline{\Omega_{\delta/2}}) \cap C^2(\overline{\Omega_{\delta/2}} \setminus \Omega_{\delta})$}.$$

\medskip

\textbf{Step 7.} \hypertarget{Step7}
Consider the transformation $\varphi$ as in \eqref{eq_def_phi_trans}. 

\smallskip

\textbf{Step 7.1.} \hypertarget{Step7.1}
Set $v:= \varphi(u)$, from the estimate in Step \hyperlink{Step4}{\rm{4}} we get
$$\tfrac{1}{C'} \leq v \leq C' \quad \hbox{ in $\Omega_{\delta/4}$},$$
where $C'=C'(\delta)$. 
Moreover, since $u \in C(\overline{\Omega})$, together with $u>0$ in $\Omega$, $u=0$ on $\partial \Omega$ and $\partial_{\nu} u >0$ on $\partial \Omega$, by Corollary \ref{corol_concav_bounda} we obtain that $\mc{C}_v$ cannot attain a maximum (over $ \overline{\Omega} \times \overline{\Omega} \times [0,1]$) on $\partial (\Omega \times \Omega) \times [0,1]$.%
\footnote{Notice that, up to now, $v$ may be unbounded and $\mc{C}_v$ not well defined on the boundary.}

Moreover, since $u \in C^1(\overline{\Omega}) \cap C^2(\overline{\Omega}\setminus \Omega_{\eta})$ by Step \hyperlink{Step3}{\rm{3}}, and applying Proposition \ref{prop_info_bound_c2}, we obtain that, for $\delta$ sufficiently small,
\begin{equation}\label{eq_for_apply_prop}
\begin{cases}
D^2 v <0 &\quad \hbox{in $\Omega \setminus \Omega_{\delta}$},
\\
v(x) - v(x_0) < \nabla v (x_0) \cdot (x-x_0) &\quad \hbox{for each $x_0 \in \Omega \setminus \Omega_{\delta}$, $x \in \Omega\setminus \{x_0\}$}.
\end{cases}
\end{equation}

\smallskip

\textbf{Step 7.2.} \hypertarget{Step7.2}
 Set $v_{\eps}:= \varphi(u_{\eps}).$ 
From the estimates in Step \hyperlink{Step4}{\rm{4}} and Step \hyperlink{Step7.1}{\rm{7.1}} we get
$$\tfrac{1}{C'} \leq v_{\eps} \leq C' \quad \hbox{ in $\Omega_{\delta/4}$}, \qquad |\nabla v_{\eps}| \leq C'' \quad \hbox{ in $\Omega_{\delta/3}$}$$
for some $C'=C'(\delta)>0$ and $C''=C''(\delta)>0$; this, combined with the convergence in Step \hyperlink{Step6}{\rm{6}}, gives 
$$v_{\eps} \to v \quad \hbox{in $C^1(\overline{\Omega_{\delta/2}}) \cap C^2(\overline{\Omega_{\delta/2}} \setminus \Omega_{\delta})$}.$$
Thanks to Proposition \ref{prop_concav_bound_perturb}, $\mc{C}_{v_{\eps}}$ cannot attain its positive maximum (over $ \overline{\Omega_{\delta/2}} \times \overline{\Omega_{\delta/2}} \times [0,1]$) on $\partial (\Omega_{\delta/2} \times \Omega_{\delta/2}) \times [0,1]$, for $\eps$ small.

\medskip

\textbf{Step 8.} \hypertarget{Step8}
Let us consider now the equation satisfied by $v_{\eps}$. Consider $\psi = \varphi^{-1}$, we have $u_{\eps} = \psi(v_{\eps})$.

By assuming \eqref{eq_phi_trans_techn} and \eqref{eq_ipot_aggiunt_phi_eigenf}, we make the following choice
$$K(\psi(t) )\equiv (\psi'(t))^{p},$$
i.e. $K(t) := \frac{1}{(\varphi'(t))^{p}}$. 
With this choice, set 
$$H_{\eps}(\xi):= \left(\eps + |\xi|^2\right)^{\frac{p}{2}}$$
we have
$$-\dive\big( (\nabla H_{\eps})(\nabla v_{\eps})\big) = B_{\eps}(x,v_{\eps}, \nabla v_{\eps})$$
where
\begin{equation}\label{eq_def_Beps}
B_{\eps}(x,t,\xi):= p \frac{f(x,\psi(t))}{(\psi'(t))^{p-1}} + p H_{\eps}(\xi)^{\frac{p-2}{p}} \left((p-1) |\xi|^2 - \eps \right)\frac{\psi''(t)}{\psi'(t)} .
\end{equation}

\begin{remark}[Eigenfunction case]
\label{rem_weighted_eigenfunctions}
Before proceeding to the main proofs, we comment here the case of the $p$-linear growth of $f$, when it has the particular form 
$$f(x,t) = a(x) |t|^{p-2} t.$$
We show here how to adapt the previous steps to this case. 
We further assume
\begin{equation}\label{eq_ipot_aggiunt_phi_eigenf}
\tag{$\varphi s$}
\hbox{if $N\geq 3$ and $p \in (2, N)$, then } \; \big|\varphi''(t) (\varphi'(t))^{-3}\big|\leq C\left(1+|t|^{\frac{N+p}{N-p}}\right) \; \hbox{ for $t>0$,}
\end{equation}
which means, in terms of $K$,
\begin{itemize}
\item if $N\geq 3$ and $p \in (2,N)$, then $\big|K(t)^{-\frac{p-2}{p}} K'(t)\big| \lesssim 1+ t^{\frac{N+p}{N-p}}$ for $t > 0$.
\end{itemize}

\smallskip

\emph{Step \hyperlink{Step1}{1} and \hyperlink{Step4}{4}, existence:} 
both existence results for $J$ and $I_{\eps}$ can be made by considering (recall that $a \in L^{\infty}(\Omega)$)
$$\inf\left\{ J(u) \mid u \in W^{1,p}_0(\Omega), \; \int_{\Omega} a(x) u^p = 1 \right\}, \quad \inf\left\{ I_{\eps}(u) \mid u \in W^{1,p}_0(\Omega), \; \int_{\Omega} a(x) u^p = 1 \right\};$$
we notice that the functionals are coercive on the subspace where the weighted $p$-norm is prescribed. 
Thus we can find Lagrange multipliers $\lambda$ and $\lambda_{\eps}$ and solutions $u$ and $u_{\eps}$ (see e.g. \cite[Theorem 6.3.2]{Ber77}). 

\smallskip

\emph{Step \hyperlink{Step2}{2} and \hyperlink{Step4.1}{4.1}, uniqueness and convergence:} 
we observe that, by compact embeddings, $u_{\eps} \wto \bar{u}$ in $W^{1,p}_0(\Omega)$ implies $u_{\eps} \to \bar{u}$ in $L^p(\Omega)$, thus (being $a$ bounded) $\int_{\Omega} a(x) u_{\eps}^p\to \int_{\Omega} a(x) \bar{u}^p$; in particular, $\int_{\Omega} a(x) \bar{u}^p = 1$. 
To get $u=\bar{u}$ we can thus repeat the same arguments, observing that we have uniqueness up to a scaling (thanks to Lemma \ref{lem_uniquen}) and the noninvariant constraint $\int_{\Omega} a(x) u^p = 1 $.

\smallskip

\emph{Step \hyperlink{Step4.2}{4.2}, equiboundedness for the perturbed problem (when $p\leq N$)}: 
$u_{\eps}$ satisfies the equation
$$-\dive\big(A^{\eps}(u_{\eps},\nabla u_{\eps})\big) = f_{\eps}(x,u_{\eps},\nabla u_{\eps}) \quad \hbox{ in $\Omega$},$$
where $A_{\eps}(t,\xi)=(\eps K(t)^{\frac{2}{p}} + |\xi|^2)^{\frac{p-2}{2}} \xi$ but now
$$f_{\eps}(x,t,\xi):= \lambda_{\eps} a(x) t^{p-1} - \frac{\eps}{p} (\eps + |\xi|^2 K(t)^{-\frac{2}{p}})^{\frac{p-2}{2}} K'(t).$$

First, we observe the equiboundedness of the Lagrange multipliers $\lambda_{\eps}$: indeed, 
for any positive $\varphi \in W^{1,p}_0(\Omega)$, $\lambda_{\eps}$ is given by
$$\lambda_{\eps} = \frac{1}{ \int_{\Omega} a(x) u_{\eps}^{p-1} \varphi} \left( \int_{\Omega} A^{\eps}(u_{\eps}, \nabla u_{\eps}) \cdot \nabla \varphi + \frac{\eps}{p} \int_{\Omega} \left(\eps + |\nabla u_{\eps}|^2 K(u_{\eps})^{-\frac{2}{p}}\right)^{\frac{p-2}{2}} K'(u_{\eps}) \varphi \right). $$
Since $u_{\eps} \to u$ in $W^{1,p}_0(\Omega)$ and almost everywhere, by exploiting that $u_{\eps}$ and $\nabla u_{\eps}$ are dominated by $L^p$ functions, the Young inequality for products and the assumptions on $K$, we obtain that the right-hand side is bounded. 

Next, by the assumptions on $K$, we can apply \cite[Theorem 7.1, Section 4.7]{LaUr66} (see also the proof therein and \cite[Theorem 5.1, Section 2.5]{LaUr66} for the case $p=N$) and get $u_{\eps} \in L^{\infty}(\Omega)$ with
$$\norm{u_{\eps}}_{\infty} \leq C$$
where $C=C(N,p, |\Omega|, \norm{u_{\eps}}_{q})$, $q=p^*$ if $p<N$ and $q<\infty$ if $p=N$, and thus equibounded.
\end{remark}

\subsection{Proof of general results}
\label{sec_gener_result}

We show now that the conclusions of Theorems \ref{thm_main_conc_exact} and \ref{thm_approx_concav_harmonic} hold in more general cases; 
we will furnish the statements of general results for the sake of completeness, even if they appear a bit cumbersome.
We refer to Corollaries \ref{corol_gen_exact_concav} and \ref{corol_gen_perturb_concav} for statements with more explicit assumptions, and to Section \ref{sec_applic} for applications to specific nonlinearities, some of which mentioned in the introduction.

Let us consider the following abstract assumptions on $B_{\eps}$, introduced in \eqref{eq_def_Beps}:
\begin{equation}\label{eq_cond_monoton}
\tag{$B_1$}
t \mapsto B_{\eps}(x,t,\xi) \quad \hbox{nonincreasing}, 
\end{equation}
\begin{equation}\label{eq_cond_strict_monoton}
\tag{$B_2$}
\partial_t B_{\eps} \leq - \mu \quad \hbox{for some $\mu>0$ (uniform in $\eps>0$)},
\end{equation}
\begin{equation}\label{eq_cond_posit}
\tag{$B_3$}
B_{\eps}> 0, 
\end{equation}
\begin{equation}\label{eq_cond_concav}
\tag0{$B_4$}
(x,t) \mapsto B_{\eps}(x,t,\xi) \quad \hbox{jointly harmonic concave}.
\end{equation}
The domain of reference is given by
\begin{equation}\label{eq_domain_refernc_Beps}
(x,t,\xi) \in \Omega_{\delta/2} \times \in v_{\eps} (\Omega_{\delta/2}) \times \nabla v_{\eps}(\Omega_{\delta/2})
\end{equation}
 where $t \in v_{\eps} (\Omega_{\delta/2}) \subset [\tfrac{1}{C'},C'] $ while $|\xi| \in |\nabla v_{\eps}|(\Omega_{\delta/2}) \subset [0, C'']$ 
(for $\eps$ small) thanks to Step \hyperlink{Step7.2}{\rm{7.2}}. The main information given by $C'$ and $C''$ is that $\xi$ is bounded and $t$ is bounded and far from zero, while $x$ is far from the boundary.

For the sake of clearness, at some point (see Sections \ref{sec_pertub_concav} and \ref{sec_disc_assumpt}) we will require a particular form of $f$, that is 
\begin{equation}\label{eq_particular_f_choice}
\tag{$F_1$}
f(x,t) \equiv h(x,t) g(t) + k(x,t) 
\end{equation}
where, set $G(t):=\int_0^t g(\tau) d\tau$, we require
\begin{equation}\label{eq_ipotesi_aggiuntiva_G>0}
\tag{$F_2$}
G(t)>0 \quad \hbox{ for $t>0$}
\end{equation}
and%
\footnote{Notice that, in the regular case $p=2$, that is $\eps=0$, we can allow $h=0$.}
\begin{equation}\label{eq_ipotesi aggiuntiva_h>0}
\tag{$F_3$}
h(x,t) > 0 \quad \hbox{for $t>0$, $x \in \Omega$};
\end{equation}
notice that, differently from $f$, we do not require $h$ to be finite in $t=0$. 
Notice anyway that the arguments may be adapted to other kind of nonlinearities.%
\footnote{Rather than \eqref{eq_particular_f_choice}, other possible choices could be set up according to the specific problem. For instance, $f(x,t)=\zeta(g(t))$ for some suitable function $\zeta$.}
In this setting we will require $g$ and $\psi$ to be related by 
\begin{equation}\label{eq_cond_on_g}
\tag{$F_4$}
\frac{g(\psi(t))}{(\psi'(t))^{p-1}} = \frac{\psi''(t)}{\psi'(t)}
 \end{equation}
that is (up to additive constants and positive multiplicative constants) $\varphi(t) = \int_1^t (G(s))^{-\frac{1}{p}} ds$.

Thus $B_{\eps}$ takes the form
\begin{equation}\label{eq_charact_Beps}
B_{\eps}(x,t,\xi)= p \frac{\psi''(t)}{\psi'(t)} \left( h(x,\psi(t)) + H_{\eps}(\xi)^{\frac{p-2}{p}} \left((p-1) |\xi|^2 - \eps \right) \right) + p \frac{k(x,\psi(t))}{(\psi'(t))^{p-1}}.
\end{equation}
In Section \ref{sec_disc_assumpt} we will give some sufficient conditions on $\psi$, $k$, and $h$ to ensure \eqref{eq_cond_monoton}, \eqref{eq_cond_strict_monoton}, \eqref{eq_cond_posit} and \eqref{eq_cond_concav}.

\subsubsection{Exact concavity.} 


We can state now the main result about exact concavity. See also Corollary \ref{corol_gen_exact_concav} below.

\begin{theorem}[Exact concavity]
\label{thm_gen_exact_concav}
Let $\Omega \subset \R^N$, $N\geq 2$, be open, bounded, convex, with $\partial \Omega \in C^{1,\alpha}$,
and let $p\in (1,+\infty)$.
Let $u$ be a solution of \eqref{eq_general_fxu}, and let $f$ and $\varphi$ satisfy \hyperlink{(f1)}{\rm{(f1)}}--\hyperlink{(f3)}{\rm{(f3)}}, \eqref{eq_def_phi_trans}-\eqref{eq_phi_trans_techn}, and one among \hyperlink{(fs)}{\rm{(fs)}} or $f(x,t) = a(x) |t|^{p-2} t$ and \eqref{eq_ipot_aggiunt_phi_eigenf}.
Assume moreover that \eqref{eq_cond_monoton} and \eqref{eq_cond_concav} hold. 
Then $\varphi(u)$ is concave.
\end{theorem}

\begin{proof}
Assume first $\Omega$ smooth and strongly convex.
We start observing that, exploiting the boundedness of $|\nabla u_{\eps}|$ by Step \hyperlink{Step5}{\rm{5}}, the equation verified by $v_{\eps}$
$$-\sum_{i,j}\tilde{a}^{\eps}_{ij}(\nabla v_{\eps}) \partial_{ij} v_{\eps} = B_{\eps}(x,v_{\eps},\nabla v_{\eps}),$$
where $\tilde{a}^{\eps}_{ij}(\xi) := p (\eps + |\xi|^2)^{\frac{p-4}{2}} \left( (p-2) \xi_i \xi_j + (\eps + |\xi|^2) \delta_{ij}\right)$, is uniformly elliptic.

We check the assumptions of Theorem \ref{thm_concav_inter_class}: by \eqref{eq_cond_monoton}, \eqref{eq_cond_concav} and the regularity assumed on $f$ (see also \eqref{eq_cond_monoton_strict} for the strict monotonicity), together with $v_{\eps} \in C^2(\Omega_{\delta/2})$ given by Step \hyperlink{Step6}{\rm{6}}, we see that Theorem \ref{thm_concav_inter_class} applies and thus $\mc{C}_{v_{\eps}} $ cannot assume a maximum in $\Omega_{\delta/2} \times \Omega_{\delta/2} \times [0,1]$. 
Combined with Step \hyperlink{Step7.2}{\rm{7.2}}, we see that $\mc{C}_{v_{\eps}}\in C(\overline{\Omega_{\delta/2}} \times \overline{\Omega_{\delta/2}} \times [0,1])$ cannot have a positive maximum, which means that $\mc{C}_{v_{\eps}} \leq 0$ on $\Omega_{\delta/2} \times \Omega_{\delta/2} \times [0,1]$. 
By the (pointwise) convergence obtained in Step \hyperlink{Step4}{\rm{4}}, we obtain that $\mc{C}_{v} \leq 0$ on $\Omega_{\delta/2} \times \Omega_{\delta/2} \times [0,1]$.
This is on the other hand true for any $\delta>0$ (small): this means that $\mc{C}_{v} \leq 0$ on $\Omega \times \Omega \times [0,1]$, that is $\varphi(u)$ is concave.

To pass from a smooth, strongly convex domain to a more general convex $\Omega$ (with $\partial \Omega \in C^{1,\alpha}$), we make an approximation process -- based on the uniqueness and minimality of the solution -- as in \cite[Section 4.1]{BMS22} (see also \cite[Section 5]{Sak87}); we give some details in the case of the eigenfunction.
First, being $\Omega$ convex, from \cite[Corollary 4.2]{Tru67} (see also \cite[Corollary 3.7]{CaDe04})
all the solutions belong to $C(\overline{\Omega})$ and thus, being $\partial \Omega \in C^{1,\alpha}$, thanks to \cite[Theorem 1]{Lie88} all the solutions are in $C^{1,\beta}(\overline{\Omega})$; thus uniqueness -- up to scaling -- holds by Lemma \ref{lem_uniquen}.
Let now $u \in W^{1,p}_0(\Omega) \cap C(\overline{\Omega})$ be the solution on $\Omega$ -- obtained by minimization with constraint $\int_{\Omega} a(x) u^p = 1$ -- and let $\Omega^k \subset \Omega $ be a sequence of smooth strongly convex sets which approximate $\Omega$ in the Hausdorff distance (see Proposition \ref{prop_boundar_approximat}), and define $J_k(v):= \frac{1}{p} \int_{\Omega^k} |\nabla v|_p^p - \int_{\Omega^k} a(x) v^p$. 
We consider $u_k \in W^{1,p}_0(\Omega^k)$ critical points of $J_k$ obtained by constrained minimization $\int_{\Omega^k} a(x) u_k^p = 1$, and extend $u_k$ to $\Omega$ by $u_{k}=0$ on $\Omega \setminus \Omega^k$. 

We clearly have $J_k(u_k) = J(u_{k})$ and thus
$$\frac{1}{p}\norm{u_k}_{W^{1,p}_0(\Omega)}^p = J(u_k) +1= J_k(u_k) +1\leq J_k(u_1) +1\leq J_1(u_1)+1 = J(u_1)+1$$
which implies $u_k \wto \bar{u}$ in $W^{1,p}_0(\Omega)$ and strongly in $L^p(\Omega)$; in particular $\int_{\Omega} a(x) \bar{u}^p =1$. 

Let now $\eps_k:= 2 \sup_{\Omega \setminus \Omega^k} u$, and define $v_k:= (\int_{\Omega^k} a(x) (u-\eps_k)_+)^{-1} (u-\eps_k)_+$. 
Notice that $\eps_k \to 0$, since $\max_{\overline{\Omega} \setminus \Omega^k} u = u(x_k)$ with $d(x_k, \partial \Omega) \leq \sup_{x \in \Omega} d(x, \Omega_k) 
 \to 0$. 
Since $v_k \to u $ in $W^{1,p}_0(\Omega)$ we obtain
$$J(\bar{u}) \leq \liminf_{k} J(u_k) = \liminf_{k} J_k(u_k) \leq \liminf_k J_k(v_k) = J(u) \leq J(\bar{u})$$
from which, by uniqueness of the minimizer, $\bar{u}=u$. As a consequence, since the concavity claim holds for each $u_k$, then by pointwise convergence it holds also for $u$.
\end{proof}

\subsubsection{Perturbed concavity} 
\label{sec_pertub_concav}

We pass now to the statement and proof of a general result about perturbed concavity; see also Corollary \ref{corol_gen_perturb_concav} below.
We assume here 
\begin{equation}\label{eq_cond_phi_boun}
\tag{$\varphi_3$}
\varphi \in C(\overline{\Omega});
\end{equation}
 see Remark \ref{rem_varphi_unbounded} for more general $\varphi$ unbounded, and Section \ref{sec_eigen_phi_unbound} for the logarithmic case. 
For the sake of simplicity, we also assume \eqref{eq_particular_f_choice}, \eqref{eq_ipotesi aggiuntiva_h>0}, \eqref{eq_cond_on_g} and
\begin{equation}\label{eq_cond_sempl_perturb}
\tag{$F_5$}
k(x,t) \equiv 0\quad \hbox{and} \quad h(x,t) \equiv h(x)
\end{equation}
and set, for each $\delta>0$, 
\begin{equation}\label{eq_def_mfrak}
\mathfrak{m}_{\delta}:= \min_{x \in \overline{\Omega_{\delta/2}}} h(x)>0, \quad \mathfrak{M}_{\delta} := \max_{x \in \overline{\Omega_{\delta/2}}} h(x)>0,
\end{equation}
in light of \eqref{eq_ipotesi aggiuntiva_h>0}.
We also assume 
\begin{equation}\label{eq_cond_harm_concave}
\tag{$F_6$}
t \mapsto \frac{\psi''(t)}{\psi'(t)} \quad \hbox{ harmonic concave}.
\end{equation}

\begin{theorem}[Perturbed concavity]
\label{thm_gen_perturb_concav}
Let $\Omega \subset \R^N$, $N\geq 2$, be bounded, strongly convex, with $\partial \Omega \in C^{2,\alpha}$, and let $p\in (1,+\infty)$.
Let $u$ be a solution of \eqref{eq_general_fxu}, and let $f$ and $\varphi$ satisfy \hyperlink{(f1)}{\rm{(f1)}}--\hyperlink{(f3)}{\rm{(f3)}}, \eqref{eq_def_phi_trans}-\eqref{eq_phi_trans_techn}, \eqref{eq_particular_f_choice}, \eqref{eq_ipotesi aggiuntiva_h>0}, \eqref{eq_cond_on_g} and \hyperlink{(fs)}{\rm{(fs)}}, together with \eqref{eq_cond_phi_boun}, \eqref{eq_cond_sempl_perturb} and \eqref{eq_cond_harm_concave}. 
Assume moreover that \eqref{eq_cond_strict_monoton} holds. 
Then $v=\varphi(u)$ satisfies 
\begin{equation}\label{eq_general_error_estimate}
\mc{C}_{v} \leq \frac{1}{\mu} \frac{p}{m_{\delta}} \bigg( 2 + \frac{\mathfrak{M}_{\delta} - \mathfrak{m}_{\delta}}{\mathfrak{m}_{\delta}}\bigg) (\mathfrak{M}_{\delta}-\mathfrak{m}_{\delta}) \quad \hbox{ in $ \overline{\Omega} \times \overline{\Omega} \times [0,1]$}
\end{equation}
for some $\delta>0$, where
$$m_{\delta} := \min_{x \in \overline{\Omega_{\delta/2}} } \frac{\psi'(v(x))}{\psi''(v(x))}>0.$$
\end{theorem}

\begin{proof} 
By Step \hyperlink{Step7.1}{\rm{7.1}} we have that the maximum $(x_0, y_0, \lambda_0)$ of $\mc{C}_v$ cannot be attained on the boundary, thus it belongs to $\Omega \times \Omega \times [0,1]$; in particular, there exists $\delta>0$ sufficiently small such that $(x_0,y_0,\lambda_0) \in \Omega_{\delta/2} \times \Omega_{\delta/2} \times [0,1]$. Clearly, $(x_0,y_0,\lambda_0)$ is also the maximum point over $ \overline{\Omega_{\delta/2}} \times \overline{\Omega_{\delta/2}} \times [0,1]$.

Let now $(x_{\eps},y_{\eps},\lambda_{\eps}) \in \overline{\Omega_{\delta/2}} \times \overline{\Omega_{\delta/2}} \times [0,1]$ be the point of maximum of $\mc{C}_{v_{\eps}}$; by Step \hyperlink{Step7.2}{\rm{7.2}}, for $\eps>0$ small, it cannot be on the boundary, thus $(x_{\eps},y_{\eps},\lambda_{\eps}) \in \Omega_{\delta/2} \times \Omega_{\delta/2} \times [0,1]$. 
Moreover, by the convergence in Step \hyperlink{Step7.2}{\rm{7.2}}, we have $(x_{\eps},y_{\eps},\lambda_{\eps}) \to (x_0,y_0,\lambda_0)$; notice that $(x_{\eps}, y_{\eps}, \lambda_{\eps})$, differently from $(x_0,y_0,\lambda_0)$, depends also on $\delta$.
Then, by Theorem \ref{thm_approx_concav_bucsquas} 
$$\nabla v_{\eps}(x_{\eps}) = \nabla v_{\eps} (z_{\eps}) = \nabla v_{\eps} (y_{\eps}) =: \xi_{\eps},$$
where $z_{\eps}:=\lambda_{\eps} x_{\eps} + (1-\lambda_{\eps}) y_{\eps}$. 
By \eqref{eq_cond_strict_monoton} in particular we have 
$$\partial_t B_{\eps}(z_{\eps},t,\xi_{\eps}) \leq - \mu <0 \quad \hbox{for $t \in [v_{\eps}(z_{\eps}), \lambda_{\eps} v_{\eps}(x_{\eps}) + (1-\lambda_{\eps})v_{\eps}(y_{\eps})]$}. $$
Moreover
$$\mc{C}_{v_{\eps}}(x_{\eps},y_{\eps},\lambda_{\eps}) \leq \frac{1}{\mu} \mc{HC}_{B_{\eps}(\cdot, v_{\eps}(\cdot), \xi_{\eps})}(x_{\eps},y_{\eps},\lambda_{\eps}). $$
Being \eqref{eq_cond_strict_monoton} uniform in $\eps$ we can pass to the limit and obtain
$$\mc{C}_{v}(x_0, y_0, \lambda_0) \leq \frac{1}{\mu} \mc{HC}_{B_0(\cdot, v(\cdot), \xi_0)}(x_0, y_0, \lambda_0).$$
We need to estimate $\mc{HC}_{B_0(\cdot, v(\cdot), \xi_0)}(x_0, y_0, \lambda_0)$. 
We observe
$$B_0(x,t,\xi)= p \frac{\psi''(t)}{\psi'(t)} \big( h(x,\psi(t)) + (p-1) |\xi|^p \big) + p \frac{k(x,\psi(t))}{(\psi'(t))^{p-1}}.$$
In light of \eqref{eq_cond_sempl_perturb}%
\footnote{Otherwise we can, respectively, estimate the $\mc{HC}_{B_0(\cdot, v(\cdot), \xi_0)}(x,y,\lambda)$ by exploiting Proposition \ref{prop_diff_harm_concav}, and make some error estimates for $h(x,t)$, or we factorize $h(x,t)$ as $c(x) h_1(x,t)$, where $c$ has small oscillations and $h_1$ has some joint concavity property.} 
we have
$$B_0(x,t,\xi)= \frac{ b_{\xi}(x)}{\rho(t)} $$
where
 $$\rho(t):= \frac{\psi'(t)}{\psi''(t)} \quad b_{\xi}(x) := p \big( h(x) + (p-1) |\xi|^p \big);$$
here $\rho$ is convex and positive for $t>0$, and $b$ is locally bounded and positive.
Namely (we omit the subscripts) we have 
 \begin{align*}
\MoveEqLeft 
\mc{HC}_{B_0(\cdot, v(\cdot), \xi_0)}(x,y,\lambda) \\
\equiv & 
\frac{\frac{b(x)}{\rho(v(x)) }\frac{b(y)}{\rho(v(y))}}{\lambda \frac{b(y)}{ \rho(v(y))} + (1-\lambda)\frac{b(x)}{ \rho(v(x))}} - \frac{b(\lambda x + (1-\lambda)y)}{ \rho(\lambda v(x) + (1-\lambda) v(y))} \\
=& \frac{1}{ \rho(\lambda v(x) + (1-\lambda) v(y))} \Big( b(x) b(y) \frac{ \rho(\lambda v(x) + (1-\lambda) v(y))}{ \lambda b(y) \rho(v(x)) + (1-\lambda) b(x) \rho(v(y))} - b(\lambda x + (1-\lambda)y)\Big) \\
\leq & \frac{1}{ \rho(\lambda v(x) + (1-\lambda) v(y))} \left( \frac{\mathfrak{M}_{\delta,\xi}^2}{\mathfrak{m}_{\delta,\xi}} \frac{ \rho(\lambda v(x) + (1-\lambda) v(y))}{ \lambda \rho(v(x)) + (1-\lambda) \rho(v(y))} - \mathfrak{m}_{\xi,\delta} \right) \\
\leq & \frac{1}{ \rho(\lambda v(x) + (1-\lambda) v(y))} \left( \frac{\mathfrak{M}_{\delta,\xi}^2-\mathfrak{m}_{\delta,\xi}^2}{\mathfrak{m}_{\delta,\xi}}\right)
\end{align*}
where 
$$\mathfrak{m}_{\delta,\xi}:= \min_{x \in \overline{\Omega_{\delta/2}}} b_{\xi}(x) \quad \mathfrak{M}_{\delta,\xi} := \max_{x \in \overline{\Omega_{\delta/2}}} b_{\xi}(x) .$$
By definition of $m_{\delta}>0$ we have (see also Step \hyperlink{Step7}{\rm{7}}) 
$$\mc{HC}_{B_0(\cdot, v(\cdot), \xi_0)}(x_0, y_0, \lambda_0) \leq \frac{1}{m_{\delta}} \bigg( 2 + \frac{\mathfrak{M}_{\delta,\xi} - \mathfrak{m}_{\delta,\xi}}{\mathfrak{m}_{\delta,\xi}}\bigg) (\mathfrak{M}_{\delta,\xi}-\mathfrak{m}_{\delta,\xi});$$
therefore
$$\mc{C}_{v}(x_0, y_0, \lambda_0) \leq \frac{1}{\mu} \frac{p}{m_{\delta}} \bigg( 2 + \frac{\mathfrak{M}_{\delta} - \mathfrak{m}_{\delta}}{\mathfrak{m}_{\delta}}\bigg) (\mathfrak{M}_{\delta}-\mathfrak{m}_{\delta})$$
(notice indeed that $\mathfrak{M}_{\delta,\xi} - \mathfrak{m}_{\delta,\xi} = \mathfrak{M}_{\delta} - \mathfrak{m}_{\delta}$ and that $\mathfrak{m}_{\delta, \xi} \geq \mathfrak{m}_{\delta}$).
This means that \eqref{eq_general_error_estimate} holds.
\end{proof}

\begin{remark}
Consider \eqref{eq_general_error_estimate}. 
If we additionally assume $h \in L^{\infty}(\Omega)$ and $\inf_{\Omega} h >0$, then the maximum and minimum of $h$ on $\Omega_{\delta}$ can be estimated with the ones in $\Omega$, that we call respectively $\mathfrak{M}$ and $\mathfrak{m}$, leading to
$$\mc{C}_{v} \leq \frac{1}{\mu} \frac{p}{m_{\delta}} \bigg( 2 + \frac{\mathfrak{M}- \mathfrak{m}}{\mathfrak{m}}\bigg) (\mathfrak{M}-\mathfrak{m}) \quad \hbox{ in $ \overline{\Omega} \times \overline{\Omega} \times [0,1]$};$$
notice that $\mathfrak{M}-\mathfrak{m} = \textup{osc}(h)$.
Moreover, if $h \in W^{1,\infty}(\Omega)$ then
$$\mc{C}_{v} \leq \frac{1}{\mu} \frac{p}{m_{\delta}} \bigg( 2 + \frac{\diam(\Omega) \norm{\nabla h}_{\infty} }{\mathfrak{m}}\bigg) \diam(\Omega) \norm{\nabla h}_{\infty} \quad \hbox{ in $ \overline{\Omega} \times \overline{\Omega} \times [0,1]$}. $$
\end{remark}

\begin{remark}[Unbounded transformations]
\label{rem_varphi_unbounded}
Consider the case $\varphi$ negatively unbounded in the origin (e.g. the logarithm), i.e.
$$\lim_{t \to 0} \varphi(t)=-\infty.$$ 
Assume $\Omega$ strictly convex and with smooth boundary.
By \cite[Lemma 3.2]{GrPo93} we have that the mid-concavity function $\mc{C}^m_v$ is nonpositive near the boundary (in the sense of \eqref{eq_distante_bordo}), thus it is bounded from above. 
If $\mc{C}^m_v \leq 0$ we are done, otherwise it admits a (positive) maximum point $(x_0,y_0) \in \Omega \times \Omega$. Then we argue as in the general case (essentially with $\lambda_0=\frac{1}{2}$, by adapting Theorem \ref{thm_approx_concav_bucsquas} to $\mc{C}^m_v$) and obtain
$$\mc{C}^m_{v} \leq \frac{1}{\mu} \frac{p}{m_{\delta}} \bigg( 2 + \frac{\mathfrak{M}_{\delta} - \mathfrak{m}_{\delta}}{\mathfrak{m}_{\delta}}\bigg) (\mathfrak{M}_{\delta}-\mathfrak{m}_{\delta}) \quad \hbox{ in $ \overline{\Omega} \times \overline{\Omega}$}.$$
By \cite[Corollary 1]{NgNi93} we can estimate the concavity function in terms of the mid-concavity function up to a factor $2$, that is
$$\mc{C}_{v} \leq \frac{2}{\mu} \frac{p}{m_{\delta}} \bigg( 2 + \frac{\mathfrak{M}_{\delta} - \mathfrak{m}_{\delta}}{\mathfrak{m}_{\delta}}\bigg) (\mathfrak{M}_{\delta}-\mathfrak{m}_{\delta}) \quad \hbox{ in $ \overline{\Omega} \times \overline{\Omega} \times [0,1]$}.$$
To this relation we can apply Hyers-Ulam Theorem \cite{HyUl52} (as in Remark \ref{rem_ulam_thm}).

We observe that this argument cannot be applied to the case of eigenfunctions (and logarithmic transformation): indeed, in this case $\mu=0$ and the above estimate cannot be set. See Section \ref{sec_eigen_phi_unbound} for a different approach. 
On the other hand, different problems could be treated in this way, such as showing that the solutions of suitable perturbations of the eigenfunction equation, e.g.
$$f(x,u) = a(x) u^{p-1} + \sigma u^q,$$
with $q \in [0, p-1)$ and $\sigma$ small, are almost log-concave; see e.g. \cite[Proposition 3.4]{ABCS23}. We leave the details to the interested reader.
Notice that it is natural to consider, in the choice of the transformation, the biggest power, which in this case is $p-1$ (see also Corollary \ref{corol_sum_powers} and \cite{Lio81}).
\end{remark}

\subsubsection{Discussions on the assumptions and corollaries}
\label{sec_disc_assumpt}

To discuss the properties of $B_{\eps}$ in \eqref{eq_def_Beps}, we notice that $ (p-1) |\xi|^2 - \eps $ has a variable sign when $\eps \neq 0$. 
We thus simplify the argument by requiring \eqref{eq_particular_f_choice} together with \eqref{eq_ipotesi_aggiuntiva_G>0}, \eqref{eq_ipotesi aggiuntiva_h>0} and \eqref{eq_cond_on_g}, which imply that $B_{\eps}$ takes the form \eqref{eq_charact_Beps}.

Recalled the domain of reference given in \eqref{eq_domain_refernc_Beps}, for these values of $(x,t,\xi)$ we clearly have
$$h(x, \psi(t)) + H_{\eps}(\xi)^{\frac{p-2}{p}} \left((p-1) |\xi|^2 - \eps\right) >0\quad \hbox{for $\eps$ small}.$$
Considered
\begin{equation}\label{eq_assump_Theta}
\tag{$A_1$}
0< \Theta < \inf_{\Omega_{\delta/2} \times \in [\tfrac{1}{C'},C'] } h(x,\psi(t))
 \end{equation}
$\Theta$ to be fixed (independent on $\eps$), we have also
$$ \Theta + H_{\eps}(\xi)^{\frac{p-2}{p}} \left((p-1) |\xi|^2 - \eps\right) >0 \quad \hbox{for $\eps$ small}.$$
We restrict to such $\eps$, and rewrite $B_{\eps}$ as
 $$B_{\eps}(x,t,\xi)=p \frac{\psi''(t)}{\psi'(t)} \left( \Theta + H_{\eps}(\xi)^{\frac{p-2}{p}} \left((p-1) |\xi|^2 - \eps \right) \right) + p \big(h(x,\psi(t))-\Theta\big) \frac{\psi''(t)}{\psi'(t)} + p \frac{k(x,\psi(t))}{(\psi'(t))^{p-1}}.$$
 
\medskip

$\bullet$ \emph{Monotonicity}. 
If we assume
\begin{equation}\label{eq_ipot_psi_monot}
\tag{$A_2$}
 t \mapsto \frac{\psi''(t)}{\psi'(t)}, \quad t \mapsto \frac{\psi''(t)}{\psi'(t)} \big(h(x,\psi(t))-\Theta \big) \quad \hbox{and} \quad t \mapsto \frac{k(x,\psi(t))}{(\psi'(t))^{p-1}}\quad \hbox{ nonincreasing},
\end{equation}
then \eqref{eq_cond_monoton} holds; 
if one of the three above is strictly decreasing, then
\begin{equation}\label{eq_cond_monoton_strict}
t \mapsto B_{\eps}(x,t,\xi) \quad \hbox{strictly decreasing}. 
\end{equation}

If moreover
\begin{equation}\label{eq_ipot_psi_mustrett}
\tag{$A_3$}
\partial_t \left (\frac{\psi''}{\psi'}\right)(t) \leq - \tilde{\mu}<0, 
\end{equation}
for some $\tilde{\mu}>0$, then
there exists $\mu=\mu(\delta)>0$ such that \eqref{eq_cond_strict_monoton} holds.

\medskip

$\bullet$ \emph{Positivity.} If we assume
\begin{equation}\label{eq_ipot_psi_positiv}
\tag{$A_4$}
\frac{\psi''(t)}{\psi'(t)} \geq 0 , \quad \frac{k(x,\psi(t))}{(\psi'(t))^{p-1}}\geq 0, \quad \frac{\psi''(t)}{\psi'(t)} +\frac{k(x,\psi(t))}{(\psi'(t))^{p-1}}> 0
\end{equation}
then \eqref{eq_cond_posit} holds.

\medskip

$\bullet$ \emph{Concavity.} 
The harmonic concavity request in \eqref{eq_cond_concav} is, in general, difficult to check, due to the sum of different terms (but in some particular cases it can be checked only through the harmonic concavity of $ t \mapsto \frac{\psi''(t)}{\psi'(t)} $, see Case 2 of the proof Theorem \ref{thm_main_conc_exact}, Remark \ref{rem_BMS_assumptions} below and Remark \ref{rem_comp_singular} for some comments). 
To discuss the harmonic concavity of the sum, we rely on Proposition \ref{prop_propriet_concav_varie}. 
Indeed, if we assume
\begin{equation}\label{eq_ipot_psi_concav}
\tag{$A_5$}
 t \mapsto t^2 \frac{\psi''(t)}{\psi'(t)}, \quad (x,t) \mapsto t^2 \frac{\psi''(t)}{\psi'(t)} \big(h(x,\psi(t))-\Theta \big) \quad \hbox{and} \quad (x,t) \mapsto t^2 \frac{k(x,\psi(t))}{(\psi'(t))^{p-1}} \quad \hbox{jointly concave},
\end{equation}
then
\begin{equation}\label{eq_strong_cond_jointconcav}
(x,t) \mapsto t^2 B_{\eps}(x,t,\xi) \quad \hbox{jointly concave}.
\end{equation}
If we assume $B_{\eps}$ positive (see \eqref{eq_ipot_psi_positiv}), by Proposition \ref{prop_propriet_concav_varie} we see that \eqref{eq_strong_cond_jointconcav} is stronger than \eqref{eq_cond_concav}.

\smallskip

As a consequence of the above discussion, we obtain the following corollaries of Theorems \ref{thm_gen_exact_concav} and \ref{thm_gen_perturb_concav}.

\begin{corollary}[Exact concavity] 
\label{corol_gen_exact_concav}
Let $\Omega \subset \R^N$, $N\geq 2$, be open, bounded, convex, with $\partial \Omega \in C^{1,\alpha}$,
and let $p\in (1,+\infty)$.
Let $u$ be a solution of \eqref{eq_general_fxu}, and let $f$ and $\varphi$ satisfy \hyperlink{(f1)}{\rm{(f1)}}--\hyperlink{(f3)}{\rm{(f3)}}, \eqref{eq_def_phi_trans}-\eqref{eq_phi_trans_techn},
\eqref{eq_particular_f_choice}, \eqref{eq_ipotesi_aggiuntiva_G>0}, \eqref{eq_ipotesi aggiuntiva_h>0} \eqref{eq_cond_on_g}, one among \hyperlink{(fs)}{\rm{(fs)}} or $f(x,t) = a(x) |t|^{p-2} t$ and \eqref{eq_ipot_aggiunt_phi_eigenf}.
Assume moreover that \eqref{eq_assump_Theta} and \eqref{eq_ipot_psi_monot}, together with \eqref{eq_cond_concav} (implied e.g. by \eqref{eq_ipot_psi_positiv}-\eqref{eq_ipot_psi_concav}). 
Then $\varphi(u)$ is concave.
\end{corollary}

\begin{corollary}[Perturbed concavity]
\label{corol_gen_perturb_concav}
Let $\Omega \subset \R^N$, $N\geq 2$, be bounded, strongly convex, with $\partial \Omega \in C^{2,\alpha}$, and let $p\in (1,+\infty)$.
Let $u$ be a solution of \eqref{eq_general_fxu}, and let $f$ and $\varphi$ satisfy 
\hyperlink{(f1)}{\rm{(f1)}}--\hyperlink{(f3)}{\rm{(f3)}},
\eqref{eq_def_phi_trans}, \eqref{eq_phi_trans_techn}, 
\eqref{eq_particular_f_choice}, \eqref{eq_ipotesi aggiuntiva_h>0}, \eqref{eq_cond_on_g}, \hyperlink{(fs)}{\rm{(fs)}}, together with \eqref{eq_cond_phi_boun}, \eqref{eq_cond_sempl_perturb} and \eqref{eq_cond_harm_concave}. 
Assume moreover that \eqref{eq_assump_Theta}, \eqref{eq_ipot_psi_monot} and \eqref{eq_ipot_psi_mustrett} hold. 
Then $v=\varphi(u)$ satisfies 
$$\mc{C}_{v} \leq \frac{1}{\mu} \frac{p}{m_{\delta}} \bigg( 2 + \frac{\mathfrak{M}_{\delta} - \mathfrak{m}_{\delta}}{\mathfrak{m}_{\delta}}\bigg) (\mathfrak{M}_{\delta}-\mathfrak{m}_{\delta}) \quad \hbox{ in $ \overline{\Omega} \times \overline{\Omega} \times [0,1]$}$$
for some $\delta>0$, where
$m_{\delta} = \min_{x \in \overline{\Omega_{\delta/2}} } \frac{\psi'(v(x))}{\psi''(v(x))}>0$ and $\mathfrak{m}_{\delta}, \mathfrak{M}_{\delta}$ as in \eqref{eq_def_mfrak}.
\end{corollary}

We observe that conditions \eqref{eq_ipot_psi_monot}, \eqref{eq_ipot_psi_mustrett}, \eqref{eq_ipot_psi_positiv} and \eqref{eq_ipot_psi_concav} depend on $f$ in an implicit way (with respect to $g$); stating general explicit assumptions on $f$ which imply these conditions might be a bit cumbersome; in turn, we prefer to present various examples, including the one in Theorem \ref{thm_main_conc_exact}, the ones in Remark \ref{rem_BMS_assumptions} below, and the one in Corollary \ref{corol_sum_powers}.

\begin{remark}
\label{rem_BMS_assumptions}
We observe that, if $h \equiv 1$ and $k\equiv 0$, that is
$$f(x,t)\equiv g(t),$$
(we avoid the presence of an $x$-dependent weight to keep the presentation simple) then \eqref{eq_ipot_psi_monot}, \eqref{eq_ipot_psi_positiv} and \eqref{eq_cond_concav} collapse to
$$ t \mapsto \frac{\psi''(t)}{\psi'(t)} \quad \hbox{nonincreasing and harmonic concave},$$
thus monotonicity \eqref{eq_cond_monoton} and harmonic concavity \eqref{eq_cond_concav} of $B_{\eps}$ are ensured by this request.
A sufficient condition to fulfill this requirement can be expressed in terms of $g$ (recall indeed the relation \eqref{eq_cond_on_g}), and it is given in \cite{BMS22} (see also \cite[Theorem 4.4]{CaFr85}, \cite[Corollary 2]{Kaw85W}) which reads as follows:
\begin{itemize}
\item[(i)] $G^{1/p}$ concave,
\item[(ii)] $\frac{g}{G}$ harmonic concave. 
\end{itemize}
See \cite[Section 4.5]{BMS22} for details. This ensures that $\varphi(t) = \int_1^t (G(s))^{-\frac{1}{p}} ds$ is concave. 
Class of functions covered by this choice (but not by the assumptions in \cite{Ken85}) are, for example, $g(t)=\log(1+t)$ (allowing non-power nonlinearities), $g(t)=1-t$ (allowing linear perturbation of the Laplacian); see also Section \ref{sec_singular} for another interesting application to singular equations.
We will see that the possibility of $k(x,t)$ in \eqref{eq_particular_f_choice} of not being zero allows more general cases than \cite{BMS22}, even in the autonomous semilinear framework (see Corollary \ref{corol_sum_powers}).
\end{remark}

\section{Applications}
\label{sec_applic}

As we saw in Corollaries \ref{corol_gen_exact_concav} and \ref{corol_gen_perturb_concav}, the machinery of the previous Sections allows to treat the case of a general $g$ (in the spirit of \cite{BMS22, MRS24}, see Remark \ref{rem_BMS_assumptions} and Corollary \ref{corol_sum_powers} below); we focus here our attention on the power case, that is
$$f(x,t) = a(x) t^{q}$$
with $q \in (0,p-1]$, and $a(x) > 0$, which clearly satisfies \hyperlink{(f1)}{\rm{(f1)}}--\hyperlink{(f3)}{\rm{(f3)}} and one among \hyperlink{(fs)}{\rm{(fs)}} and $f(x,t) = a(x) t^{p-1}$. 

We are now ready to prove the main theorems.

\begin{proof}[Proof of Theorem \ref{thm_main_conc_exact}] 
We write ($k\equiv 0$)
\begin{equation}\label{eq_singular_decompos}
f(x,t) = \big(a(x) t^{-q_1}\big) t^{q_2} \equiv h(x,t) g(t) 
\end{equation}
with $ q= q_2 -q_1 $, $q_1 \in [0, p-1-q]$ and $q_2 \in [q,p-1]$. 
Thus we have \eqref{eq_particular_f_choice}--\eqref{eq_cond_on_g}.

\textbf{Case 1:} $q_2 \neq p-1$. By \eqref{eq_cond_on_g} $\varphi$ can be chosen as
\begin{equation}\label{eq_camb_var_pot}
\varphi(t)= \zeta t^{\gamma}
\end{equation}
with $\gamma := \frac{p-1-q_2}{p} \in (0,\frac{p-1}{p}) \subset (0,1)$ and $\zeta := \zeta(p,q_2)= \left(\frac{q_2+1}{p}\right)^{1/p} \frac{p}{p-1-q_2}$; for exact concavity the constant $\zeta>0$ plays no role, while for the perturbed concavity it is involved in the coefficients of the estimate. 
By \eqref{eq_camb_var_pot} we have $\psi(t)= \frac{1}{\zeta^{1/\gamma}} t^{1/\gamma}$ and
$$\frac{\psi''(t)}{\psi'(t)} = \frac{1-\gamma}{\gamma} \frac{1}{t}$$
which is positive, decreasing, harmonic concave (thus \eqref{eq_cond_harm_concave} holds), and such that $t^2\frac{\psi''(t)}{\psi'(t)} = \frac{1-\gamma}{\gamma} t$ is concave. 
Moreover
$$ \frac{\psi''(t)}{\psi'(t)} \big(h(x,\psi(t))-\Theta \big) = \frac{1-\gamma}{\gamma} \frac{1}{t} \bigg(\frac{1}{\zeta^{-q_1/\gamma}} a(x) t^{-q_1/\gamma} - \Theta \bigg)$$
where we recall (see \eqref{eq_assump_Theta}) that $\Theta$ is chosen in such a way $\frac{1}{\zeta^{-q_1/\gamma}} a(x) t^{-q_1/\gamma} - \Theta$ is far from zero. 
In particular the above quantity is always positive. 
With no loss of generality we discuss monotonicity and concavity of
$$\kappa(x,t):= \frac{1}{t} \big( a(x) t^{-q_1/\gamma} -\tilde{\Theta} \big)$$
where $\tilde{\Theta} := \zeta^{-q_1/\gamma} \Theta$.

\medskip

\emph{Monotonicity:} to meet the condition \eqref{eq_ipot_psi_monot} for \eqref{eq_cond_monoton} we compute
$$\partial_t \kappa(x,t)= \frac{1}{t^2} \Big(- \Big(\frac{q_1}{\gamma}+1\Big) a(x) t^{-q_1/\gamma-1} + \tilde{\Theta}\Big);$$
by assuming $q_1 < \gamma$ and $\tilde{\Theta}$ sufficiently small (depending also on $q_1, q_2, \gamma$), we have the claim (both \eqref{eq_cond_monoton} and \eqref{eq_cond_strict_monoton}). 
Namely, choosing $\tilde{\Theta} < \frac{1}{2} \big(\frac{q_1}{\gamma}+1\big) (\inf_{\Omega_{\delta/2}} a) \norm{v}_{\infty}^{-q_1/\gamma-1}$ (and consequently $\eps$ small), we have
$$\partial_t \kappa(x,t) \leq -\frac{1}{2} \Big(\frac{q_1}{\gamma}+1\Big) \Big(\inf_{\Omega_{\delta/2}} a\Big) \norm{v}_{\infty}^{-q_1/\gamma-1}$$
and thus
$$\partial_t B_{\eps}(x,t,\xi) \leq - \frac{p}{2} \frac{1-\gamma}{\gamma} \frac{1}{\zeta} \Big(\frac{q_1}{\gamma}+1\Big) \Big(\inf_{\Omega_{\delta/2}} a\Big) \norm{u}_{\infty}^{-q_1-\gamma}=: - \mu .$$
Hence we actually obtain \eqref{eq_ipot_psi_mustrett}, which implies the stronger \eqref{eq_cond_strict_monoton}.

\medskip

\emph{Concavity:} to meet the condition \eqref{eq_ipot_psi_concav} for \eqref{eq_cond_concav} we discuss the joint concavity of
$$(x,t) \mapsto t^2 \kappa(t) = a(x) t^{1-q_1/\gamma} - \tilde{\Theta} t ;$$
clearly it is sufficient to discuss the joint concavity of 
$$(x,t) \mapsto a(x) t^{1-q_1/\gamma}.$$
We restrict to $q_1 \in [0, \gamma]$ and exploit that $x\mapsto a(x)$ is $\theta$-concave, for some $\theta \in [1,+\infty]$.

Consider first $q_1 \in (0, \gamma)$: we observe that $t \mapsto t^{1-q_1/\gamma}$ is $\omega = \frac{\gamma}{\gamma-q_1} \in (0, +\infty)$ concave. 
By Proposition \ref{prop_propriet_concav_varie} we have $(x,t) \mapsto a(x) t^{1-q_1/\gamma}$ is jointly $(\frac{1}{\theta} + \frac{1}{\omega})^{-1}$-concave.
By imposing $(\frac{1}{\theta} + \frac{1}{\omega})^{-1}=1$ we get $\gamma = \frac{\theta(p-1-q)}{1+ \theta p}$ and $q_1 = \frac{p-1-q}{1+\theta p}$, $ q_2 = \frac{p-1+ \theta p q}{1+\theta p}$; in particular, by the restriction $q_1 < \gamma$, we have $\theta>1$.
If now $q_1=\gamma$, then $\omega=\infty$ and the product function is concave if $\theta =1$; in this case $\gamma = q_1= \frac{p-1-q}{1+ p} ,$ $q_2 = \frac{p-1+pq}{1+p}$.
If instead $q_1=0$, then $\omega=1$ and the product function is concave if $\theta=\infty$ (i.e. $a$ is constant); in this case $\gamma = \frac{p-1-q}{p}$, $q_2=q$.
In all the above relations we have that $\gamma>0$ forces $p < q-1$.
Under these assumptions we have \eqref{eq_cond_concav}.

\smallskip

\textbf{Case 2:} $q_2 = p-1$, (possibly $p<q-1$), i.e. $q_1 = p-1-q$. Then $\varphi(t) = \log(t)$ (up to constants), and $\psi(t) = e^t$, thus $\frac{\psi''(t)}{\psi'(t)} = 1$ which is positive and nonincreasing, but $t^2 \frac{\psi''(t)}{\psi'(t)} = t^2$ is not concave.
Moreover
$$\frac{\psi''(t)}{\psi'(t)} \big(h(x,\psi(t))-\Theta \big) = a(x) e^{-q_1 t} - \Theta.$$
The function is nonincreasing (thus we have \eqref{eq_cond_monoton}), but never harmonic concave 
if $q_1 \neq 0$. 
Thus we assume $q_1=0$ (that is, $q_2=q=p-1$) and hence $\frac{\psi''(t)}{\psi'(t)} \big(h(x,\psi(t))-\Theta \big) = a(x) - \Theta$ which is (jointly, harmonic) concave if $a(x)$ is so.
Again, under these assumptions we have \eqref{eq_cond_concav}.

\smallskip

We conclude by applying Corollary \ref{corol_gen_exact_concav}. 
\end{proof}

We see that the statements in Example \ref{exam_hardy_henon} follow by Theorem \ref{thm_main_conc_exact} and straightforward computations. Focusing on one of these results, we propose now another application of Corollary \ref{corol_confron_3}, where $a_2$ is nonconstant; similar statements hold for the other examples.

\begin{corollary}[Hardy-Hénon type equation]
\label{corol_hard_hen_perturb}
Let $\Omega$ open, bounded and convex be a subset of the hyperoctant $\{ x\in \R^N \mid x_i > 0 \; \hbox{for each $i$}\}$, $N\geq 2$, and let $p\in (1,+\infty)$.
Consider $\omega \in [0,1]$ and the positive solutions $u$ and $v$ of 
$$\begin{cases}
-\Delta_p u = |x|_2^{\omega} & \quad \hbox{ in $\Omega$},\\ 
u =0 & \quad \hbox{ on $\partial \Omega$},
\end{cases}
\qquad
\begin{cases}
-\Delta_p v = |x|_1^{\omega} & \quad \hbox{ in $\Omega$},\\ 
v =0 & \quad \hbox{ on $\partial \Omega$}.
\end{cases}$$
Then $v^{\frac{p-1}{\omega+ p}}$ is concave (strictly, if $p=2$) and
$$\norm{u^{\frac{p-1}{\omega+ p}}-v^{\frac{p-1}{\omega+ p}}}_{\infty} \leq C \left((N^{\omega}-1) \norm{\max\{|x|_2^{\omega}, |x|_1^{\omega} \}}_{L^{\infty}(\Omega)}\right)
^{\kappa_{q}\frac{1}{\omega+ p}}$$
if $p\geq 2$, $C=C(p, \Omega, \norm{|x|_1^{\omega}}_{\infty},\norm{|x|_2^{\omega}}_{\infty})>0$, $q=p^*$ if $2 \leq p<N$ and $q<\infty$ if $p\geq N \geq 2$, while
$$\norm{u^{\frac{p-1}{\omega+ p}}-v^{\frac{p-1}{\omega+ p}}}_{\infty} \leq C \left((N^{\omega}-1) \norm{\max\{|x|_2^{\omega}, |x|_1^{\omega} \}}_{L^{\infty}(\Omega)}\right)
^{\kappa_{q}
\frac{p-1}{\omega+ p}}$$
if $p\leq 2$ and $\partial \Omega \in C^{1,\alpha}$, $C=C(p, \Omega, \alpha, \norm{|x|_1^{\omega}}_{\infty}, \norm{|x|_2^{\omega}}_{\infty})>0$, $q=2^*$ if $N\geq 3$ and $q <\infty$ if $N= 2$.
Here $|x|_1$ is the $1$-norm in $\R^N$ and $|x|=|x|_2$ is the Euclidean norm.
In particular the difference goes to zero as $\omega \to 0$. 
\end{corollary}

\begin{proof}
By the equivalence of the norms in $\R^N$ we know that $|x|_1 \leq N |x|_2 \leq N^{3/2} |x|_1$, thus
$$\norm{|x|_2^{\omega} - |x|_1^{\omega}}_{L^{\infty}(\Omega)} \leq (N^{\omega}-1) \norm{\max\{|x|_2^{\omega}, |x|_1^{\omega} \}}_{L^{\infty}(\Omega)}.$$
The claim follows from Example \ref{exam_hardy_henon} and Corollary \ref{corol_confron_3}.
\end{proof}

\begin{remark}
We notice that, for each $\omega \in (0,1]$, highlighting the dependence on $\omega$, we have that $v_{\omega}^{\frac{p-1}{\omega+p}}$ is concave, and moreover $\norm{v_{\omega}^{\frac{p-1}{\omega+p}}}_{\infty} \leq C$. 
Thus by \cite[Theorem 10.9]{Roc97}, up to a subsequence, we have $v_{\omega}^{\frac{p-1}{\omega+p}} \to \bar{v}$ as $\omega \to 0$ in $L^{\infty}_{loc}(\Omega)$, for some $\bar{v}$ concave: by Corollary \ref{corol_hard_hen_perturb} we thus obtain that, up to a subsequence, $u_{\omega}^{\frac{p-1}{\omega+ p}}$ converge to a concave function as $\omega \to 0$.
\end{remark}

\begin{proof}[Proof of Theorem \ref{thm_approx_concav_harmonic}]
The conclusion follows from the same arguments of proof of Theorem \ref{thm_main_conc_exact} and applying Corollary \ref{corol_gen_perturb_concav}. 
In particular, by choosing $q_1=0$ we obtain ($v=u^{\gamma}$)
$$\mc{C}_{u^{\gamma}}(x_0, y_0, \lambda_0) \leq p \frac{1}{\mu_{a,u}} \frac{1}{m_{\delta,u}} \left( 2 + \frac{\osc(a)}{\mathfrak{m}_{\delta}}\right) \osc(a)$$
where
$$m_{\delta, u} := \frac{\gamma}{1-\gamma} \min_{\overline{\Omega_{\delta/2}} } u(x),
\qquad \mathfrak{m}_{\delta}:= \min_{x \in \overline{\Omega_{\delta/2}}}a(x), 
\quad \mu_{a,u,\delta}:= \frac{p}{2} \frac{1-\gamma}{\gamma} \frac{1}{\zeta} \mathfrak{m}_{\delta} \norm{u}_{\infty}^{-\gamma},$$
with $\gamma = \frac{p-1-q}{p}$, $\zeta = \left(\frac{q+1}{p}\right)^{1/p} \frac{p}{p-1-q}$.
That is
$$\mc{C}_{u^{\gamma}}(x_0, y_0, \lambda_0) \leq \zeta \bigg( \frac{\norm{u}_{\infty}}{\min_{\overline{\Omega_{\delta/2}} } u(x)} \bigg)^{\gamma} \bigg( 2 + \frac{\osc(a)}{\mathfrak{m}_{\delta}}\bigg) \frac{\osc(a)}{\mathfrak{m}_{\delta}}.
\vspace{-2em}
$$
\end{proof}

\begin{proof}[Proof of Corollary \ref{cor_conv_oscul}]
First we observe that
$$\norm{a_n - a_{\infty}}_{\infty} = O(\osc(a_n)).$$
Moreover we have, by the equiboundedness of $\norm{a_n}_{\infty}$ and regularity theory \cite[Theorem 2.1]{BPS22}, $\norm{u_n}_{\infty} \leq C$ and, by Corollary \ref{corol_stima_basso}, $\min_{\overline{\Omega_{\delta/2}}} u_n \geq C > 0$.
The claim thus follows by Theorem \ref{thm_approx_concav_harmonic}.
\end{proof}

As a final result in this Section, we show an application in the case $k\nequiv 0$ (see \eqref{eq_particular_f_choice}), which allows to include some cases covered by \cite{Ken85}, but not by \cite{BMS22} (see Remark \ref{rem_BMS_assumptions}), of the type $f(x,t) \equiv g(t) + k(t)$ (we avoid the presence of $x$-dependent weights for the sake of exposition).
For simplicity we consider the sum of powers.

\begin{corollary}
\label{corol_sum_powers}
Let $\Omega \subset \R^N$, $N\geq 2$, be open, bounded, convex, and let $p\in (1,+\infty)$.
Let $q,r \in (0, p-1)$ and $u$ be a positive solution of 
$$\begin{cases}
-\Delta_p u = u^q + u^r & \quad \hbox{ in $\Omega$},\\ 
u =0 & \quad \hbox{ on $\partial \Omega$},
\end{cases}$$
with $r \leq q$ and 
\begin{equation}\label{eq_cond_sum_power}
(p+1)q - p r \leq p-1.
\end{equation}
Then $u^{\frac{p-1-q}{p}}$ is concave.
\end{corollary} 

\begin{proof}
We clearly have \hyperlink{(f1)}{\rm{(f1)}}--\hyperlink{(f3)}{\rm{(f3)}} and \hyperlink{(fs)}{\rm{(fs)}}.
Set $g(t)=t^q$ and $k(t)=t^r$, we have \eqref{eq_particular_f_choice}--\eqref{eq_cond_on_g} and (up to constants) $\varphi(t) = t^{\frac{p-q-1}{p}}$. 
We see that $t \mapsto \frac{k(\psi(t))}{(\psi'(t))^{p-1}} = t^{\frac{pr-(p-1)q-(p-1)}{p-q-1}}$ is nonincreasing and $t \mapsto t^2 \frac{k(\psi(t))}{(\psi'(t))^{p-1}}=t^{\frac{pr-(p+1)q+(p-1)}{p-q-1}}$ is concave by the assumptions; thus we conclude.
\end{proof}

\begin{remark}
Notice that for $p=2$, \eqref{eq_cond_sum_power} gives $3q - 2r \leq 1$, which actually is the condition to impose in order to have
\begin{equation}\label{eq_cond_extra_kenn}
 t \mapsto t^{\gamma} f(t) \quad \hbox{strictly decreasing}, \quad t \mapsto t^{\frac{3\gamma-1}{\gamma}} f(t^{1/\gamma}) \quad \hbox{concave},
 \end{equation}
$\gamma=\frac{1-q}{2}$, requested in \cite[Theorem 3.3]{Ken85}.
Additionally, we highlight that, in Corollary \ref{corol_sum_powers}, it is the biggest exponent who leads the choice of the transformation function. 
Moreover, it remains open to determinate if the restrictions on $q,r$ in Corollary \ref{corol_sum_powers} are sharp.
\end{remark}

\subsection{Weighted eigenfunctions}
\label{sec_eigen_phi_unbound}

We consider 
$$f(x,t) = a(x) |t|^{p-2} t.$$
The approximation argument has been commented in Remark \ref{rem_weighted_eigenfunctions}, while exact concavity has been considered in Case 2 of the proof of Theorem \ref{thm_main_conc_exact}. 
We give now some insights on perturbed concavity.

When $p=2$ and $\varphi(t)=\log(t)$ the equation solved by $v=\varphi(u)$ is given by
$$-\Delta v = |\nabla v|^2 + 1;$$
thus even in the semilinear case we see that the main issue is given by the fact that the nonlinearity of the transformed equation does not have a derivative in $t$ far from zero.
This problem has been tackled in different ways: in \cite[Proposition 3.4]{ABCS23} the authors study a perturbed equation, with an additional term which allows to gain the desired assumption on the nonlinearity (see Remark \ref{rem_varphi_unbounded}).
In \cite[Proposition 2.8]{BuSq20} they consider instead an approximation process, by staying far from the boundary; this approach would require cumbersome computations in order to be implemented in the nonregular quasilinear setting, since a double approximation process should be set up.
Another possibility is to pass to the parabolic equation, see \cite{GMrS25}.

\medskip

Here we propose a different approach, by estimating the concavity function of the transformation 
$$\varphi_{\sigma}(t) :=\log(u-\sigma)$$
in the set where it is well defined, i.e. $\{u>\sigma\}$; notice that, since we cannot ensure the quasiconcavity of $u$, this set is generally not a priori convex. 
Near the boundary we have instead an information on $\log(u)$.
For the sake of simplicity we present the argument for $p=2$, but it can be easily adapted to any $p \in (1,+\infty)$.

\begin{theorem}
\label{thm_weigh_eigen_sigma}
Let $\Omega \subset \R^N$, $N\geq 2$, be open, bounded, strictly convex, with $\partial \Omega \in C^{2,\alpha}$. 
Then there exists $\sigma_0>0$ such that, for each $\sigma \in (0,\sigma_0]$, set 
$$\Lambda_{\sigma} :=\{u> \sigma\}$$
we have%
\footnote{Notice that, arguing as in Remark \ref{rem_varphi_unbounded}, we have also 
$$\mc{C}_{\log(u)}(x,y,\lambda) \leq 2\max_{\Lambda_{\sigma}\times \Lambda_{\sigma}} \mc{C}_{\log(u)} \quad \hbox{for $(x, y,\lambda) \in \big(\Omega \times \Omega) \setminus\big(\Lambda_{\sigma} \times \Lambda_{\sigma}\big) \times [0,1].
$}$$}
$$\mc{C}_{\log(u)}^m(x,y) \leq \max_{\Lambda_{\sigma}\times \Lambda_{\sigma}} \mc{C}^m_{\log(u)} \quad \hbox{for $(x, y) \in \big(\Omega \times \Omega) \setminus\big(\Lambda_{\sigma} \times \Lambda_{\sigma}\big)$}$$
and
$$\mc{C}_{\log(u-\sigma)}(x,y,\lambda) (x,y) \leq \frac{1}{\sigma} 
C \, \osc(a) $$
for any $(x,y) \in \Lambda_{\sigma}\times \Lambda_{\sigma}$ with $[x,y] \subset \Lambda_{\sigma}$ and $\lambda \in [0,1]$, where 
$C=C(a,u,\delta) := e^{\norm{u}_{\infty}} \bigg( 2 + \frac{\osc(a)}{\mathfrak{m}_{\delta}}\bigg) \frac{1}{\mathfrak{m}_{\delta}}$.
\end{theorem}

\begin{proof}
By \cite[Lemma 3.2]{GrPo93} (see also \eqref{eq_distante_bordo} and \cite{Mor24}) 
$\mc{C}^m_{\log(u)}$ does not attain a maximum near the boundary $\partial (\Omega \times \Omega)$. Thus there exists $\Omega_{\delta_0}$ such that
$$\mc{C}_{\log(u)}^m(x,y) \leq \max_{\Omega_{\delta_0} \times \Omega_{\delta_0}} \mc{C}_{\log(u)}^m \quad \hbox{for $(x, y) \in \big(\Omega \times \Omega) \setminus\big(\Omega_{\delta_0}\times \Omega_{\delta_0}\big)$.}$$
For $\sigma$ sufficiently small we have $\Omega_{\delta_0} \subset \Lambda_{\sigma}$.
Consider the transformation $\varphi_{\sigma}: (\sigma, +\infty) \to \R$, 
$\varphi_{\sigma}(t):=\log(t-\sigma),$ 
with inverse $\psi_{\sigma}: \R\to (\sigma, +\infty)$ given by $\psi_{\sigma}(t)=e^t+\sigma$. 
Set $v_{\sigma}:= \varphi_{\sigma}(u)$.
The function $v_{\sigma}$ satisfies
$$-\Delta v_{\sigma} = \frac{\psi_{\sigma}''(v_{\sigma})}{\psi_{\sigma}'(v_{\sigma})} |\nabla v_{\sigma}|^2 + \frac{f(x,\psi_{\sigma}(v_{\sigma}))}{\psi_{\sigma}'(v_{\sigma})}$$
where $\frac{\psi_{\sigma}''(t)}{\psi_{\sigma}'(t)} \equiv 1$, which is positive and nonincreasing, and
$$ \frac{f(x,\psi_{\sigma}(t))}{\psi_{\sigma}'(t)} = a(x) \frac{e^{t} +\sigma}{e^{t}} = a(x) \left(1 +\sigma e^{-t}\right)$$
which is positive and decreasing in $t$. 
In particular $\partial_t \Big(\frac{f(x,\psi_{\sigma}(t))}{\psi'_{\sigma}(t)}\Big) = - \sigma a(x) e^{-t}$ and thus
$$\partial_t \bigg(\frac{f(x,\psi_{\sigma}(t))}{\psi'_{\sigma}(t)}\bigg) \leq - \sigma e^{-\norm{u}_{\infty}} \mathfrak{m}_{\delta} =: - \mu_{\sigma} < 0$$
which implies, by Theorem \ref{thm_approx_concav_bucsquas}, 
$$\mc{C}_{\log(u-\sigma)}(x,y,\lambda)(x,y,\lambda) \leq \frac{1}{\mu_{\sigma}} \bigg( 2 + \frac{\osc(a)}{\mathfrak{m}_{\delta}}\bigg) \osc(a)$$
for each $[x,y] \subset \{u > \sigma\}.$
\end{proof}

\subsection{Singular equations}
\label{sec_singular}

We consider now singular equations, that is $q<0$.%
\footnote{We thank Lorenzo Brasco for having brought to our attention this problem.}
To deal with this case, we introduce a regularized problem (with respect to the source), apply previous results and then pass to the limit.
We start by recalling the following results in \cite[Theorems 1.3 and 1.5]{CST16} (see therein the definition of solution) and \cite[Corollary 2.4]{CaDe04}; see also \cite[Teorema 2.3]{Gla02} for the case $a(x) \equiv 1$ and \cite{BoOr10} for the case $p=2$.

\begin{theorem}[\cite{CST16, CaDe04}]
\label{thm_exist_singular}
Let $\Omega \subset \R^N$, $N\geq 2$, be open, bounded, with $\partial \Omega \in C^{2,\alpha}$, and let $p \in (1,+\infty)$.
Assume $q \in (-\infty, 0)$ and $a \in L^{\infty}(\Omega)$. 
Then there exists a unique solution $u$ of \eqref{eq_general_problem*}. 
If $q \in [-1,0)$ then $u \in W^{1,p}_0(\Omega)$, while if $q \in (-\infty, -1)$ then $u^{\frac{p-1-q}{p}} \in W^{1,p}_0(\Omega)$. 
Moreover, if $a$ is constant, then $u \in C(\overline{\Omega})$.
\end{theorem}

We focus on the case $a(x)\equiv 1$ and $q \in [-1, 0)$.
Consider, for any $\eta>0$, the regularized equation
$$
\begin{cases}
-\Delta_p u_{\eta} = (u_{\eta} + \eta )^q & \quad \hbox{ in $\Omega$},\\ 
u_{\eta} >0 & \quad \hbox{ in $\Omega$}, \\
u_{\eta} =0 & \quad \hbox{ on $\partial \Omega$},
\end{cases}
$$
so that we have
$$f(x,t) \equiv g(t)=(t+\eta)^q.$$
In this case, it is hard to deduce concavity properties of powers of $u_{\eta}$, which makes the approach due to \cite{Kor83Co, Ken85} difficult to apply. 
We will base the proof thus on the abstract approach developed here and in \cite{BMS22}.
We obtain the following result.
\begin{theorem}
Let $\Omega \subset \R^N$, $N\geq 2$, be open, bounded, convex, with $\partial \Omega \in C^{2,\alpha}$, and let $p\in (1,+\infty)$. Let $u \in W^{1,p}_0(\Omega)\cap C(\overline{\Omega})$ be the solution of \eqref{eq_general_problem*} with $q \in [-1,0)$ and $a(x)\equiv 1$. Then $u^{\frac{p-1-q}{p}}$ is concave.
\end{theorem}

\begin{proof}
By \cite[Lemma 4.5 and proof of Theorem 4.6]{CST16} we have that 
$$u_{\eta} \to u \quad \hbox{in $W^{1,p}_0(\Omega)$;}$$
in particular almost everywhere and $u_{\eta}$ is dominated. 

Assume first $q \in (-1,0)$.
It is straightforward to verify that $g(t)=(t+\eta)^q$ satisfies the assumptions of Corollary \ref{corol_gen_exact_concav} (see in particular \eqref{eq_def_phi_trans}-\eqref{eq_phi_trans_techn} and Remark \ref{rem_BMS_assumptions} (i)-(ii)); in particular, since
$$G(t)=\frac{(t+\eta)^{q+1}-\eta^{q+1}}{q+1}$$
being $q \in (-1, 0)$, we have that $G^{1/p}$ is concave and $\frac{g}{G}$ is harmonic concave (since $\frac{G}{g}$ is convex). 
Therefore, for each $\eta>0$ we have $\varphi_{\eta}(u_{\eta})$ concave, where
$$\varphi_{\eta}(t) := \int_1^t \left((s+\eta)^{q+1}-\eta^{q+1}\right)^{-1/p} ds;$$
we notice that $\varphi_{\eta}$ cannot be expressed in terms of elementary functions. 
By $u_{\eta} \to u$ (notice that $\left((s+\eta)^{q+1}-\eta^{q+1}\right)^{-1/p} \leq s^{-\frac{q+1}{p}}$) we have (up to constants)
\begin{equation}\label{eq_conv_qo_sing}
\varphi_{\eta}(u_{\eta}) \to u^{\frac{p-1-q}{p}} \quad \hbox{almost everywhere}.
\end{equation}
Thus, being $u$ continuous, we obtain that $u^{\frac{p-1-q}{p}}$ is concave. 

Consider now $q= -1$; we have
$$G(t) = \log(t+\eta) - \log(\eta)$$
and again $G^{1/p}$ is concave and $\frac{g}{G}$ is harmonic concave. 
Therefore, for each $\eta>0$ we have $\varphi_{\eta}(u_{\eta})$ concave, where
$$\varphi_{\eta}(t) := \int_1^t \left(\log(s+\eta)-\log(\eta)\right)^{-1/p} ds; $$
again $\varphi_{\eta}$ does not have an elementary expression.
Being (observe that $\frac{\log(\frac{1}{\eta})}{\log(s+\eta)-\log(\eta)} \leq C_t $ for $s \in (t, +\infty)$)
$$\big(\log(\tfrac{1}{\eta})\big)^{\frac{1}{p}} \varphi_{\eta}(t) \sim t - 1 \quad \hbox{as $\eta \to 0$}$$
we have (up to constants)
$$\big(\log(\tfrac{1}{\eta})\big)^{\frac{1}{p}} \varphi_{\eta}(u_{\eta}) \to u $$
and hence $u $ is concave (notice $\frac{p-1-q}{p}=1$ when $q=-1$).
\end{proof}

\begin{remark}
Consider \eqref{eq_conv_qo_sing}. Being $\varphi_{\eta}(u_{\eta})$ concave on the bounded set $\Omega$, converging almost everywhere pointwise to $u^{\frac{p-1-q}{p}}$, then a known result \cite[Theorem 10.8]{Roc97} automatically implies that $\varphi_{\eta}(u_{\eta}) \to u^{\frac{p-1-q}{p}}$ in $L^{\infty}_{loc}(\Omega)$. 
We highlight that one may use this result to directly obtain the concavity of $u^{\frac{p-1-q}{p}}$ and gaining the continuity of $u$ in $\Omega$ \emph{a posteriori}, furnishing an alternative proof to \cite{CaDe04}.
\end{remark}

\begin{remark}
\label{rem_comp_singular}
(i) The case $q < -1$ was treated in the semilinear framework $p=2$, in a different way by \cite{BGP99}, by dealing with the boundary through a direct use of the equation itself and a maximum principle on the function $x \mapsto |\nabla u|^2 + \frac{2}{q+1} u^{q+1}$ (see \cite[Lemmas 3.1 and 3.3]{BGP99}). 
In our setting, the range $q < -1$ creates more difficulties for the applicability of the concavity results; in particular $t \mapsto t^{\frac{p-1-q}{p}}$ is no more concave with singular derivative.

We notice that also regularity issues arise when $q < -1$ (see Theorem \ref{thm_exist_singular} and \cite{CES19, Gua21}). 
In particular, the solution does not generally belong to $W^{1,p}_0(\Omega)$, but $u^{\frac{p-1-q}{p}}$ does, where the power is exact the same of the expected concavity.

(ii) Consider now $a(x)$ nonconstant. To apply the arguments of Section \ref{sec_gener_result}, the very first thing to check is if $t \mapsto t^2 \frac{\psi''(t)}{\psi'(t)}$ is concave, which is stronger (see also \cite[Proposition B.3]{GMrS25}) and more difficult than showing $t \mapsto \frac{\psi'(t)}{\psi''(t)}$ convex (as done in \cite{BMS22}, see also Remark \ref{rem_BMS_assumptions}); both leads to the harmonic concavity of $ \frac{\psi''(t)}{\psi'(t)}$, but in our case we need the above property to deal with sums (see \eqref{eq_ipot_psi_concav}). 
We have
$$t^2 \frac{\psi''(t)}{\psi'(t)} = \left( \varphi^2 G^{1/p} \tfrac{g}{G}\right)(\psi(t))$$
which seems not easy to handle (recall that, in our case, $\varphi$ and $\psi$ are not explicit);
some numerical simulations for $q=-1+\frac{1}{n}$ suggest indeed that the above function is not always concave.
\end{remark}

\section{Further results}
\label{sec_further_results}

In this Section we furnish several results in different frameworks: the techniques mainly rely on perturbation arguments, based on fine estimates.

\subsection{Superhomogeneous equations}
\label{sec_superhomog}

In the superhomogeneous case $f(x,t)= g(t)=t^q$ with $q>1$, quasiconcavity was already conjectured by Sacks (see \cite[Remark 9]{Kaw85W}). 
As for the semilinear case \cite{Lin94,LeVa08}, we expect the solutions to be $\frac{p-1-q}{p}$-concave (i.e. $u^{-\frac{q-p+1}{p}}$ is convex) in $\Omega$ for any $q \in (p-1, p^*-1)$; notice that this is weaker than the $\log$-concavity. 
We further remark that this power of $u$ is related to the definition of \emph{pressure} in the porous medium equation \cite{LeVa08} and it has a relevant role in the classification of solutions of the Sobolev critical equation through the use of the so called \emph{P-function}
\cite{CFP24}
(see also \cite{SeZo02,CFR20} for $p<N$ and $q=p^*-1$).

When $p=2$ we see that the equation solved by $v=\varphi(u)=-u^{-\frac{q-1}{2}}$ (which we want to show being concave) is
$$- \Delta v = - \frac{1}{v}\left(\frac{q+1}{q-1} |\nabla v|^2 + \frac{q-1}{2}\right);$$
thus even in the semilinear case we see that the nonlinearity does not satisfy the basic assumption on the monotonicity of Theorem \ref{thm_concav_inter_class}, i.e. it is not nonincreasing. 
Notice anyway that, being negative, it is harmonic concave by direct definition, and moreover $\lim_{t\to 0} \varphi'(t) = +\infty$ (see Corollary \ref{corol_concav_bounda}).

The strategy employed by \cite{Lin94} is the following: the author exploits the strong log-concavity of the first eigenfunction to show that for $q$ near $1$ the solutions are quasiconcave, while through a continuation argument it is shown to hold for each $q>1$ -- when the solution is a ground state -- see Remark \ref{rem_lin_argument}. 
We recall that a ground state is (up to a scaling) a minimizer of $\inf\left\{ \int_{\Omega} |\nabla u|^p \mid u \in W^{1,p}_0(\Omega), \; \int_{\Omega} |u|^q = 1\right\}$.
We refer also to \cite[Lemma 4.8]{LeVa08} where an evolutive argument has been used (see also \cite{Lio81, Ken88}).

We know by \cite[Lemma 3]{Lin94} (see also \cite[Proposition 4.3]{BrFr20}) that there exists a unique positive solution whenever $\Omega$ is bounded and convex, $p=2$ and $q>1$, $q$ close to $1$.
For a general $p>2$, in \cite[Corollary 1.4]{BrLi23} they recently showed the uniqueness of the positive ground state when $q>p-1$, $q$ close to $p-1$, and $\Omega$ is connected and $\partial \Omega \in C^{1,\alpha}$. 
In a similar setting, here we show concavity properties of a general solution of the equation, through a uniform convergence (see also \cite{GrPa08}).

\begin{theorem}
\label{thm_superhomog}
Let $\Omega \subset \R^N$ $N\geq 2$, be open, bounded, $\partial \Omega \in C^{1,\alpha}$, and let $p\in (1,+\infty)$. 
Consider
$$\begin{cases}
-\Delta_p u_q = u_q^{q} & \quad \hbox{ in $\Omega$},\\ 
u_q >0 & \quad \hbox{ in $\Omega$}, \\
u_q =0 & \quad \hbox{ on $\partial \Omega$},
\end{cases}$$
with $q \in (p-1, p^*-1)$. 
Then $v_q:=\frac{u_q}{\norm{u_q}_{\infty}}\to u$ in $L^{\infty}(\Omega)$ as $q \to p -1$, where $u$ is the first eigenfunction of the $p$-Laplacian in $\Omega$, normalized with $\norm{u}_{\infty}=1$.

Assume now in addition $\Omega$ convex. Then $\log(u)$ is concave and
$$\norm{\log(v_q) - \log(u)}_{\infty} \to 0 \quad \hbox{as $q \to p-1$}$$
together with
$$\log(v_q) \to \log(u) \quad \hbox{in $L^{\infty}(\Omega_{\delta})$}$$
for each $\delta>0$. 
As a consequence:
\begin{itemize}
\item if $N=2$, then for each $\eps>0$ we have $\log(u_q)$ is $\eps$-uniform concave in $\Omega_{\delta}$ for $q \in (p-1, q_0(\Omega,p,\eps, \delta))$.
\item if $p=2$, then $\log(u_q)$ is strongly concave in $\Omega_{\delta}$, for $q \in (p-1, q_0(\Omega,N,\delta))$; in particular, the level sets of $u_q$ in $\Omega_{\delta}$ 
are strictly convex and $u_q$ has a single (and nondegenerate) critical point in $\Omega_{\delta}$. 
\end{itemize}
\end{theorem}

\begin{proof}
By \cite[Corollary 4.2]{Tru67} (see also \cite[Corollary 3.7]{CaDe04}) we have $u_q \in C(\overline{\Omega})$. 
We employ a Gidas-Spruck blow-up procedure (see e.g. \cite{AzCl02}).
Set
$$M_q:= \norm{u_q}_{\infty}, \quad v_q:= \frac{u_q}{M_q}>0,$$
which solve
\begin{equation}\label{eq_q_large_rescaled}
\begin{cases}
-\Delta_p v_q = M_q^{q-p+1} v_q^{q} & \quad \hbox{ in $\Omega$},\\ 
v_q =0 & \quad \hbox{ on $\partial \Omega$}.
\end{cases}
\end{equation}
Notice that $\norm{v_q}_{\infty}=1$. We want to show that $M_q^{q-p+1} $ is equibounded. 
Assume by contradiction that $M_q^{q-p+1} \to +\infty$; let us set
$$u^*_q(x):= \frac{1}{M_q} u_q(\eps_q x +x_q), \quad x \in \Omega$$
where $\eps_q:=M_q^{-\frac{q-p+1}{p}} \to 0$, $x_q \in \Omega$ such that $u_q(x_q) = \norm{u_q}_{\infty}$.
We clearly have $u^*_q(0) = \norm{u^*_q}_{\infty}=1$. 
Moreover $u^*_q >0$ solve
$$\begin{cases}
-\Delta_p u^*_q = (u^*_q)^{q} & \quad \hbox{ in $\Omega_q$},\\ 
u^*_q =0 & \quad \hbox{ on $\partial \Omega_q$,}
\end{cases}$$
where $\Omega_q:= \frac{1}{\eps_q} (\Omega - x_q)$.
By \cite{Dib83, Tol84} and Ascoli-Arzelà theorem we have that $u^*_q$ converges in $C^{1,\alpha}_{loc}(\Lambda)$ to a function $u^* \geq 0$ which satisfies (weakly)
$$-\Delta_p u^* = (u^*)^{p-1} \quad \hbox{ in $\Lambda$}$$
where $\Lambda = \R^N_+$ or $\Lambda = \R^N$, depending on the fact that $x_q$ approaches or not $\partial \Omega$. 
Since $u^*(0)=1>0$, by the strong maximum principle \cite[Theorem 5]{Vaz84} we have that $u^*>0$.
But this is in contradiction with \cite[Lemma 2.1]{ILS08} (see also \cite[Theorem II]{SeZo02}). 
Thus $M_q^{q-p+1} $ is equibounded.

Let us come back to \eqref{eq_q_large_rescaled}. 
Arguing as before, one can show that, up to a subsequence, $v_q$ converge in $C^{1,\alpha}_{loc}(\Lambda)$ to some $u$, together with $M_q^{q-p+1} \to \lambda \in \R$, which thus satisfy
$$\begin{cases}
-\Delta_p u = \lambda u^{p-1} & \quad \hbox{ in $\Omega$},\\ 
u =0 & \quad \hbox{ on $\partial \Omega$}.
\end{cases}$$
Being $\norm{u}_{\infty}=1$, we have $u \nequiv 0$, thus $u >0$ by the strong maximum principle. Hence $u$ must coincide with the first eigenfunction, and $\lambda$ with the first eigenvalue. 
By a classical topological argument, we see that the whole family $v_q$ converges to $u$.

The $\eps$-uniform concavity comes from Theorem \ref{thm_intr_recap_plaplac} and Proposition \ref{prop_eps_strong_conc}. 
The case $p=2$ is already known, but let us give here some details: by the uniform estimates with respect to $q$ (see also \cite[Proposition 2.4 and Theorem 2.5]{BrLi23}) and regularity theory for the classical Laplacian \cite[Theorem 2.30]{FrRo22}, we know that $\log(u_q) \to \log(u)$ in $C^2(\Omega_{\delta})$. 
Moreover, by Theorem \ref{thm_intr_recap} we have the strong concavity of $\log(u)$, and thus the claim by Proposition \ref{prop_strong_conc}.
\end{proof}

\begin{remark}
\label{rem_lin_argument}
As already mentioned, similarly to the semilinear case, we expect the solutions to be $\frac{p-1-q}{p}$-concave for any $q \in (p-1, p^*-1)$, which is weaker than $\log$-concavity. Anyway, to conclude as in \cite{Lin94}, we would need the following three ingredients:
\begin{itemize}
\item[(i)] the first eigenfunction $u$ is strongly log-concave and $v_q \to u$ in $C^2_{loc}(\Omega)$;
\item[(ii)] the map $q \in (p-1, p^*-1)\mapsto u_q$ is well defined (that is, $u_q$ is unique) and continuous (in $C^2_{loc}$ topology, see also \cite{BrLi23});
\item[(iii)] an Hessian constant rank theorem in the spirit of \cite{KoLe87} holds, which allows to obtain strict concavity (strongly far from the boundary) from concavity.
\end{itemize}
We refer to \cite{GMsS24, Lin94} for details.
Due to (i)--(iii) this recipe seems not the right way to proceed for the quasilinear framework (see also comments in \cite{BMS23} regarding (iii)).
\end{remark}

\begin{remark}
We highlight that a uniform bound from below on $|\nabla u_q|$ (see \cite[Theorem 2.5]{BrLi23}) could be used, together with a strict quasiconcavity on $\Omega_{\delta}$ (if proved), to ensure that the critical point of $u_q$ is unique and nondegenerate in the whole $\Omega$, somehow extending Theorem \ref{thm_intr_recap_plaplac} to the superhomogeneous case.
We recall that in the semilinear case $p=2$ \cite{CaCh98} (see also \cite{DGM21}) the authors showed that every semistable solution has a unique critical point when $q \leq p-1$ and $a(x) \equiv 1$, but for $q>p-1$ every solution is unstable \cite[Theorem 2]{KaSi01}. 
When $q$ is sufficiently close to the critical exponent $p^*-1$, anyway, uniqueness of the critical point has been achieved by \cite{GlGr04}.
We refer also to \cite[Theorem 6.1]{Sci07} for results in the quasilinear setting in presence of symmetric domains.
\end{remark}

\subsection{Large $p$: towards strict quasiconcavity}

By \cite[Theorem 1]{Kaw90}, as $p\to +\infty$ we know that solutions $u_p$ of the torsion problem $-\Delta_p u_p = 1$ satisfy 
\begin{equation}\label{eq_conv_p_dist}
\lim_{p \to +\infty} u_p = d(\cdot,\partial \Omega) \quad \hbox{in $L^{\infty}(\Omega)$}
\end{equation}
which is a concave function; this is coherent with $\frac{p-1}{p} \to 1$. When $\Omega$ is strictly convex we have that the distance function is also strictly quasiconcave (see Proposition \ref{prop_boundar_approximat}). 
Moreover, when $\partial \Omega \in C^2$ and it is strongly convex, then \eqref{eq_conv_p_dist} can be improved \cite[Theorem 2]{Kaw90}.
We further mention that some partial concavity results when $p=+\infty$ are contained in \cite[Section 4]{Juu10}.

For any $p$ we already know that $u_p$ is quasiconcave: exploiting the ideas of the previous Sections, the strict quasiconcavity and the uniform convergence \eqref{eq_conv_p_dist}
we obtain the following result (see \cite[Definition 2.8]{GrRa22}).

\begin{proposition}
\label{prop_large_p}
Let $\Omega \subset \R^N$, $N\geq 1$, be open and strictly convex. Let $\eps>0$. 
Then there exists $p_0 =p_0(\Omega, N,\eps) \in (1,+\infty)$ such that, for any $p \in [p_0, +\infty)$, the positive solution of the torsion problem
$$\begin{cases}
-\Delta_p u_p = 1 & \quad \hbox{ in $\Omega$},\\ 
u_p =0 & \quad \hbox{ on $\partial \Omega$},
\end{cases}$$
is \emph{$\eps$-uniformly quasiconcave}, that is for some $\rho=\rho(\Omega, N,\eps, p)>0$ 
$$u_p \left(\tfrac{x+y}{2}\right)\geq \min \left\{u_p(x), u_p(y)\right\} + \rho \qquad \hbox{for each $x,y \in \Omega$, $|x-y|\geq \eps$}.$$
\end{proposition}

\begin{proof}
It is sufficient to recall that $u_p$ converge uniformly to $d(\cdot, \partial \Omega)$, which is strictly quasiconcave, and argue as in Proposition \ref{prop_eps_strong_conc} (notice that a \emph{quasiconcavity function} can be defined straightforwardly).
\end{proof}

We highlight again that, if one could show that the relation obtained in Proposition \ref{prop_large_p} holds for $\eps=0$, then uniqueness and nondegeneracy of the critical point would hold for large $p$ in any dimension $N$, extending \cite{BMS23} in the case of strictly convex domains.

\begin{remark}
The behaviour as $p\to 1$ of the solutions of $-\Delta_p u_p = 1$ is quite more nasty, essentially blowing up at $+\infty$ or flattering 
to $0$ depending on the Cheeger constant of $\Omega$, see \cite[Section 3 and 4]{Kaw90} and \cite{BEM16}; in particular, if $\Omega$ is a ball $B_R(0)$, then
$u_p^{\frac{p-1}{p}} \to \frac{R}{N}.$

Regarding the eigenfunction problem, by \cite[Remark 10]{KaFr03} the solutions $u_p$ of $-\Delta_p u_p = \lambda_p u_p$ as $p\to 1$ essentially converge in BV to the characteristic function of the Cheeger set (recall that, by \cite{AlCa09}, such set is convex, thus such characteristic function is quasiconcave); 
the convergence as $p \to +\infty$ has been instead investigated in \cite{JLM99}.
\end{remark}

\subsection{Fractional equations}
\label{sec_fractional}

Very few is known regarding concavity results for fractional equations: some partial results are contained in \cite{Gre14} (see also \cite{JKS19} and references therein).
Anyway, the only result, close to our framework and known to the authors, is contained in \cite{Kul17}: here, when $N=2$, $p=2$, $s=\frac{1}{2}$ and $f \equiv 1$, the author shows that the positive solution of the torsion problem
$$
\begin{cases}
\sqrt{-\Delta} u = 1 & \quad \hbox{ in $\Omega$},\\ 
u =0 & \quad \hbox{ on $\Omega^c$,}
\end{cases}
$$
on $\Omega \subset \R^2$, is concave in $\Omega$ (actually strictly concave, for some class of domains $\Omega$). 
The advantage of this case is that the solution is exactly concave, thus there is no need of a transformation $\varphi(u)$, for which is not known if $\varphi(u)$ solves a different fractional PDE (since Leibniz rule does not hold for $(-\Delta)^s$); moreover, when $s=\frac{1}{2}$ the equation satisfied by the $s$-harmonic extension is simpler. 
On the other hand, the lack of a transformation with singular derivative, together with the Neumann boundary condition of the extended equation, brings to much more difficulties in handling the boundary.

By looking at the cases $s=\frac{1}{2}$ and $s=1$, we conjecture as in \cite{Kul17} that for $(-\Delta)^s u = 1$ the solution is $\frac{1}{2s}$-concave (which is coherent with $1=(-\Delta)^s u \stackrel{s \to 0} \to u$, $u$ constant, i.e. $\infty$-concave), or more likely $\min\{\frac{1}{2s},1\}$-concave. 
More generally, one may expect that the concavity properties of $(-\Delta)^s u_s =g(u_s)$ are better than the ones of $-\Delta u = g(u)$: in particular, one may expect that if $u$ is $\alpha$-concave, than $u_s$ is $\alpha$-concave as well. 
What we do in this Section is some partial results in this direction, through perturbation arguments.
 
Indeed, consider the $p$-fractional Laplacian 
$$(-\Delta)^s_p u(x):= \int_{\R^N} \frac{|u(x)-u(y)|^{p-2} (u(x)-u(y))}{|x-y|^{N+sp}} dy \quad \hbox{for $x \in \R^N$}$$
where the integral is in the principal value sense.
We highlight that the choice of the constant in front of the integral (in this case, equal to $1$) influences the statements of the following results.
Through a perturbation argument, we show the following.
\begin{theorem}
\label{thm_fractional_log}
Let $\Omega \subset \R^N$, $N\geq 2$, be open, bounded, with $\partial \Omega \in C^{1,1}$, and let $p\in (1,+\infty)$. 
Let $u_s\in W^{s,p}_0(\Omega)$ be positive solutions of
$$\begin{cases}
(-\Delta)^s_p u_s = \lambda_s u_s^{p-1} & \quad \hbox{ in $\Omega$},\\ 
u_s =0 & \quad \hbox{ on $\Omega^c$},
\end{cases}$$
normalized at $\norm{u_s}_p=1$. 
Then 
$$\norm{u_{s} - u}_{\infty} \to 0 \quad \hbox{as $s \to 1$}$$
where $u$ is the positive solution of
$$\begin{cases}
-\Delta_p u = \lambda u^{p-1} & \quad \hbox{ in $\Omega$},\\ 
u =0 & \quad \hbox{ on $\partial \Omega$},
\end{cases}$$
normalized at $\norm{u}_p=1$.
Assume now $\Omega$ convex. Then $\log(u)$ is concave and, for every $\delta>0$, we have
$$\norm{\log(u_{s}) - \log(u)}_{L^{\infty}(\Omega_{\delta})} \to 0.$$
As a consequence:
\begin{itemize}
\item if $N=2$, then for each $\eps>0$ we have $\log(u_s)$ is $\eps$-uniformly concave in $\Omega_{\delta}$ for $s \in (s_0( \Omega, p,\eps,\delta), 1)$.
\item if $p=2$, then for each $\eps>0$ we have $\log(u_s)$ is $\eps$-strongly concave in $\Omega_{\delta}$ for $s \in (s_0(\Omega,N, \eps,\delta), 1)$.
 \end{itemize}
\end{theorem}

\begin{proof}
Since $\norm{u_s}_p \equiv 1$, by \cite[Theorem 1.2]{BPS16} we have that 
$$(1-s) \lambda_s \leq C$$ 
as $s \to 1$. 
By \cite[Theorem 2.10]{BPS16} (see also \cite[Corollary 4.2]{IMS14}) we know that $u_s \in L^{\infty}(\Omega)$, and more precisely
$$\norm{u_s}_{\infty} \leq \Big(C \frac{1}{(N-sp)^{p-1}} s (1-s) \lambda_s\Big)^{\frac{N}{s p^2}} \norm{u_s}_p \quad \hbox{if $p<N$ (thus $sp<N$)},$$
$$\norm{u_s}_{\infty} \leq \big(C \diam(\Omega)^{N(s-\frac{1}{2})} s (1-s) \lambda_s\big)^{\frac{2}{N}} \norm{u_s}_p \quad \hbox{if $p=N$ (and $s\geq \frac{3}{4}$)},$$
$$\norm{u_s}_{\infty} \leq \left(C \diam(\Omega)^{sp - N} (1-s) \lambda_s\right)^{\frac{1}{p}} \norm{u_s}_p \quad \hbox{if $p>N$ (and $s > \frac{N}{p}$)},$$
where $C=C(N,p)>0$. 
In particular, 
\begin{equation}\label{eq_uniform_linf_frac}
\norm{u_s}_{\infty} \leq C \quad \hbox{for $s \geq \frac{3}{4}$ (and $s>\frac{N}{p}$ if $p>N$)}.
\end{equation}
By \cite[(2.27) in proof of Theorem 2.10]{BPS16} we also know that 
$$[u_s]_{C^{0,s-\frac{N}{p}}(\overline{\Omega})} \leq \big(C (1-s) \lambda_s\big)^{\frac{1}{p}} \norm{u_s}_p\quad \hbox{if $p>N$ (and $s > \frac{N}{p}$)}$$
thus $\norm{u_s}_{C^{0,s_0-\frac{N}{p}}(\overline{\Omega})} \leq C$ for $s \geq s_0$ ($s_0>\frac{N}{p}$ however fixed).
We need to deal with $p \leq N$.

We know \cite[Theorem 1.1]{IMS14} that $u_s \in C^{0,\alpha_s}(\overline{\Omega})$ for some $\alpha_s= \alpha_s(p,N) \in (0,s)$, and more precisely, set $f_s:= \lambda_s u_s^{p-1}$, we have
$$\norm{u_s}_{C^{0,\alpha_s}(\overline{\Omega})} \leq C_s \norm{f_s}_{\infty}^{\frac{1}{p-1}}$$
for some $C_s=C_s(\Omega, p, N)$; from the proof therein we easily see that $\alpha_s$ and $C_s$ are equibounded for $s \to 1$.
On the other hand $\norm{f_s}_{\infty}$ are not equibounded in $s$, thus we give a closer inspection of the proof: we know indeed that
$$K_1 := 0 \leq (-\Delta)^s_p u_s \leq \lambda_s \norm{u_s}_{\infty}^{p-1} =: K_2^s.$$
The role of $K_1$ is played in \cite[Theorem 5.2 and Lemma 5.3]{IMS14} (and used in \cite[Theorem 5.4]{IMS14}), while the role of $K_2^s$ is played in the proof of \cite[Theorem 1.1]{IMS14} (implicitly also in \cite[Theorem 5.4 and Corollary 5.5]{IMS14}) only to bound $\norm{u_s}_{\infty}$: we can thus substitute this bound with the finer \eqref{eq_uniform_linf_frac}, which is uniform in $s$.
Hence, for some $\alpha \in (0,1)$ we have
$$\norm{u_s}_{C^{0,\alpha}(\overline{\Omega})} \leq C$$
for each $s$ large.
By \eqref{eq_diff_soluz_origin} we have
$$\norm{u_s - u}_{\infty} \leq \big( \norm{u_s}_{C^{0,\alpha_s}(\overline{\Omega})} + \norm{u}_{C^{0,\alpha}(\overline{\Omega})}\big)^{1-\theta_p} \norm{u_s - u}_p^{\theta_p}$$
where $\theta_p = \frac{\alpha}{\alpha + \frac{N}{p}}$,
thus
$$\norm{u_s - u}_{\infty} \leq C \norm{u_s - u}_p^{\theta_p}.$$
By \cite[Theorem 1.2]{BPS16} we have that there exists $s_k \to 1$ such that $u_{s_k} \to u$ in $L^p(\Omega)$, thus
$$\norm{u_{s_k} - u}_{\infty} \to 0.$$
Actually we can say better: by \cite[Theorem 5.1]{BoSa20}%
\footnote{Notice the difference in the notation in that paper, where $(-\Delta)^s_p$ is substituted with $(1-s) (-\Delta)^s_p$.} we know that, for each $s_k \to 1$, being $(1-s_k)\lambda_{s_k} \to \lambda$, there exists $u_{s_{k_n}} \to \bar{u}$ in $L^p(\Omega)$, where $\bar{u}=u$ by uniqueness of the problem and the $L^p$-constraint; by topological arguments, we have $u_s \to u$ as $s \to 1$ in $L^p(\Omega)$, which implies (by the previous argument) the convergence in $L^{\infty}(\Omega)$.
In particular, for each $\delta>0$, 
$$\norm{\log(u_{s}) - \log(u)}_{L^{\infty}(\Omega_{\delta})} \to 0 \quad \hbox{as $s \to 1$}.$$
We conclude by Theorem \ref{thm_intr_recap_plaplac}, Theorem \ref{thm_intr_recap} and Propositions \ref{prop_strong_conc} and \ref{prop_eps_strong_conc}.
\end{proof}

\begin{remark}
Arguing as in the end of the proof of Proposition \ref{thm_superhomog}, we observe that a $C^2$ convergence $u_s \to u$ (if proved) would imply that $u_s$ is actually strongly convex for $s \in (s_0(\delta), 1)$, and thus uniqueness and nondegeneracy of the critical point of $u_s$. We notice that a $C^{2,\beta}$ estimate, uniform in $s$, is given for $p=2$ in \cite[Lemma 4.4]{CaSi14} for equations in $\R^N$; see also \cite[Remark 4.1]{BPS16} for further comments.
\end{remark}

\begin{remark}
Differently from Remark \ref{rem_diff_log}, a global convergence of the type $\norm{\log(u_{s}) - \log(u)}_{L^{\infty}(\Omega)} \to 0$ is here not possible: indeed, considered $(-\Delta)^{s_1} u_1 = f(u_1)$ and $(-\Delta)^{s_2} u_2 = g(u_2)$ with suitable $f$ and $g$, by regularity results \cite[Theorem 4.4]{IMS14} and fractional Hopf boundary lemma \cite[Theorem 1.5(2)]{DpQ17}, we have $0<\frac{1}{C} \leq \frac{u_1}{d(\cdot, \Omega)^{s_1}}, \frac{u_2}{d(\cdot, \Omega)^{s_2}} \leq C$, thus $\frac{u_1}{u_2}, \frac{u_2}{u_1} \in L^{\infty}(\Omega)$ if $s_1=s_2$; for the same reason, this does not happen if $s_1 \neq s_2$.
\end{remark}

\begin{remark}
Similar arguments could be employed for $p$-subhomogeneous problems 
$$(-\Delta)^s_p u_s = u_s^q$$
with $q<p-1$, by exploiting some estimate uniform in $s$ and \cite[Theorem 4.5]{BoSa20}. 
Furthermore, in the case $q=0$, one could exploit \cite[Theorem 3.11]{War16} to deduce concavity properties for nonautonomous equations, i.e. $(-\Delta)^s_p = a(x)$, when $a$ is close to a constant, in the spirit of Section \ref{sec_diff_solutions}.
If one has a stability argument also for $s_k \to s^* \in (0,1)$, then one could use the result by \cite{Kul17} to say that, for $p=2$ and $s <\frac{1}{2}$, close to $\frac{1}{2}$, the solutions of the torsion problem $(-\Delta)^s u = 1$, $\Omega \subset \R^2$ are $\eps$-uniformly concave, at least for some class of smooth domains \cite[Definition 2.1]{Kul17}. 
\end{remark}

Finally, we point out that it could be interesting to study the \emph{fractional concavity} of solutions of fractional equations, where the notion has been introduced in \cite{DpQR22}.

\appendix

\section{Some facts on the $p$-Laplacian}
\label{sec_append_p_lapl}

In what follows, we recall some results in the literature folklore, but of which the authors were not able to find a proof. We will state the results in the case of a general $f=f(x,t)$, that is equation \eqref{eq_general_fxu}.

\subsection{Uniqueness}

We state a uniqueness result when $ \frac{f(x, t)}{t^{p-1}}$ is strictly decreasing (see also \cite[Theorem 2.1]{GuVe89}). 
When $f(x,t) \equiv g(t)$, we refer also to \cite[Proposition 3.8]{BMS22} which essentially says that the same uniqueness holds if $ \frac{g(t)}{t^{p-1}}$ is nonincreasing and decreasing only on a small region $[0,\delta]$; see also \cite{Mos23} for further comments. See also \cite{BeKa02, KaLi06} where the interesting tool of hidden convexity has been used to deal with general $\Omega$. 
Finally, we refer to Section \ref{sec_superhomog} for the superhomogeneous case.

\begin{lemma}[Brezis-Oswald uniqueness]
\label{lem_uniquen}
Let $\Omega \subset \R^N$, $N \geq 1$, be open, bounded, connected, satisfying the interior sphere condition, and let $p \in (1,+\infty)$. 
Assume that $f: \Omega \times (0,+\infty) \to \R$ satisfies
\begin{itemize}
\item $t \mapsto \frac{f(x, t)}{t^{p-1}}$ is nonincreasing, 
\item $f$ is bounded on bounded sets,
\end{itemize}
and let $u,v \in W^{1,p}_0(\Omega) \cap C^1(\overline{\Omega})$ be two solutions of
\eqref{eq_general_fxu}. 
Then
$$v = k u \quad \hbox{for some $k \in (0,+\infty)$}.$$
Assume moreover that
\begin{itemize}
\item $t \mapsto \frac{f(x, t)}{t^{p-1}}$ is strictly decreasing.
\end{itemize}
Then $v=u$. 
\end{lemma}

\begin{proof}
First, we notice that $\Delta_p u, \Delta_p v \in L^{\infty}(\Omega)$ and that $\frac{u}{v}, \frac{v}{u} \in L^{\infty}(\Omega)$: indeed, being $u \in C^1(\overline{\Omega})$, 
by Hopf boundary lemma \cite[Theorem 5]{Vaz84} 
we have $\partial_{\nu} u, \partial_{\nu} v > 0$ on $\partial \Omega$, and thus by de l'Hopital rule we have, for any $x \in \partial \Omega$, 
$$\limsup_{t \to 0} \frac{u(x+ t \nu)}{v(x + t \nu)} \leq \limsup_{t \to 0} \frac{\partial_t u(x+ t \nu)}{\partial_t v(x + t \nu)} = \limsup_{t \to 0} \frac{\nabla u(x+t\nu) \cdot \nu}{\nabla v(x+t\nu) \cdot \nu} = \frac{\partial_{\nu} u(x)}{\partial_{\nu} v(x)} < +\infty,$$
thus the claim.
Hence we can apply \cite[Lemma 2]{DiSa87} and obtain
$$\int_{\Omega} \left( \frac{f(x,u)}{u^{p-1}} - \frac{f(x,v)}{v^{p-1}}\right) (u^p - v^p) = \int_{\Omega} \left( \frac{-\Delta_p u}{u^{p-1}} - \frac{-\Delta_p v}{v^{p-1}}\right) (u^p - v^p) \geq 0.$$
Since $t \mapsto \frac{f(x, t)}{t^{p-1}}$ is nonincreasing, we have
\begin{equation}\label{eq_corol_lemma37}
\frac{f(x,u)}{u^{p-1}} = \frac{f(x,v)}{v^{p-1}};
\end{equation}
notice that, if $t \mapsto \frac{f(x, t)}{t^{p-1}}$ is strictly decreasing, we have the second claim.
By Picone's inequality \cite{AlHu98} we gain 
\begin{equation}\label{eq_corol_picone}
|\nabla v|^{p-2} \nabla v \cdot \nabla \left(\tfrac{u^p}{v^{p-1}}\right) \leq |\nabla u|^p
\end{equation}
where the equality is attained if and only if $u$ and $v$ are proportional. Integrating, we get 
$$\int_{\Omega} |\nabla v|^{p-2} \nabla v \cdot \nabla \left(\tfrac{u^p}{v^{p-1}}\right) \leq \int_{\Omega} |\nabla u|^p.$$
Notice that $\frac{u^p}{v^{p-1}} = (\frac{u}{v})^{p-1} u \in L^{\infty}(\Omega) \subset L^p(\Omega)$ and $\nabla \left(\frac{u^p}{v^{p-1}}\right) =
 p \left(\frac{u}{v}\right)^{p-1} \nabla u + (p-1)\left(\frac{u}{v}\right)^{p} \nabla v
 \in L^p(\Omega)$, 
thus $\frac{u^p}{v^{p-1}} \in W^{1,p}_0(\Omega)$ and by the equations
and the above relation we obtain
$$\int_{\Omega} f(x,v) \frac{u^p}{v^{p-1}} = \int_{\Omega} |\nabla v|^{p-2} \nabla v \cdot \nabla \left(\tfrac{u^p}{v^{p-1}}\right) \leq \int_{\Omega} |\nabla u|^p = \int_{\Omega} |\nabla u|^{p-2} \nabla u \cdot \nabla u = \int_{\Omega} f(x,u) u $$
that is
$$\int_{\Omega}\left(\frac{f(x,v)}{v^{p-1}} - \frac{f(x,u)}{u^{p-1}}\right)u^p \leq 0.$$
On the other hand, due to \eqref{eq_corol_lemma37}, the last is an equality, and so it must be \eqref{eq_corol_picone}, which implies the claim.
\end{proof}

\subsection{Comparison principle}
We show now a comparison principle for quasilinear equations in presence of a $p$-subhomogeneous function; notice that we are not requiring $f$ itself to be decreasing. 
The proof is inspired by \cite[Theorem 1.5]{MST24}; notice moreover that the result was obtained, when $f$ is a power, by \cite[Theorem 4.1]{BPZ22} in a more general setting through the use of hidden convexity.
See also \cite[Theorems 1.3 and 3.3]{DaSc04}, \cite[Lemma 2.1]{BMS23}. 

Finally, notice that the uniqueness result part of Lemma \ref{lem_uniquen} can be deduced from this comparison principle.

\begin{lemma}[Comparison principle]
\label{lem_comp_princip}
Let $\Omega \subset \R^N$, $N \geq 1$, be open, bounded, connected, satisfying the interior sphere condition, and let $p \in (1,+\infty)$. 
Assume that $f: \Omega \times (0,+\infty) \to \R$ satisfies
\begin{itemize}
\item $t \mapsto \frac{f(x, t)}{t^{p-1}}$ strictly decreasing, 
\item $f$ is bounded on bounded sets,
\end{itemize}
and let $u,v \in
 C^1(\overline{\Omega})$ be a subsolution and a supersolution,
namely
$$ -\Delta_p u - f(x,u) \leq -\Delta_p v - f(x,v) \quad \hbox{ in $\Omega$}$$
with 
$$u,v >0 \; \hbox{ in $\Omega$}, \quad u\leq v \; \hbox{ on $\partial \Omega$}.$$
Then
$$u \leq v \quad \hbox{in $\Omega$.}$$
\end{lemma}

\begin{proof}
Set $w:= (u^p-v^p)^+$, we need to prove that $w\equiv 0$. Consider $\Omega_+:= \textup{supp}(w)=\{u \geq v\}$.
As in the proof of Lemma \ref{lem_uniquen} we have $\frac{u}{v}, \frac{v}{u} \in L^{\infty}(\Omega)$,
thus we compute
$$\nabla \left(\tfrac{w}{u^{p-1}}\right) = \chi_{\Omega_+} \left( \nabla u - \nabla \left(\tfrac{v^p}{u^{p-1}}\right) + \delta (p-1) \nabla u\right);$$
this first tells us that $\frac{w}{u^{p-1}} \in W^{1,p}_0(\Omega)$. Moreover
$$|\nabla u|^{p-2} \nabla u \cdot \nabla \left(\tfrac{w}{u^{p-1}}\right) = \chi_{\Omega_+} \left( |\nabla u|^p - |\nabla u|^{p-2} \nabla u \cdot \nabla \left(\tfrac{v^p}{u^{p-1}}\right)\right);$$
similarly 
$$\nabla \left(\tfrac{w}{v^{p-1}}\right) = \chi_{\Omega_+} \left( -\nabla v + \nabla \left(\tfrac{u^p}{v^{p-1}}\right)\right),$$
thus $\frac{w}{v^{p-1}} \in W^{1,p}_0(\Omega)$ and
$$|\nabla v|^{p-2} \nabla v \cdot \nabla \left(\tfrac{w}{v^{p-1}}\right) = \chi_{\Omega_+} \left( - |\nabla v|^p + |\nabla v|^{p-2} \nabla v \cdot \nabla \left(\tfrac{u^p}{v^{p-1}}\right)\right).$$
Therefore
\begin{align*}
\MoveEqLeft |\nabla u|^{p-2} \nabla u \cdot \nabla \left(\tfrac{w}{u^{p-1}}\right) - |\nabla v|^{p-2} \nabla v \cdot \nabla \left(\tfrac{w}{v^{p-1}}\right) \\
& = \chi_{\Omega_+} \left( |\nabla v|^p - |\nabla u|^{p-2} \nabla u \cdot \nabla \left(\tfrac{v^p}{u^{p-1}}\right)\right)
+ \chi_{\Omega_+} \left( |\nabla u|^p - |\nabla v|^{p-2} \nabla v \cdot \nabla \left(\tfrac{u^p}{v^{p-1}}\right)\right) 
\geq 0
\end{align*}
by Picone's inequality \eqref{eq_corol_picone}. 
 In particular, being $\frac{w}{u^{p-1}}, \frac{w}{v^{p-1}}$ nonnegative test functions,
$$\int_{\Omega} \left( f(x,u) \frac{w}{u^{p-1}} - f(x,v) \frac{w}{v^{p-1}}\right) \geq \int_{\Omega_+} \left( |\nabla u|^{p-2} \nabla u \cdot \nabla \left(\tfrac{w}{u^{p-1}}\right) - |\nabla v|^{p-2} \nabla v \cdot \nabla \left(\tfrac{w}{v^{p-1}}\right) \right) \geq 0.$$
On the other hand, by the monotonicity $\frac{f(x,u)}{u^{p-1}} - \frac{f(x,v)}{v^{p-1}} \leq 0$ on $\Omega_+$, 
thus
$$ \left( \frac{f(x,u)}{u^{p-1}} - \frac{f(x,v)}{v^{p-1}}\right) w \equiv 0 \quad \hbox{on $\Omega$}.$$
Finally, by the strict monotonicity, we gain $w \equiv 0$, that is the claim.
\end{proof} 

As a consequence we obtain a uniform bound from below. 
\begin{corollary}
\label{corol_stima_basso}
Let $\Omega \subset \R^N$, $N \geq 1$, be open, bounded, connected, satisfying the interior sphere condition, and let $p \in (1,+\infty)$. 
Assume that $f_n, f_0: \Omega \times (0,+\infty) \to \R$ satisfy
\begin{itemize}
\item $t \mapsto \frac{f_0(x, t)}{t^{p-1}}$ strictly decreasing, 
\item $f_0$ is bounded on bounded sets,
\item $f_n(x,t) \geq f_0(x,t)$.
\end{itemize}
Let $u_n \in W^{1,p}_0(\Omega) \cap C^1(\overline{\Omega})$ be positive solutions of $-\Delta u_n = f_n(x,t)$, $-\Delta u_0 = f_0(x,t)$ with Dirichlet boundary conditions. 
Then, for each $\delta>0$, there exists $C=C(\Omega, N, \delta)>0$ such that
$$\inf_{\Omega_{\delta}} u_n \geq C >0.$$
\end{corollary}

\begin{remark}
We observe that, in the specific case $f(x,t) \equiv a(x)$, we can obtain the uniqueness result of Lemma \ref{lem_uniquen} as a consequence of Corollary \ref{corol_confron_2}, as well as the estimate in Corollary \ref{corol_stima_basso} in the specific case $f_n(x,t) \equiv a_n(x) \geq a_{0}$. 
\end{remark}

\addtocontents{toc}{\protect\setcounter{tocdepth}{2}}
\addtocontents{toc}{\SkipTocEntry} 
\subsection*{Acknowledgments}

The authors thank Lorenzo Brasco, Umberto Guarnotta, Luigi Montoro, Berardino Sciunzi and Sunra Mosconi for having made us aware of some references cited in Section \ref{sec_singular} and Appendix \ref{sec_append_p_lapl}.
The authors express their gratitude to the anonymous referee for their valuable comments.



\begin{thebibliography}{100}

\small

\bibitem{APP81} A. Acker, L.$\,$E. Payne, G. Philippin,
\emph{On the convexity of level lines of the fundamental mode in the clamped membrane problem, and the existence of convex solutions in a related free boundary problem},
Z. Angew. Math. Phys. \textbf{32} (1981), 
683--694. 

\bibitem{AlHu98} W. Allegretto, Y.$\,$X. Huang,
\emph{A Picone's identity for the p-Laplacian and applications},
Nonlinear Anal. \textbf{32} (1998), 
819--830.

\bibitem{AAGS23} N.$\,$M. Almousa, J. Assettini, M. Gallo, M. Squassina,
\emph{Concavity properties for quasilinear equations and optimality remarks},
Differential Integral Equations \textbf{37} (2024), 
1--26. 

\bibitem{ABCS23} N.$\,$M. Almousa, C. Bucur, R. Cornale, M. Squassina,
\emph{Concavity principles for nonautonomous elliptic equations and applications},
Asymptot. Anal. \textbf{135} (2023), 
509-524.

\bibitem{ALL97} O. Alvarez, J.-M. Lasry, P.-L. Lions,
\emph{Convex viscosity solutions and state constraints},
J. Math. Pures Appl. \textbf{17} (1997), 
265--288.

\bibitem{AlCa09} F. Alter, V. Caselles, 
\emph{Uniqueness of the Cheeger set of a convex body},
Nonlinear Anal. \textbf{70} (2009), 32--44.

\bibitem{AmGl14} A.$\,$L. Amadori, F. Gladiali,
\emph{Bifurcation and symmetry breaking for the Henon equation},
Adv. Differential Equations \textbf{19} (2014), 
755--782. 

\bibitem{Ant24} C.$\,$A. Antonini, 
\emph{Smooth approximation of Lipschitz domains, weak curvatures and isocapacitary estimates}, 
Calc. Var. Partial Differential Equations \textbf{63} (2024), 91(1--34). 

\bibitem{ADSZ10} M. Avriel, W.$\,$E. Diewert, S. Schaible, I. Zang,
``Generalized Concavity'',
SIAM, Philadelphia, 2010.

\bibitem{AvBr20} A. Avila, F. Brock,
\emph{A unified approach to symmetry for semilinear equations associated to the Laplacian in $\R^N$},
J. Math. Anal. Appl. \textbf{488} (2020), 
124087(1--28).

\bibitem{AzCl02} C. Azizieh, P. Clément,
\emph{A priori estimates and continuation methods for positive solutions of $p$-Laplace equations},
J. Differential Equations \textbf{179} (2002), 213--245.

\bibitem{BaZa17} J.$\,$M. Ball, A. Zarnescu,
\emph{Partial regularity and smooth topology-preserving approximations of rough domains},
Calc. Var. Partial Differential Equations \textbf{56} (2017), 
13(1--32).

\bibitem{Bas76} R.$\,$F. Basener,
\emph{Nonlinear Cauchy-Riemann equations and $q$--pseudoconvexity}, 
Duke Math. J. \textbf{43} (1976), 
203--213.

\bibitem{BeKa02} M. Belloni, B. Kawohl,
\emph{A direct uniqueness proof for equations involving the $p$-Laplace operator},
Manuscripta Math. \textbf{109} (2002), 
229--231.

\bibitem{Ber77} M.$\,$S. Berger,
``Nonlinearity and Functional Analysis'', 
Lectures on Nonlinear Problems in Mathematical Analysis, Academic Press Inc., New York, 1977.

\bibitem{BGP99} S. Berhanu, F. Gladiali. G. Porru,
\emph{Qualitative properties of solutions to elliptic singular problems},
J. Inequal. Appl. \textbf{3} (1999), 
313--330.

\bibitem{BiGu09} B. Bian, P. Guan, 
\emph{A microscopic convexity principle for nonlinear partial differential equations},
Invent. Math. \textbf{177} (2009), 
307--335.

\bibitem{BiSa13} M. Bianchini, P. Salani, 
\emph{Power concavity for solutions of nonlinear elliptic problems in convex domains},
in ``Geometric properties for parabolic and elliptic PDE’s'', eds. R. Magnanini, S. Sakaguchi, A. Alvino, Springer INdAM Series \textbf{2} (2013), 
35--48.

\bibitem{Blo97} Z. Błocki,
\emph{Smooth exhaustion functions in convex domains},
Proc. Amer. Math. Soc. \textbf{125} (1997), 
477--484. 

\bibitem{BoOr10} L. Boccardo, L. Orsina,
\emph{Semilinear elliptic equations with singular nonlinearities},
Calc. Var. Partial Differential Equations \textbf{37} (2010), 
363--380.

\bibitem{BoSa20} J.$\,$F. Bonder, A. Salort,
\emph{Stability of solutions for nonlocal problems},
Nonlinear Anal. \textbf{200} (2020), 
112080(1--13).

\bibitem{BMS22} W. Borrelli, S. Mosconi, M. Squassina,
\emph{Concavity properties for solutions to $p$-Laplace equations with concave nonlinearities},
Adv. Calc. Var. \textbf{17} (2022), 
79--97.

\bibitem{BMS23} W. Borrelli, S. Mosconi, M. Squassina,
\emph{Uniqueness of the critical point for solutions of some $p$-Laplace equations in the plane},
Rend. Lincei Mat. Appl. \textbf{34} (2023), 
61--88.

\bibitem{BrLi76} H. J. Brascamp, E. H. Lieb,
\emph{On extensions of the Brunn-Minkowski and Prékopa-Leindler theorems, including inequalities for log concave functions, and with an application to the diffusion equation},
J. Funct. Anal. \textbf{22} (1976), 
366--389.

\bibitem{BrFr20} L. Brasco, G. Franzina,
\emph{An overview on constrained critical points of Dirichlet integrals},
Rend. Semin. Mat. Univ. Politec. Torino \textbf{78} (2020), 
7--50. 

\bibitem{BrLi23} L. Brasco, E. Lindgren,
\emph{Uniqueness of extremals for some sharp Poincaré-Sobolev constants},
Trans. Amer. Math. Soc. \textbf{376} (2023), 
3541--3584.

\bibitem{BPS16} L. Brasco, E. Parini, M. Squassina,
\emph{Stability of variational eigenvalues for the fractional $p$–Laplacian},
Discrete Contin. Dyn. Syst. \textbf{16} (2016), 
1813--1845. 

\bibitem{BPZ22} L. Brasco, F. Prinari, A.$\,$C. Zagati,
\emph{A comparison principle for the Lane–Emden equation and applications to geometric estimates},
Nonlinear Anal. \textbf{220} (2022), 
112847(1--41).

\bibitem{BuSq20} C. Bucur, M. Squassina,
\emph{Approximate convexity principles and applications to PDEs in convex domains},
Nonlinear Anal. \textbf{192} (2020), 
111661(1--21).

\bibitem{BEM16} H. Bueno, G. Ercole, S. Macedo,
\emph{Asymptotic behavior of the $p$-torsion functions as $p$ goes to $1$},
Arch. Math. \textbf{107} (2016), 
63--72.

\bibitem{BPS22} S.-S. Byun, D.$\,$K. Palagachev, P. Shin,
\emph{Global Hölder continuity of solutions to quasilinear equations with Morrey data},
Commun. Contemp. Math. \textbf{24} (2022), 
 2150062(1--41). 
 
\bibitem{CaCh98} X. Cabré, S. Chanillo,
\emph{Stable solutions of semilinear elliptic problems in convex domains},
Selecta Math. (N.S.) \textbf{4} (1998), 
1--10.

\bibitem{CaSi14} X. Cabré, Y. Sire, 
\emph{Nonlinear equations for fractional Laplacians, I: regularity, maximum principles, and Hamiltonian estimates},
Ann. Inst. H. Poincaré Anal. Non Linéaire \textbf{31} (2014), no. 1, 23--53.

\bibitem{CaFr85} L.$\,$A. Caffarelli, A. Friedmann,
\emph{Convexity of solutions of semilinear elliptic equations},
Duke Math. J. \textbf{52} (1985), 
431--456.

\bibitem{CaSp82} L.$\,$A. Caffarelli, J. Spruck,
\emph{Convexity properties of solutions to some classical variational problems},
Comm. Partial Differential Equations \textbf{7} (1982), 
1337--1379.

\bibitem{CaDe04} A. Canino, M. Degiovanni,
\emph{A variational approach to a class of singular semilinear elliptic equations},
J. Convex Anal. \textbf{11} (2004), 
147--162.

\bibitem{CES19} A. Canino, F. Esposito, B. Sciunzi,
\emph{On the Höpf boundary lemma for singular semilinear elliptic equations},
J. Differential Equations \textbf{266} (2019), 
5488--5499.

\bibitem{CST16} A. Canino, B. Sciunzi, A. Trombetta,
\emph{Existence and uniqueness for $p$-Laplace equations involving singular nonlinearities},
Nonlinear Differ. Equ. Appl. \textbf{26} (2016), 
8(1--18).

\bibitem{ChWe23} A. Chau, B. Weinkove,
\emph{Concavity of solutions to semilinear equations in dimension two},
Bull. Lond. Math. Soc. \textbf{55} (2022), 
706--716.

\bibitem{CCPP24} G. Ciraolo, M. Cozzi, M. Perugini, L. Pollastro,
\emph{A quantitative version of the Gidas-Ni-Nirenberg Theorem},
J. Funct. Anal. \textbf{287} (2024), 
110585(1--29).

\bibitem{CFP24} G. Ciraolo, A. Farina, C.$\,$C. Polvara,
\emph{Classification results, rigidity theorems and semilinear PDEs on Riemannian manifolds: a P-function approach},
\href{https://doi.org/10.48550/arXiv.2406.13699}{arXiv:2406.13699} (2024), pp. 24.

\bibitem{CFR20} G. Ciraolo, A. Figalli, A. Roncoroni,
\emph{Symmetry results for critical anisotropic $p$-Laplacian equations in convex cones},
Geom. Funct. Anal. \textbf{30} (2020), 770--803.

\bibitem{CrFr20} G. Crasta, I. Fragalà,
\emph{The Brunn–Minkowski inequality for the principal eigenvalue of fully nonlinear homogeneous elliptic operators},
Adv. Math. \textbf{359} (2020), 
106855(1--24).

\bibitem{DaSc04} L. Damascelli, B. Sciunzi,
\emph{Regularity, monotonicity and symmetry of positive solutions of $m$-Laplace equations},
J. Differential Equations \textbf{206} (2004), 
483--515.

\bibitem{DGM21} F. De Regibus, M. Grossi, D. Mukherjee,
\emph{Uniqueness of the critical point for semi-stable solutions in $\R^2$},
Calc. Var. Partial Differential Equations \textbf{60} (2021), 
25(1--13).

\bibitem{DpQ17} L.$\,$M. Del Pezzo, A. Quaas,
\emph{A Hopf’s lemma and a strong minimum principle for the fractional $p$-Laplacian},
J. Differential Equations \textbf{263} (2017), 
765--778.

\bibitem{DpQR22} L.$\,$M. Del Pezzo, A. Quaas, J.$\,$D. Rossi,
\emph{Fractional convexity},
Math. Ann. \textbf{381} (2022), 
1687--1719.

\bibitem{DeZo94} M.$\,$C. Delfour, J.-P. Zolésio,
\emph{Shape analysis via oriented distance functions},
J. Funct. Anal. \textbf{123} (1994), 
129--201.

\bibitem{DiSa87} J.$\,$I. Díaz, J.$\,$E. Saa,
\emph{Existence et unicité de solutions positives pour certaines équations elliptiques quasilinéaires},
C. R. Acad. Sci. Paris 
(1987), 
521--524.

\bibitem{Dib83} E. DiBenedetto,
\emph{$C^{1+\alpha}$ local regularity of weak solutions of degenerate elliptic equations},
Nonlinear Anal. \textbf{7} (1983), 
827--850.

\bibitem{Dok76} P. Doktor,
\emph{Approximation of domains with Lipschitzian boundary},
Cas. Pestovani Mat. \textbf{101} (1976), 
237--255.

\bibitem{FrRo22} X. Fernández-Real, X. Ros-Oton,
``Regularity Theory for Elliptic PDE'',
Zurich Lectures in Advanced Mathematics, EMS Press, 2022.

\bibitem{FeVe22} L. Ferreri, G. Verzini, 
\emph{Asymptotic properties of an optimal principal eigenvalue with
spherical weight and Dirichlet boundary conditions},
Nonlinear Anal. \textbf{224} (2022), 
113103(1--25).

\bibitem{Foo84} R.$\,$L. Foote,
\emph{Regularity of the distance function},
Proc. Amer. Math. Soc. \textbf{92} (1984), 
153--155.

\bibitem{GMrS25} M. Gallo, R. Moraschi, M. Squassina,
\emph{Quantitative and exact concavity principles for parabolic and elliptic equations},
\href{https://doi.org/10.48550/arXiv.2504.09494}{arXiv:2504.09494} (2025), pp. 38.

\bibitem{GMsS24} M. Gallo, S. Mosconi, M. Squassina,
\emph{Power law convergence and concavity for the Logarithmic Schrödinger equation}, 
\href{https://doi.org/10.48550/arXiv.2411.0161}{arXiv:2411.01614} (2024), pp. 53.

\bibitem{GiGi82} M. Giaquinta, E. Giusti,
\emph{On the regularity of the minima of variational integrals},
Acta Math. \textbf{148} (1982), 
31--46. 

\bibitem{GNN79} B. Gidas, W.-M. Ni, L. Nirenberg, 
\emph{Symmetry and related properties via the maximum principle},
Commun. Math. Phys. \textbf{68} (1979), 
209--243.

\bibitem{GiTr01} D. Gilbarg, N.$\,$S. Trudinger,
``Elliptic Partial Differential Equations of Second Order'',
Classics in Mathematics, Springer-Verlag, Berlin Heidelberg, 2001.

\bibitem{Gla02} F. Gladiali,
\emph{Proprietà qualitative e non esistenza di soluzioni positive per alcuni problemi ellittici non lineari},
Ph.D. Thesis, Sapienza Università di Roma (2002).

\bibitem{GlGr04} F. Gladiali, M. Grossi,
\emph{Strict convexity of level sets of solutions of some nonlinear elliptic equations},
Proc. Roy. Soc. Edinburgh Sect. A \textbf{134} (2004), 
363--373.

\bibitem{GlGr22} F. Gladiali, M. Grossi,
\emph{On the number of critical points of solutions of semilinear equations in $\R^2$},
Amer. J. Math. \textbf{144} (2022), 
 1221--1240.

\bibitem{Gom07} J.$\,$M. Gomes,
\emph{Sufficient conditions for the convexity of the level sets of ground-state solutions},
Arch. Math. \textbf{88} (2007), 
269--278.

\bibitem{Gre14} A. Greco,
\emph{Fractional convexity maximum principle},
\href{https://web.ma.utexas.edu/mp_arc/c/14/14-73.pdf?}{utexas.edu} (2014), pp. 12.

\bibitem{GrPo93} A. Greco, G. Porru, 1993,
\emph{Convexity of solutions to some elliptic partial differential equations},
SIAM J. Math. Anal. \textbf{24} (1993), 
833--839.

\bibitem{GrRa22} G. Grelier, M. Raja,
\emph{On uniformly convex functions},
J. Math. Anal. Appl. \textbf{505} (2022), 
125442(1--25).

\bibitem{Gri11} P. Grisvard,
``Elliptic Problems in Nonsmooth Domains'',
Classics in Applied Mathematics, SIAM, USA, 2011.

\bibitem{GrPa08} C. Grumiau, E. Parini,
\emph{On the asymptotics of solutions of the Lane-Emden problem for the $p$-Laplacian},
Arch. Math. \textbf{91} (2008), 354--365.
 
\bibitem{Gua21} U. Guarnotta,
\emph{Existence results for singular convective elliptic problems},
Ph.D. Thesis, Università degli Studi di Palermo (2021).

\bibitem{GuVe89} M. Guedda, L. Veron,
\emph{Quasilinear elliptic equations involving critical Sobolev exponents},
Nonlinear Anal. \textbf{13} (1989), 
879--902.

\bibitem{HNS16} F. Hamel, N. Nadirashvili, Y. Sire,
\emph{Convexity of level sets for elliptic problems in convex domains or convex rings: two counterexamples},
Amer. J. Math. \textbf{138} (2016), 
499--527.

\bibitem{HNST18} A. Henrot, C. Nitsch, P. Salani, C. Trombetti,
\emph{Optimal concavity of the torsion function},
J. Optim. Theory Appl. \textbf{178} (2018), 
26--35.

\bibitem{HePi18} A. Henrot, M. Pierre,
``Shape Variation and Optimization'',
Tracts in Mathematics \textbf{28}, EMS, Germany, 2018.

\bibitem{Hor94} L. Hörmander,
``Notions of Convexity'',
Modern Birkhäuser Classics, Birkhäuser, Boston, 2007.

\bibitem{HyUl52} D.$\,$H. Hyers, S.$\,$M. Ulam,
\emph{Approximately convex functions},
Proc. Amer. Math. Soc. \textbf{3} (1952), 
821--828.

\bibitem{IMS14} A. Iannizzotto, S. Mosconi, M. Squassina,
\emph{Global Hölder regularity for the fractional p-Laplacian},
Rev. Mat. Iberoam. \textbf{32} (2016), 
1353--1392.

\bibitem{IST20} K. Ishige, P. Salani, A. Takatsu,
\emph{To logconcavity and beyond},
Commun. Contemp. Math. \textbf{22} (2020), 
1950009(1--17). 

\bibitem{ILS08} L. Iturriaga, S. Lorca, J. Sánchez, 
\emph{Existence and multiplicity results for the $p$-Laplacian with a $p$-gradient term},
NoDEA Nonlinear Differential Equations Appl. \textbf{15} (2008), 
729--743.

\bibitem{JKS19} S. Jarohs, T. Kulczucki, P. Salani,
\emph{Starshapedness of the superlevel sets of solutions to equations involving the fractional Laplacian in starshaped rings},
Math. Nachr. \textbf{292} (2019), 
1008--1021.

\bibitem{Juu10} P. Juutinen,
\emph{Concavity maximum principle for viscosity solutions of singular equations},
NoDEA Nonlinear Differential Equations Appl. \textbf{17} (2010), 
601--618.

\bibitem{JLM99} P. Juutinen, P. Lindqvist, J.$\,$J. Manfredi,
\emph{The $\infty$-eigenvalue problem},
Arch. Ration. Mech. Anal. \textbf{148} (1999), 
89--105.

\bibitem{Kad64} J. Kadlec,
\emph{On the regularity of the solution of the Poisson problem on a domain with boundary locally similar to the boundary of a convex open set},
Czechoslovak Math. J. \textbf{14} (1964), 
386--393 (in Russian).

\bibitem{Kaw85R} B. Kawohl, 
``Rearrangements and convexity of level sets in PDE'',
Lecture Notes in Math. \textbf{1150}, Springer-Verlag Heidelberg, 1985.

\bibitem{Kaw85W} B. Kawohl, 
\emph{When are solutions to nonlinear elliptic boundary value problems convex?},
Comm. Partial Differential Equations \textbf{10} (1985), 
1213--1225.

\bibitem{Kaw90} B. Kawohl,
\emph{On a family of torsional creep problems}
J. Reine Angew. Math. \textbf{410} (1990), 
1--22.

\bibitem{KaFr03} B. Kawohl, V. Fridman,
\emph{Isoperimetric estimates for the first eigenvalue of the $p$-Laplace operator and the Cheeger constant},
Comment. Math. Univ. Carolin. \textbf{44} (2003), 
659--667.

\bibitem{KaLi06} B. Kawohl, P. Lindqvist,
\emph{Positive eigenfunctions for the p-Laplace operator revisited},
Analysis (Berlin) \textbf{26} (2006), 
545--550.

\bibitem{KaSi01} J. Karátson, P.$\,$L. Simon,
\emph{On the stability properties of nonnegative solutions of semilinear problems with convex or concave nonlinearity},
J. Comput. Appl. Math. \textbf{131} (2001), 
497--501.

\bibitem{KeMn93} G. Keady, A. McNabb, 
\emph{The elastic torsion problem: solutions in convex domains},
New Zealand J. Math. \textbf{22} (1993), 
43--64.

\bibitem{Ken84} A.$\,$U. Kennington,
\emph{An improved convexity maximum principle and some applications},
Ph.D. Thesis at University of Adelaide (1984). 
(Bull. Aust. Math. Soc. \textbf{31} (1985), 
159-160).
 
\bibitem{Ken85} A.$\,$U. Kennington,
\emph{Power concavity and boundary value problems},
Indiana Univ. Math. J. \textbf{34} (1985), 
687--704.
 
\bibitem{Ken88} A.$\,$U. Kennington,
\emph{Convexity of level curves for an initial value problem},
J. Math. Anal. Appl. \textbf{133} (1988), 
324--330.

\bibitem{KST24} G. Khan, S. Saha, M. Tuerkoen,
\emph{Concavity properties of solutions of elliptic equations under conformal deformations},
\href{https://doi.org/10.48550/arXiv.2403.03200}{arXiv:2403.03200} (2024), pp. 18.

\bibitem{Kor83Co} N.$\,$J. Korevaar, 
\emph{Convex solutions to nonlinear elliptic and parabolic boundary value problems},
Indiana Univ. Math. J. \textbf{32} (1983), 
603--614.

\bibitem{KoLe87} N.$\,$J. Korevaar, J.$\,$L. Lewis,
\emph{Convex solutions of certain elliptic equations have constant rank Hessians},
Arch. Ration. Mech. Anal. \textbf{97} (1987), 
19--32.

\bibitem{KrPa81} S.$\,$G. Krantz, H.$\,$R. Parks,
\emph{Distance to $C^k$ hypersurfaces},
J. Differential Equations \textbf{40} (1981), 
116--120.

\bibitem{KrPa99} S.$\,$G. Krantz, H.$\,$R. Parks,
``The Geometry of Domains in Space'',
Birkhäuser Advanced Texts, Springer-Science+Business Media, New York, 1999.

\bibitem{Kul17} T. Kulczycki,
\emph{On concavity of solutions of the Dirichlet problem for the equation $(-\Delta)^{1/2} \varphi =1$ in convex planar regions},
J. Eur. Math. Soc. \textbf{19} (2017), 
1361--1420.

\bibitem{LaUr66} O.$\,$A.\ Ladyzhenskaya, N.\ N.\ Ural'tseva, 
``Linear and Quasilinear Elliptic Equations'', 
Academic Press Inc. \textbf{46}, New York, 1968.

\bibitem{Le007} A. Lê,
\emph{On the local Hölder continuity of the inverse of the $p$-Laplace operator},
Proc. Amer. Math. Soc. \textbf{135} (2007), 
3553--3560.

\bibitem{LeVa08} K. Lee, J.$\,$L. Vázquez,
\emph{Parabolic approach to nonlinear elliptic eigenvalue problems},
Adv. Math. \textbf{219} (2008), 
2006--2028.

\bibitem{LePe20} M. Lewicka, Y. Peres,
\emph{Which domains have two-sided supporting unit spheres at every boundary point?},
Expo. Math. \textbf{38} (2020), 
548--558.

\bibitem{Lie85} G.$\,$M. Lieberman,
\emph{Regularized distance and its applications},
Pacific J. Math. \textbf{117} (1985), 
329--352.

\bibitem{Lie88} G.$\,$M. Lieberman,
\emph{Boundary regularity for solutions of degenerate elliptic equations},
Nonlinear Anal. \textbf{12} (1988), 
1203--1219.

\bibitem{Lin94} C.-S. Lin,
\emph{Uniqueness of least energy solutions to a semilinear elliptic equation in $\R^2$},
Manuscripta Math. \textbf{84} (1994), 
13--19.
 
\bibitem{Lin81} P.$\,$O. Lindberg,
\emph{Power convex functions},
in ``Generalized Concavity in Optimization and Economics'', eds. S. Schaible, W.$\,$T. Ziemba, Academic Press Inc., London, 1981.

\bibitem{Lind94} P. Lindqvist,
\emph{A note on the nonlinear Rayleigh quotient},
Potential Anal. \textbf{2} (1993), 
199--218.

\bibitem{Lin17} P. Lindqvist,
``Notes on the $p$-Laplace equation'',
Report of University of Jyväskylä, Department of Mathematics and Statistics, 2017.

\bibitem{Lio81} P.-L. Lions, 
\emph{Two geometrical properties of solutions of semilinear problems},
Appl. Anal. \textbf{12} (1981), 
267--272.

\bibitem{MSY12} X.-N. Ma, S. Shi, Y. Ye,
\emph{The convexity estimates for the solutions of two elliptic equations},
Comm. Partial Differential Equations \textbf{37} (2012), 
2116--2137.

\bibitem{MaLi71} L.$\,$G.\ Makar-Limanov,
\emph{Solution of Dirichlet’s problem for the equation $\Delta u = -1$ in a convex region},
Mat. Zametki \textbf{9} (1971), 
89--92.
 
\bibitem{Maz11} V. Maz’ya,
``Sobolev Spaces, with Applications to Elliptic Partial Differential Equations'',
Comprehensive Studies in Mathematics \textbf{342}, Springer-Verlag Berlin Heidelberg, 2011.

\bibitem{MMPS14} S. Merchána, L. Montoro, I. Peral, B. Sciunzi, 
\emph{Existence and qualitative properties of solutions to a quasilinear elliptic equation involving the Hardy-Leray potential},
Ann. Inst. H. Poincaré Anal. Non Linéaire \textbf{31} (2014), 
1--22.

\bibitem{MST24} L. Montoro, B. Sciunzi, A. Trombetta,
\emph{A comparison principle for a doubly singular quasilinear anisotropic problem},
Commun. Contemp. Math. 
(2024), 
2350060(1--18). 

\bibitem{Mos23} S. Mosconi,
\emph{A non-smooth Brezis-Oswald uniqueness result},
Open Math. \textbf{21} (2023), 
20220594(1--28).

\bibitem{MRS24} S. Mosconi, G. Riey, M. Squassina,
\emph{Concave solutions to Finsler $p$-Laplace type equations},
Discrete Contin. Dyn. Syst. Ser. A \textbf{44} (2024), 
3669--3697.

\bibitem{Mor24} R. Moraschi, 
\emph{Concavity properties for solutions of nonlinear elliptic and parabolic equations},
Master Thesis, Università Cattolica del Sacro Cuore (2024).

\bibitem{NgNi93} T. Ng, K. Nikodem,
\emph{On approximately convex functions},
Proc. Amer. Math. Soc. \textbf{118} (1993), 
103--108.

\bibitem{Roc97} R.$\,$.T. Rockafellar,
``Convex Analysis'',
Princeton Landmarks in Mathematics and Physics \textbf{30}, Princeton University Press, 1970.

\bibitem{Sak87} S. Sakaguchi,
\emph{Concavity properties of solutions to some degenerate quasilinear elliptic Dirichlet problems},
Ann. Sc. Norm. Super. Pisa Cl. Sci. 4 \textbf{14} (1987), 
403--421.

\bibitem{Sci07} B. Sciunzi,
\emph{Some results on the qualitative properties of positive solutions of quasilinear elliptic equations},
NoDEA Nonlinear Differential Equations Appl. \textbf{14} (2007), 
315--334.

\bibitem{SeZo02} J. Serrin, H. Zou,
\emph{Cauchy-Liouville and universal boundedness theorems for quasilinear elliptic equations and inequalities},
Acta. Math. \textbf{189} (2002), 79--142.

\bibitem{SWYY85} I.$\,$M. Singer, B. Wong, S.-T. Yau, S.$\,$S.-T. Yau,
\emph{An estimate of the gap of the first two eigenvalues in the Schrödinger operator},
Ann. Sc. Norm. Super. Pisa Cl. Sci. 4 \textbf{12} (1985), 
319--333.

\bibitem{Ste22} S. Steinerberger,
\emph{On concavity of solutions of the nonlinear Poisson equation},
Arch. Ration. Mech. Anal. \textbf{244} (2022), 
209--224.

\bibitem{Tol84} P. Tolksdorf,
\emph{Regularity for a more general class of quasilinear elliptic equations},
J. Differential Equations \textbf{51} (1984), 
126--150.

\bibitem{Tru67} N.$\,$S. Trudinger,
\emph{On Harnack type inequalities and their application to quasilinear elliptic equations},
Comm. Pure Appl. Math. \textbf{20} (1967), 
721--747.

\bibitem{Vaz84} J.$\,$L. Vázquez,
\emph{A strong maximum principle for some quasilinear elliptic equations},
Appl. Math. Optim. \textbf{12} (1984), 
191--202.

\bibitem{War16} M. Warma,
\emph{Local Lipschitz continuity of the inverse of the fractional $p$-Laplacian, Hölder type continuity and continuous dependence of solutions to associated parabolic equations on bounded domains},
Nonlinear Anal. \textbf{135} (2015), 
129--157.

\bibitem{Wet11} T. Weth,
\emph{On the lack of directional quasiconcavity of the fundamental mode in the clamped membrane problem},
Arch. Math. \textbf{97} (2011), 
365--372.

\end{thebibliography}
\end{document}